\documentclass[letterpaper]{amsart}

\usepackage[letterpaper]{geometry}
\usepackage{amsmath, amsthm, amsfonts, amsbsy, thmtools, amssymb}

\usepackage{thm-restate}
\usepackage{hyperref} 
\usepackage[]{ytableau}
\usepackage[mathscr]{euscript}	
\usepackage{stmaryrd}
\usepackage[all, arc, curve, frame]{xy} 
\usepackage[T1]{fontenc}
\usepackage{tikz}
\usetikzlibrary{arrows}

\numberwithin{equation}{section}

\SetSymbolFont{stmry}{bold}{U}{stmry}{m}{n}

\DeclareMathOperator{\Sign}{Sign}

\DeclareMathOperator{\Id}{Id}

\DeclareMathOperator{\Ai}{Ai}

\DeclareMathOperator{\BBP}{BBP}
\DeclareMathOperator{\TW}{TW}

\DeclareMathOperator{\Sym}{Sym}

\title{Phase Transitions in the ASEP and Stochastic Six-Vertex Model} 
\author{Amol Aggarwal and Alexei Borodin}

\begin{document}

\newtheorem{thm}{Theorem}[section]
\newtheorem{prop}[thm]{Proposition}
\newtheorem{lem}[thm]{Lemma}
\newtheorem{cor}[thm]{Corollary}
\newtheorem{conj}[thm]{Conjecture}
\newtheorem{que}[thm]{Question}
\theoremstyle{remark}
\newtheorem{rem}[thm]{Remark}
\theoremstyle{definition}
\newtheorem{definition}[thm]{Definition}
\newtheorem{exa}[thm]{Example}

\begin{abstract}

In this paper we consider two models in the Kardar-Parisi-Zhang (KPZ) universality class, the asymmetric simple exclusion process (ASEP) and the stochastic six-vertex model. We introduce a new class of initial data (which we call {\itshape generalized step Bernoulli initial data}) for both of these models that generalizes the step Bernoulli initial data studied in a number of recent works on the ASEP. Under this class of initial data, we analyze the current fluctuations of both the ASEP and stochastic six-vertex model and establish the existence of a phase transition along a characteristic line, across which the fluctuation exponent changes from $1 / 2$ to $1 / 3$. On the characteristic line, the current fluctuations converge to the general (rank $k$) Baik-Ben-Arous-P\'{e}ch\'{e} distribution for the law of the largest eigenvalue of a critically spiked covariance matrix. For $k = 1$, this was established for the ASEP by Tracy and Widom; for $k > 1$ (and also $k = 1$, for the stochastic six-vertex model), the appearance of these distributions in both models is new. 
\end{abstract}

\maketitle

\tableofcontents

\section{Introduction}

\label{Introduction}

In this paper we study two interacting particle systems in the Kardar-Parisi-Zhang (KPZ) universality class, the asymmetric simple exclusion process (ASEP) and the stochastic six-vertex model. We begin in Section \ref{ProcessModel} by defining these two models and their associated observables. In Section \ref{EquationBoundary}, we provide some context for our results, which will be more carefully stated in Section \ref{PhaseTransitionAsymptotics}.

\subsection{The ASEP and Stochastic Six-Vertex Model}

\label{ProcessModel}

\subsubsection{The Asymmetric Simple Exclusion Process}

\label{AsymmetricExclusions}

Introduced to the mathematics community by Spitzer \cite{IMP} in 1970 (and also appearing two years earlier in the biology work of Macdonald, Gibbs, and Pipkin \cite{KBNAT}), the \emph{asymmetric simple exclusion process} (ASEP) is a continuous time Markov process that can be described as follows. Particles are initially (at time $0$) placed on $\mathbb{Z}$ in such a way that at most one particle occupies any site. Associated with each particle are two exponential clocks, one of rate $L$ and one of rate $R$; we assume that $R > L \ge 0$ and that all clocks are mutually independent. When some particle's left clock rings, the particle attempts to jump one space to the left; similarly, when its right clock rings, it attempts to jump one space to the right. If the destination of the jump is unoccupied, the jump is performed; otherwise it is not. This is sometimes referred to as the \emph{exclusion restriction}. 

Associated with the ASEP is an observable called the \emph{current}. For the purpose of this paper, the \emph{current} (denoted $J_t (x)$) is the number of particles strictly to the right of $x$ at time $t$, for $x \in \mathbb{R}$ and $t \in \mathbb{R}_{\ge 0}$. In all instances we encounter in this paper, $J_t (x)$ will be almost surely finite since we only consider asymmetric simple exclusion processes in which no sites in $\mathbb{Z}_{> 0}$ are initially occupied. 

One of the purposes of this paper will be to understand the large-time current fluctuations of the ASEP under a certain type of initial data; we postpone further discussion on this to Section \ref{EquationBoundary} and Section \ref{PhaseTransitionAsymptotics}.

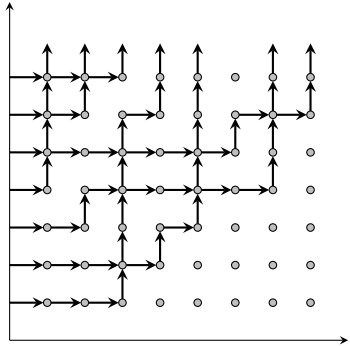
\begin{figure}[t]

\begin{center}

\begin{tikzpicture}[
      >=stealth,
			]

			\draw[->, black	] (0, 0) -- (0, 4.5);
			\draw[->, black] (0, 0) -- (4.5, 0);
			\draw[->,black, thick] (0, .5) -- (.45, .5);
			\draw[->,black, thick] (0, 1) -- (.45, 1);
			\draw[->,black, thick] (0, 1.5) -- (.45, 1.5);
			\draw[->,black, thick] (0, 2) -- (.45, 2);
			\draw[->,black, thick] (0, 2.5) -- (.45, 2.5);
			\draw[->,black, thick] (0, 3) -- (.45, 3);
			\draw[->,black, thick] (0, 3.5) -- (.45, 3.5);

			\draw[->,black, thick] (.55, .5) -- (.95, .5);
			\draw[->,black, thick] (.55, 1) -- (.95, 1);
			\draw[->,black, thick] (.55, 1.5) -- (.95, 1.5);
			\draw[->,black, thick] (.55, 2.5) -- (.95, 2.5);
			\draw[->,black, thick] (.55, 3) -- (.95, 3);
			\draw[->,black, thick] (.55, 3.5) -- (.95, 3.5);
			\draw[->,black, thick] (.5, 2.05) -- (.5, 2.45);
			\draw[->,black, thick] (.5, 2.55) -- (.5, 2.95);
			\draw[->,black, thick] (.5, 3.05) -- (.5, 3.45);
			\draw[->,black, thick] (.5, 3.55) -- (.5, 3.95);
			
			\draw[->,black, thick] (1.05, .5) -- (1.45, .5);
			\draw[->,black, thick] (1.05, 1) -- (1.45, 1);
			\draw[->,black, thick] (1.05, 2) -- (1.45, 2);
			\draw[->,black, thick] (1.05, 2.5) -- (1.45, 2.5);
			\draw[->,black, thick] (1.05, 3.5) -- (1.45, 3.5);
			\draw[->,black, thick] (1, 1.55) -- (1, 1.95);
			\draw[->,black, thick] (1, 3.05) -- (1, 3.45);
			\draw[->,black, thick] (1, 3.55) -- (1, 3.95);
			
			\draw[->,black, thick] (1.5, .55) -- (1.5, .95);
			\draw[->,black, thick] (1.5, 1.05) -- (1.5, 1.45);
			\draw[->,black, thick] (1.5, 1.55) -- (1.5, 1.95);
			\draw[->,black, thick] (1.5, 2.05) -- (1.5, 2.45);
			\draw[->,black, thick] (1.5, 2.55) -- (1.5, 2.95);
			\draw[->,black, thick] (1.5, 3.55) -- (1.5, 3.95);
			\draw[->,black, thick] (1.55, 1) -- (1.95, 1);
			\draw[->,black, thick] (1.55, 2) -- (1.95, 2);
			\draw[->,black, thick] (1.55, 2.5) -- (1.95, 2.5);
			\draw[->,black, thick] (1.55, 3) -- (1.95, 3);

			\draw[->,black, thick] (2, 1.05) -- (2, 1.45);
			\draw[->,black, thick] (2, 3.05) -- (2, 3.45);
			\draw[->,black, thick] (2, 3.55) -- (2, 3.95);
			\draw[->,black, thick] (2.05, 1.5) -- (2.45, 1.5);
			\draw[->,black, thick] (2.05, 2) -- (2.45, 2);
			\draw[->,black, thick] (2.05, 2.5) -- (2.45, 2.5);

			\draw[->,black, thick] (2.5, 1.55) -- (2.5, 1.95);
			\draw[->,black, thick] (2.5, 2.05) -- (2.5, 2.45);
			\draw[->,black, thick] (2.5, 2.55) -- (2.5, 2.95);
			\draw[->,black, thick] (2.5, 3.05) -- (2.5, 3.45);
			\draw[->,black, thick] (2.5, 3.55) -- (2.5, 3.95);
			\draw[->,black, thick] (2.55, 2) -- (2.95, 2);
			\draw[->,black, thick] (2.55, 2.5) -- (2.95, 2.5);

			\draw[->,black, thick] (3, 2.55) -- (3, 2.95);
			\draw[->,black, thick] (3.05, 2) -- (3.45, 2);
			\draw[->,black, thick] (3.05, 3) -- (3.45, 3);

			\draw[->,black, thick] (3.5, 2.05) -- (3.5, 2.45);
			\draw[->,black, thick] (3.5, 2.55) -- (3.5, 2.95);
			\draw[->,black, thick] (3.5, 3.05) -- (3.5, 3.45);
			\draw[->,black, thick] (3.5, 3.55) -- (3.5, 3.95);
			\draw[->,black, thick] (3.55, 3) -- (3.95, 3);
		
			\draw[->,black, thick] (4, 3.05) -- (4, 3.45);
			\draw[->,black, thick] (4, 3.55) -- (4, 3.95);

			\filldraw[fill=gray!50!white, draw=black] (.5, .5) circle [radius=.05];
			\filldraw[fill=gray!50!white, draw=black] (.5, 1) circle [radius=.05];
			\filldraw[fill=gray!50!white, draw=black] (.5, 1.5) circle [radius=.05];
			\filldraw[fill=gray!50!white, draw=black] (.5, 2) circle [radius=.05];
			\filldraw[fill=gray!50!white, draw=black] (.5, 2.5) circle [radius=.05];
			\filldraw[fill=gray!50!white, draw=black] (.5, 3) circle [radius=.05];
			\filldraw[fill=gray!50!white, draw=black] (.5, 3.5) circle [radius=.05];

			\filldraw[fill=gray!50!white, draw=black] (1, .5) circle [radius=.05];
			\filldraw[fill=gray!50!white, draw=black] (1, 1) circle [radius=.05];
			\filldraw[fill=gray!50!white, draw=black] (1, 1.5) circle [radius=.05];
			\filldraw[fill=gray!50!white, draw=black] (1, 2) circle [radius=.05];
			\filldraw[fill=gray!50!white, draw=black] (1, 2.5) circle [radius=.05];
			\filldraw[fill=gray!50!white, draw=black] (1, 3) circle [radius=.05];
			\filldraw[fill=gray!50!white, draw=black] (1, 3.5) circle [radius=.05];
			
			\filldraw[fill=gray!50!white, draw=black] (1.5, .5) circle [radius=.05];
			\filldraw[fill=gray!50!white, draw=black] (1.5, 1) circle [radius=.05];
			\filldraw[fill=gray!50!white, draw=black] (1.5, 1.5) circle [radius=.05];
			\filldraw[fill=gray!50!white, draw=black] (1.5, 2) circle [radius=.05];
			\filldraw[fill=gray!50!white, draw=black] (1.5, 2.5) circle [radius=.05];
			\filldraw[fill=gray!50!white, draw=black] (1.5, 3) circle [radius=.05];
			\filldraw[fill=gray!50!white, draw=black] (1.5, 3.5) circle [radius=.05];
			
			\filldraw[fill=gray!50!white, draw=black] (2, .5) circle [radius=.05];
			\filldraw[fill=gray!50!white, draw=black] (2, 1) circle [radius=.05];
			\filldraw[fill=gray!50!white, draw=black] (2, 1.5) circle [radius=.05];
			\filldraw[fill=gray!50!white, draw=black] (2, 2) circle [radius=.05];
			\filldraw[fill=gray!50!white, draw=black] (2, 2.5) circle [radius=.05];
			\filldraw[fill=gray!50!white, draw=black] (2, 3) circle [radius=.05];
			\filldraw[fill=gray!50!white, draw=black] (2, 3.5) circle [radius=.05];
			
			\filldraw[fill=gray!50!white, draw=black] (2.5, .5) circle [radius=.05];
			\filldraw[fill=gray!50!white, draw=black] (2.5, 1) circle [radius=.05];
			\filldraw[fill=gray!50!white, draw=black] (2.5, 1.5) circle [radius=.05];
			\filldraw[fill=gray!50!white, draw=black] (2.5, 2) circle [radius=.05];
			\filldraw[fill=gray!50!white, draw=black] (2.5, 2.5) circle [radius=.05];
			\filldraw[fill=gray!50!white, draw=black] (2.5, 3) circle [radius=.05];
			\filldraw[fill=gray!50!white, draw=black] (2.5, 3.5) circle [radius=.05];
			
			\filldraw[fill=gray!50!white, draw=black] (3, .5) circle [radius=.05];
			\filldraw[fill=gray!50!white, draw=black] (3, 1) circle [radius=.05];
			\filldraw[fill=gray!50!white, draw=black] (3, 1.5) circle [radius=.05];
			\filldraw[fill=gray!50!white, draw=black] (3, 2) circle [radius=.05];
			\filldraw[fill=gray!50!white, draw=black] (3, 2.5) circle [radius=.05];
			\filldraw[fill=gray!50!white, draw=black] (3, 3) circle [radius=.05];
			\filldraw[fill=gray!50!white, draw=black] (3, 3.5) circle [radius=.05];
			
			\filldraw[fill=gray!50!white, draw=black] (3.5, .5) circle [radius=.05];
			\filldraw[fill=gray!50!white, draw=black] (3.5, 1) circle [radius=.05];
			\filldraw[fill=gray!50!white, draw=black] (3.5, 1.5) circle [radius=.05];
			\filldraw[fill=gray!50!white, draw=black] (3.5, 2) circle [radius=.05];
			\filldraw[fill=gray!50!white, draw=black] (3.5, 2.5) circle [radius=.05];
			\filldraw[fill=gray!50!white, draw=black] (3.5, 3) circle [radius=.05];
			\filldraw[fill=gray!50!white, draw=black] (3.5, 3.5) circle [radius=.05];
			
			\filldraw[fill=gray!50!white, draw=black] (4, .5) circle [radius=.05];
			\filldraw[fill=gray!50!white, draw=black] (4, 1) circle [radius=.05];
			\filldraw[fill=gray!50!white, draw=black] (4, 1.5) circle [radius=.05];
			\filldraw[fill=gray!50!white, draw=black] (4, 2) circle [radius=.05];
			\filldraw[fill=gray!50!white, draw=black] (4, 2.5) circle [radius=.05];
			\filldraw[fill=gray!50!white, draw=black] (4, 3) circle [radius=.05];
			\filldraw[fill=gray!50!white, draw=black] (4, 3.5) circle [radius=.05];

\end{tikzpicture}

\end{center}	

\caption{\label{figurevertexwedge} A sample of the stochastic six-vertex model with step boundary data is depicted above.} 
\end{figure}

\subsubsection{The Stochastic Six-Vertex Model}

\label{StochasticVertex}

The stochastic six-vertex model was first introduced to the mathematical physics community by Gwa and Spohn \cite{SVMRSASH} in 1992 as a stochastic version of the more well-known six-vertex (ice) model studied by Lieb \cite{RESI} and Baxter \cite{ESMSM}; it was also studied more recently in \cite{SSVM, HSVMRSF, IPSVMSF}. This model can be defined in several equivalent ways, including as a ferromagnetic, asymmetric six-vertex model on a long rectangle with specific vertex weights (see \cite{SVMRSASH} or Section 2.1 of \cite{SSVM}); as an interacting particle system with push dynamics and an exclusion restriction (see \cite{SVMRSASH} or Section 2.2 of \cite{SSVM}); or as a probability measure on directed path ensembles (see Section 2 of \cite{SSVM} or Section 1 of \cite{HSVMRSF}). In this section, we will define the model through path ensembles. 

A \emph{six-vertex directed path ensemble} is a family of up-right directed paths in the non-negative quadrant $\mathbb{Z}_{> 0}^2$, such that each path emanates from either the $x$-axis or $y$-axis, and such that no two paths share an edge (although they may share vertices); see Figure \ref{figurevertexwedge}. In particular, each vertex has six possible \emph{arrow configurations}, which are listed in the top row of Figure \ref{sixvertexfigure}. \emph{Initial data}, or \emph{boundary data}, for such an ensemble is prescribed by dictating which vertices on the positive $x$-axis and positive $y$-axis are entrance sites for a directed path. One example of initial data is \emph{step initial data}, in which paths only enter through the $y$-axis, and all vertices on the $y$-axis are entrance sites for paths; see Figure \ref{figurevertexwedge}. 

Now, fix parameters $\delta_1, \delta_2 \in [0, 1]$ and some initial data. The \emph{stochastic six-vertex model} $\mathcal{P} = \mathcal{P} (\delta_1, \delta_2)$ will be the infinite-volume limit of a family of probability measures $\mathcal{P}_n = \mathcal{P}_n (\delta_1, \delta_2)$ defined on the set of six-vertex directed path ensembles whose vertices are all contained in triangles of the form $\mathbb{T}_n = \{ (x, y) \in \mathbb{Z}_{\ge 0}^2: x + y \le n \}$. The first such probability measure $\mathcal{P}_1$ is supported by the empty ensemble.

For each positive integer $n$, we define $\mathcal{P}_{n + 1}$ from $\mathcal{P}_n$ through the following Markovian update rules. Use $\mathcal{P}_n$ to sample a directed path ensemble $\mathcal{E}_n$ on $\mathbb{T}_n$. This gives arrow configurations (of the type shown in Figure \ref{sixvertexfigure}) to all vertices in the positive quadrant strictly below the diagonal $\mathbb{D}_n = \{ (x, y) \in \mathbb{Z}_{> 0}^2: x + y = n \}$. Each vertex on $\mathbb{D}_n$ is also given ``half'' of an arrow configuration, in the sense that it is given the directions of all entering paths but no direction of any exiting path. 

To extend $\mathcal{E}_n$ to a path ensemble on $\mathbb{T}_{n + 1}$, we must ``complete'' the configurations (specify the exiting paths) at all vertices $(x, y) \in \mathbb{D}_n$. Any half-configuration can be completed in at most two ways; selecting between these completions is done randomly, according to the probabilities given in the second row of Figure \ref{sixvertexfigure}, and independently among all vertices. 

In this way, we obtain a random ensemble $\mathcal{E}_{n + 1}$ on $\mathbb{T}_{n + 1}$; the resulting probability measure on path ensembles with vertices in $\mathbb{T}_{n + 1}$ is denoted $\mathcal{P}_{n + 1}$. Now, set $\mathcal{P} = \lim_{n \rightarrow \infty} \mathcal{P}_n$. 

As in the ASEP, there exists an observable of interest for stochastic six-vertex model called the \emph{current}. For our purposes, it will be defined as follows. Let $(X, Y) \in \mathbb{R}_{> 0}^2$. The \emph{current} (or \emph{height function}) $\mathfrak{H} (X, Y)$ of the stochastic six-vertex model at $(X, Y)$ is the number of paths that intersect the line $y = Y$ strictly to the right of $(X, Y)$. Since in this paper we only consider stochastic six-vertex models in which no vertices on the $x$-axis are entrance sites for paths, the quantity $\mathfrak{H} (X, Y)$ is always finite. 

We will again be interested in understanding the current fluctuations for the stochastic six-vertex model under certain types of initial data; we will discuss this further in the next two sections.

\begin{figure}[t]

\begin{center}

\begin{tikzpicture}[
      >=stealth,
			scale = .7
			]

			\draw[-, black] (-7.5, -.8) -- (7.5, -.8);
			\draw[-, black] (-7.5, 0) -- (7.5, 0);
			\draw[-, black] (-7.5, 2) -- (7.5, 2);
			\draw[-, black] (-7.5, -.8) -- (-7.5, 2);
			\draw[-, black] (7.5, -.8) -- (7.5, 2);
			\draw[-, black] (-5, -.8) -- (-5, 2);
			\draw[-, black] (5, -.8) -- (5, 2);
			\draw[-, black] (-2.5, -.8) -- (-2.5, 2);
			\draw[-, black] (2.5, -.8) -- (2.5, 2);
			\draw[-, black] (0, -.8) -- (0, 2);

			\draw[->, black,  thick] (-3.75, 1.1) -- (-3.75, 1.9);
			\draw[->, black,  thick] (-3.75, .1) -- (-3.75, .9);

			\draw[->, black,  thick] (-1.25, .1) -- (-1.25, .9);
			\draw[->, black,  thick] (-1.15, 1) -- (-.35, 1);

			\draw[->, black,  thick] (1.35, 1) -- (2.15, 1);
			\draw[->, black,  thick] (.35, 1) -- (1.15, 1);
			
			\draw[->, black,  thick] (3.75, 1.1) -- (3.75, 1.9);
			\draw[->, black,  thick] (2.85, 1) -- (3.65, 1);
			
			\draw[->, black,  thick] (6.25, 1.1) -- (6.25, 1.9);
			\draw[->, black,  thick] (6.25, .1) -- (6.25, .9);
			\draw[->, black,  thick] (6.35, 1) -- (7.15, 1);
			\draw[->, black,  thick] (5.35, 1) -- (6.15, 1);
				
			\filldraw[fill=gray!50!white, draw=black] (-6.25, 1) circle [radius=.1] node [black,below=21] {$1$};
			\filldraw[fill=gray!50!white, draw=black] (-3.75, 1) circle [radius=.1] node [black,below=21] {$\delta_1$};
			\filldraw[fill=gray!50!white, draw=black] (-1.25, 1) circle [radius=.1] node [black,below=21] {$1 - \delta_1$};
			\filldraw[fill=gray!50!white, draw=black] (1.25, 1) circle [radius=.1] node [black,below=21] {$\delta_2$};
			\filldraw[fill=gray!50!white, draw=black] (3.75, 1) circle [radius=.1] node [black,below=21] {$1 - \delta_2$};
			\filldraw[fill=gray!50!white, draw=black] (6.25, 1) circle [radius=.1] node [black,below=21] {$1$};

\end{tikzpicture}

\end{center}

\caption{\label{sixvertexfigure} The top row in the chart shows the six possible arrow configurations at vertices in the stochastic six-vertex model; the bottom row shows the corresponding probabilities. }
\end{figure}

\subsection{Context and Background}

\label{EquationBoundary}

The phenomenon that guides our results is commonly termed \emph{KPZ universality}. About thirty years ago in their seminal paper \cite{DSGI}, Kardar, Parisi, and Zhang considered a family of random growth models that exhibit ostensibly unusual (although now known to be quite ubiquitous) scaling behavior. As part of this work, they predicted the scaling exponents for all one-dimensional models in this family; specifically, after running such a model (with certain deterministic initial data) for some large time $T$, they predicted fluctuations of order $T^{1 / 3}$ and non-trivial spacial correlation on scales $T^{2 / 3}$. This (rather vaguely defined) family of random growth models is now called the \emph{Kardar-Parisi-Zhang (KPZ) universality class}, and consists of many more models (including the ASEP and stochastic six-vertex model) than those originally considered in \cite{DSGI}; we refer to the surveys \cite{EUC} and \cite{IUE} for more information.

In addition to predicting these exponents, Kardar, Parisi, and Zhang proposed a stochastic differential equation that in a sense embodies all of the models in their class; this equation, now known as the \emph{KPZ equation} is 
\begin{flalign}
\label{stochasticequation} 
\partial_t h = \partial_x^2 h + \displaystyle\frac{1}{2} \left( \partial_x h \right)^2 + \dot{\mathscr{W}},
\end{flalign}

\noindent where $\dot{\mathscr{W}}$ refers to space-time white noise. 

Granting the well-posedness \cite{SEPS, ATR, SE} of \eqref{stochasticequation} (which by no means immediate, due to the non-linearity $\frac{1}{2} \big( \partial_x h\big)^2$), it is widely believed that the long-time statistics of a stochastic model in the KPZ universality class should coincide with the long-time statistics of \eqref{stochasticequation}, whose initial data should be somehow chosen to ``match'' the initial data of the model, in a suitable way. In particular, distributions appearing as asymptotic marginals of the KPZ equation should also appear as the asymptotic one-point functions for models in the KPZ universality class.

For the ASEP and stochastic six-vertex model, these predictions have been verified in only a few contexts. The first such result was established by Johansson \cite{SFRM} in 2000 for the \emph{totally asymmetric simple exclusion process} (TASEP), which is the $L = 0$ case of the ASEP. Specifically, in the case of step initial data (all sites in $\mathbb{Z}_{\le 0}$ are occupied and all other sites are unoccupied), he showed that, after $T^{1 / 3}$ scaling, the current fluctuations converge to the Tracy-Widom distribution for the limiting law of the largest eigenvalue of a large Gaussian Hermitian random (GUE) matrix \cite{LSK}. This is in agreement with the result later established by Amir-Corwin-Quastel \cite{PDFECDRP}, Calabrese-Le-Doussal-Rosso \cite{FEDDPHT}, Dotsenko \cite{RADFEFDDP}, and Sasamoto-Spohn \cite{HDENWIC}, who showed that the long-time height fluctuations of the KPZ equation under what is known as \emph{narrow wedge initial data} also converge to the Tracy-Widom distribution, after $T^{1 / 3}$ scaling; see the survey \cite{EUC} for a thorough discussion on the level of mathematical justification behind these developments. 

The result of Johansson strongly relies on the \emph{free-fermionicity} (complete determinantal structure) that underlies the TASEP. This property is not believed to hold for the more general ASEP, which had posed trouble for extending Johansson's result to the ASEP with non-zero $L$ for several years. This was eventually overcome by Tracy and Widom in 2008 \cite{FDRA, IASEP}. Specifically, in the case of step initial data, they showed that the current fluctuations of the ASEP converge to the Tracy-Widom distribution after $T^{1 / 3}$ scaling \cite{AASIC}, which again establishes the KPZ universality conjecture for the ASEP with step initial data. 

Due to its two-dimensional nature and also its more complex Markov update rule, the stochastic six-vertex model has been less amenable to analysis. Only very recently has the analogous result been established for that model; this was done in the work \cite{SSVM} of Borodin, Corwin, and Gorin. In that paper, the authors analyze the stochastic six-vertex model with step initial data and show that the fluctuations of the current are of order $T^{1 / 3}$ when $(X, Y)$ lies in a certain cone (called a \emph{rarefaction fan} or \emph{liquid region}) of the positive quadrant. After rescaling by $T^{1 / 3}$, they show that the current fluctuations again converge to the Tracy-Widom distribution, thereby establishing the KPZ universality conjecture for the stochastic six-vertex model with step initial data. 

One might ask whether this type of KPZ universality phenomenon can be proven for the ASEP and stochastic six-vertex model under other classes of initial data or, equally interesting, what other types of distributions can appear as the long-time current fluctuations for these two models. For the stochastic six-vertex model, we know of no results in this direction. 

For the general ASEP with $L \ne 0$, we know of only one, which considers the case of \emph{step Bernoulli initial data}, in which sites in $\mathbb{Z}_{\le 0}$ are independently occupied with probability $b \in (0, 1]$ and no other sites are occupied (however, see the works \cite{DPPAEPWAIC, RFAEP, RGHFID} for partial results for the ASEP with flat and half-flat initial data; see also the works \cite{FEP, OCVDA} for partial results for the ASEP with stationary initial data, which are in fact completed to include the exact limiting statistics in the work \cite{CFCLEP} that is parallel with this paper). Under this initial data, Tracy and Widom \cite{ASBIC} considered the fluctuations of the current $J_T (\eta T)$ as $T$ tends to $\infty$, for fixed $\eta < 1$, and established the existence of a \emph{phase transition} (already known in the case of the TASEP \cite{LDCSCM, PTLECSCM, CFTP, CFTEP}). Specifically, they showed the following. 

\begin{enumerate}

\item{\label{process1} If $\eta \in (1 - 2b, 1)$, then the fluctuations of $J_T (\eta T)$ are GUE Tracy-Widom and of order $T^{1 / 3}$. }

\item{\label{process2} If $\eta = 1 - 2b$, then the fluctuations of $J_T (\eta T)$ are $F_1^2$ and of order $T^{1 / 3}$, where $F_1$ is the limiting law of the largest eigenvalue for a large random Gaussian real symmetric matrix. }

\item{\label{process3} If $\eta < 1 - 2b$, then the fluctuations of $J_T (\eta T)$ are Gaussian and of order $T^{1 / 2}$. }

\end{enumerate}

In particular, across the line $x = (1 - 2b) T$, the fluctuation exponent of $J_T (x)$ changes from $1 / 2$ to $1 / 3$ and, on this line, the fluctuations are of order $T^{1 / 3}$ and are $F_1^2$ instead of GUE Tracy-Widom. 

The corresponding initial data for the KPZ equation is half-Brownian initial data; this was analyzed by Corwin and Quastel in \cite{CDERF}, where it was shown that the asymptotic height fluctuations of this KPZ equation are of order $T^{1 / 3}$ and scale to the same $F_1^2$ distribution, in agreement with the second part of Tracy-Widom's result. 

Similar to the Tracy-Widom distribution, the phase transition described above was not first observed in the framework of interacting particle systems but rather in the context of random matrices \cite{PTLECSCM}. Specifically, they originally arose as the limiting law of the largest eigenvalue $\lambda_1$ of large $N \times N$ \emph{spiked covariance matrices}; these were studied at length by Baik, Ben Arous, and P\'{e}ch\'{e} in the paper \cite{PTLECSCM}, which we refer to for more information. They also appeared in the work of P\'{e}ch\'{e} \cite{LESRPHRM}, as the limiting law of the largest eigenvalue of finite-rank perturbations of large GUE matrices. 

We will not state P\'{e}ch\'{e}'s result is precise detail, but informally it can be described as follows. Fix an integer $m \ge 1$ and consider a rank $m$ perturbation of an $N \times N$ GUE matrix. 

\begin{enumerate}

\item{\label{matrix1} If the perturbation is sufficiently small, then the fluctuations $\lambda_1$ are of order $N^{- 2 / 3}$ and Tracy-Widom. }

\item{\label{matrix2} If the perturbation is \emph{critical} (in some appropriate sense), then the fluctuations of $\lambda_1$ are of order $N^{- 2 / 3}$ and scale to the distribution $F_{\BBP; 0^m}$, which is a \emph{rank $m$ Baik-Ben-Arous-P\'{e}ch\'{e} distribution} (see Definition \ref{righttransitiondistribution}). }

\item{\label{matrix3} If the perturbation is sufficiently large, then the fluctuations of $\lambda_1$ are of order $N^{- 1 / 2}$ and $G_m$, where $G_m$ denotes the law of the largest eigenvalue of an $m \times m$ GUE matrix (see Definition \ref{lefttransitiondistribution}). }

\end{enumerate}

The distribution $G_1$ is a Gaussian distribution, and the distribution $F_{\BBP, 0}$ is also known \cite{PTLECSCM} to coincide with $F_1^2$. Thus, the results of Tracy-Widom \cite{ASBIC} on the ASEP with step Bernoulli initial data strongly resemble the random matrix phase transitions described above in the case $m = 1$. 

This stimulates two questions. The first is whether one can generalize the previously mentioned result of Corwin-Quastel \cite{CDERF} on the KPZ equation with half-Brownian initial data, that is, whether it is possible to observe any higher rank Baik-Ben-Arous-P\'{e}ch\'{e} distributions as marginals of the KPZ equation with some (explicit) initial data. The second is whether one can generalize Tracy-Widom's \cite{ASBIC} to produce some explicit class of initial data (generalizing the step Bernoulli initial data) for the ASEP (and stochastic six-vertex model) under which one observes a phase transition across a characteristic line and the higher rank Baik-Ben-Arous-P\'{e}ch\'{e} distributions along the characteristic line. 

The first question was addressed by Borodin, Corwin, and Ferrari in the paper \cite{FEF}. In that work, the authors consider the KPZ equation with what they call \emph{$m$-spiked initial data} and establish that its height fluctuations (after $T^{1 / 3}$ scaling) converge to rank $m$ Baik-Ben-Arous-P\'{e}ch\'{e} distributions. 

The purpose of the present paper is to affirmatively answer the second question, thereby establishing KPZ universality of the ASEP and stochastic six-vertex model under this type of $m$-spiked distribution. We state our results more precisely in the next section.

\subsection{Results}

\label{PhaseTransitionAsymptotics}

In this section, we state our main results, which establish phase transitions for the ASEP and stochastic six-vertex model under what we call generalized step Bernoulli initial data; these are given by Theorem \ref{asymmetriclimit} and Theorem \ref{hlimit}. 

To state these results, we must define the relevant distributions, namely the GUE Tracy-Widom distribution, the Baik-Ben-Arous-P\'{e}ch\'{e} distributions, and the finite GUE distributions $G_m$. To define the former two distributions, we require the following kernels; each of these kernels will be normalized by a factor of $2 \pi \textbf{i}$ in order to simplify notation. 

\begin{definition}

\label{functionkernel}

Denoting by $\Ai (x)$ the Airy function, the (normalized) \emph{Airy kernel} $K_{\Ai} (x, y)$ is defined by 
\begin{flalign*}
K_{\Ai} (x, y) & = \displaystyle\frac{1}{2 \pi \textbf{i}} \displaystyle\oint \displaystyle\oint \exp \left( \displaystyle\frac{w^3}{3} - \displaystyle\frac{v^3}{3} - xv + yw \right) \displaystyle\frac{dw dv}{w - v}, 
\end{flalign*}

\noindent where in the latter identity the contour for $w$ is piecewise linear from $\infty e^{- \pi \textbf{i} / 3}$ to $0$ to $\infty e^{\pi \textbf{i} / 3}$, and the contour for $v$ is piecewise linear from $\infty e^{- 2 \pi \textbf{i} / 3}$ to $-1$ to $\infty e^{2 \pi \textbf{i} / 3}$; see Figure \ref{l1l2l3} in Section \ref{RightKernel} below. 

\end{definition}

\begin{definition}
\label{functionkernel2} 

Let $m$ be a positive integer, and let $\textbf{c} = (c_1, c_2, \ldots , c_m) \in \mathbb{R}^m$ be a sequence of $m$ real numbers. The (normalized) \emph{Baik-Ben-Arous-P\'{e}ch\'{e} kernel} $K_{\BBP; \textbf{c}} (x, y)$ is defined by 
\begin{flalign*}
K_{\BBP; \textbf{c}} (x, y) & = \displaystyle\frac{1}{2 \pi \textbf{i}} \displaystyle\oint \displaystyle\oint \exp \left( \displaystyle\frac{w^3}{3} - \displaystyle\frac{v^3}{3} - xv + yw \right) \displaystyle\prod_{j = 1}^m \displaystyle\frac{v + c_j}{w + c_j} \displaystyle\frac{dw dv}{w - v}, 
\end{flalign*}

\noindent where the contour for $w$ is piecewise linear from $\infty e^{- \pi \textbf{i} / 3}$ to $-E$ to $\infty e^{\pi \textbf{i} / 3}$, and the contour for $v$ is piecewise linear from $\infty e^{- 2 \pi \textbf{i} / 3}$ to $- E - 1$ to $\infty e^{2 \pi \textbf{i} / 3}$. Here, $E \in \mathbb{R}$ is chosen so that $E > \max_{1 \le i \le m} c_i$; see Figure \ref{shiftedcontours13} in Section \ref{VertexNear} below.
\end{definition}

Now we can define the GUE Tracy-Widom distribution and the Baik-Ben-Arous-P\'{e}ch\'{e} distributions. We refer to Appendix \ref{Determinants1} for our conventions on Fredholm determinants (in particular, there will sometimes be a normalization of $2 \pi \textbf{i}$). 

\begin{definition}

\label{righttransitiondistribution}

The \emph{GUE Tracy-Widom distribution} $F_{\TW} (s)$ is defined by 
\begin{flalign*}
F_{\TW} (s) = \det \big( \Id - K_{\Ai} \big)_{L^2 (s, \infty)}, 
\end{flalign*}

\noindent for each real number $s \in \mathbb{R}$. 

Similarly, if $m$ is a positive integer and $\textbf{c} = (c_1, c_2, \ldots , c_m) \in \mathbb{R}^m$ is a sequence of $m$ real numbers, the Baik-Ben-Arous-P\'{e}ch\'{e} distribution $F_{\BBP; \textbf{c}} (s)$ is defined by 
\begin{flalign*}
F_{\BBP; \textbf{c}} (s) = \det \big( \Id - K_{\BBP; \textbf{c}} \big)_{L^2 (s, \infty)}. 
\end{flalign*}
\end{definition}

 The following defines the GUE distribution $G_m$. 

\begin{definition}
\label{lefttransitiondistribution} 

For any positive integer $m$, define the distribution $G_m (s)$ through 
\begin{flalign*}
G_m (s) = Z_m^{-1} \displaystyle\int_{-\infty}^s \displaystyle\int_{-\infty}^s \cdots \displaystyle\int_{-\infty}^s \displaystyle\prod_{1 \le j < k \le m} |x_j - x_k|^2 \displaystyle\prod_{j = 1}^m e^{-x_j^2 /2} d x_j, 
\end{flalign*}

\noindent for each real number $s \in \mathbb{R}$; here, $Z_m$ is a normalization constant chosen so that $G_m (\infty) = 1$. 
\end{definition}

Having defined the above distributions, we can now proceed to state our results. They concern the ASEP and stochastic six-vertex model with \emph{generalized $(b_1, b_2, \ldots , b_m)$-Bernoulli initial data}. This initial data is not so quickly described; in particular, it relies on the \emph{stochastic higher spin vertex models}, to be defined in Section \ref{GeneralVertexModels}. Thus, we reserve the definition of this initial data to Definition \ref{generalized}. 

However, let us give a few properties of this initial data that might be of immediate interest. 

\begin{itemize}

\item{In the case of the ASEP, all sites in $\mathbb{Z}_{> 0}$ are initially unoccupied; in the case of the stochastic six-vertex model, no vertex on the $x$-axis is an entrance site for a path. }

\item{When $m = 1$ and $b_1 = b \in (0, 1)$, this degenerates to \emph{step-Bernoulli initial data}. For the ASEP, this means that particles in $\mathbb{Z}_{> 0 }$ are initially occupied independently, with probability $b$. For the stochastic six-vertex model, this means that vertices on the positive $y$-axis are entrance sites for paths independently, with probability $b$.}

\item{Suppose $m > 1$ and $b_1 = b_2 = \cdots = b_m = b \in (0, 1)$. In the case of the ASEP, one expects (with high probability) the density of particles to the left of the origin to be approximately $b$ (in the sense that one expects approximately $b |I|$ sites to initially be occupied in an interval $I \subset \mathbb{Z}_{< 0}$). Similarly, in the case of the stochastic six-vertex model, one expects the density of entrance sites on the positive $y$-axis to be approximately $b$.}

\item{However, in the setting above (with $m > 1$ and all $b_j = b$), the events that different sites are occupied (for the ASEP) or different vertices are entrance sites (for the stochastic six-vertex model) are highly correlated. This correlation is what leads to the new phase transitions $F_{\BBP; \textbf{c}}$ and higher GUE distributions $G_m$. }

\item{In the case of the TASEP ($L = 0$), this initial data corresponds to the random particle configuration formed after running $m$ steps of the geometric \emph{Push TASEP} (with geometric rates $b_1, b_2, \ldots , b_m$) \cite{LTAGMSLP} on step initial data. This is known to have a much more direct relationship with the spectral distribution of the spiked GUE matrices considered by P\'{e}ch\'{e} \cite{AGRSD, LESRPHRM}, which provides some explanation for the asymptotics stated in Theorem \ref{asymmetriclimit} below, when $L = 0$. In the general case $L \ne 0$, we know of no such explanation. }

\end{itemize}

The first property is directly stipulated as part of Definition \ref{generalized}, and the second property is explained in Remark \ref{m1generalized}. The third, fourth, and fifth properties can also be derived from Definition \ref{generalized}, but establishing them carefully does not seem so relevant for the purposes of this paper. Thus, we omit their proofs.  

Our results are as follows; Theorem \ref{asymmetriclimit} considers the ASEP with generalized step Bernoulli initial data and Theorem \ref{hlimit} considers the stochastic six-vertex model with this initial data. 

\begin{thm} 
\label{asymmetriclimit}

Fix positive real numbers $R > L$, a real number $b \in (0, 1)$, an integer $m \ge 1$, and infinite families of real numbers $\{ b_{1, T} \}_{T \in \mathbb{R}_{> 0}}, \{ b_{2, T} \}_{T \in \mathbb{R}_{> 0}}, \ldots , \{ b_{m, T} \}_{T \in \mathbb{R}_{> 0}} \subset (0, 1)$. Assume that $R - L = 1$ and that there exist real numbers $d_1, d_2, \ldots , d_m \in \mathbb{R}$ such that $\lim_{T \rightarrow \infty} T^{1 / 3} (b_{i, T} - b) = d_i$, for each $i \in [1, m]$. 

Let $T$ be a positive real number. Consider the ASEP, run for time $T$, with left jump rate $L$, right jump rate $R$, and generalized $(b_{1, T}, b_{2, T}, \ldots , b_{m, T})$-Bernoulli initial data (given by Definition \ref{generalized}). 

Then we have the following results, where in the below $\theta = 1 - 2 b$ and $\chi = b (1 - b)$. 

\begin{enumerate}
\item{Assume that $\eta \in (\theta, 1)$ is a real number. Set  
\begin{flalign}
\label{processfluctuationslargeeta}
m_{\eta} & = \left( \displaystyle\frac{1 - \eta}{2} \right)^2; \qquad f_{\eta} =\left( \displaystyle\frac{1 - \eta^2}{4} \right)^{2 / 3}. 
\end{flalign}

\noindent Then, for any real number $s \in \mathbb{R}$, we have that 
\begin{flalign*}
\displaystyle\lim_{T \rightarrow \infty} \mathbb{P} \left[ \displaystyle\frac{m_{\eta} T - J_T (\eta T)}{f_{\eta} T^{1 / 3}} \le s \right] = F_{\TW} (s). 
\end{flalign*}

}

\item{Assume that $\{ \eta_T \}_{T \ge 0}$ is a sequence of real numbers such that $\lim_{T \rightarrow \infty} T^{1 / 3} (\eta_T - \theta) = d$, for some real number $d$. Set $m_{\eta}$ and $f_{\eta}$ as in \eqref{processfluctuationslargeeta}, and for each index $j \in [1, m]$, define 
\begin{flalign}
\label{cprocessphasetransition}
c_j = - \displaystyle\frac{f_{\eta} (2 d_j + d)}{2 \chi}; \qquad \textbf{\emph{c}} = (c_1, c_2, \ldots , c_m). 
\end{flalign} 

\noindent Then, for any real number $s \in \mathbb{R}$, we have that 
\begin{flalign*}
\displaystyle\lim_{T \rightarrow \infty} \mathbb{P} \left[ \displaystyle\frac{m_{\eta_T} T - J_{T} (\eta_T T )}{f_{\theta} T^{1 / 3}} \le s \right] = F_{\BBP; \textbf{\emph{c}}} (s). 
\end{flalign*} 
}

\item{Assume that $\eta \in (-b, \theta)$ is a real number, and that $b_j = b$ for all indices $j \in [1, m]$. Set  
\begin{flalign}
\label{processfluctuationslargeetaleft}
m_{\eta}' & = \chi  - b \eta ; \qquad f_{\eta}' = \chi^{1 / 2} (\theta - \eta)^{1 / 2}. 
\end{flalign}

\noindent Then, for any real number $s \in \mathbb{R}$, we have that 
\begin{flalign*}
\displaystyle\lim_{T \rightarrow \infty} \mathbb{P} \left[ \displaystyle\frac{m_{\eta}' T - J_T (\eta T)}{f_{\eta}' T^{1 / 2}} \le s \right] = G_h (s). 
\end{flalign*}
}
\end{enumerate}
\end{thm}

\begin{thm} 
\label{hlimit}

Fix positive real numbers $\delta_1 < \delta_2 < 1$, a real number $b \in (0, 1)$, an integer $m \ge 1$, and infinite families of real numbers $\{ b_{1, T} \}_{T \in \mathbb{Z}_{> 0}}, \{ b_{2, T} \}_{T \in \mathbb{Z}_{> 0}}, \ldots , \{ b_{m, T} \}_{T \in \mathbb{Z}_{> 0}} \subset (0, 1)$. Assume that there exist real numbers $d_1, d_2, \ldots , d_m \in \mathbb{R}$ such that $\lim_{T \rightarrow \infty} T^{1 / 3} (b_{i, T} - b) = d_i$, for each $i \in [1, m]$. 

Let $T > 0$ be a positive integer. Consider the stochastic six-vertex model, run for time $T$, with left jump probability $\delta_1$, right jump probability $\delta_2$, and generalized $(b_{1, T}, b_{2, T}, \ldots , b_{m, T})$-Bernoulli initial data (given by Definition \ref{generalized}). 

We have the following results, where in the below we set 
\begin{flalign}
\label{processlocationtransition}
\chi = b (1 - b); \qquad \kappa = \displaystyle\frac{1 - \delta_1}{1 - \delta_2}; \qquad \Lambda = b + \kappa (1 - b); \qquad \theta = \kappa^{-1} \Lambda^2.
\end{flalign} 

\begin{enumerate}
\item{Assume that $x$ and $y$ are positive real numbers such that $x / y \in (\theta, \kappa)$. Set  
\begin{flalign}
\label{modelfluctuationslargeeta}
\mathcal{H} (x, y) & = \displaystyle\frac{\big( \sqrt{(1 - \delta_1) y } - \sqrt{(1 - \delta_2) x} \big)^2}{\delta_1 - \delta_2} \nonumber  \\
 \mathcal{F} (x, y) & = \displaystyle\frac{\kappa^{1 / 6} (\sqrt{\kappa x } - \sqrt{y})^{2 / 3} (\sqrt{\kappa y} - \sqrt{x})^{2 / 3} }{(\kappa - 1) x^{1 / 6} y^{1 / 6}}. 
\end{flalign}

\noindent Then, for any real number $s \in \mathbb{R}$, we have that 
\begin{flalign*}
\displaystyle\lim_{T \rightarrow \infty} \mathbb{P} \left[ \displaystyle\frac{\mathcal{H} (x, y) T - \mathfrak{H} (xT, yT)}{\mathcal{F} (x, y) T^{1 / 3}} \le s \right] = F_{\TW} (s). 
\end{flalign*}

}

\item{Assume that $\{ x_T \}_{T \in \mathbb{Z}_{> 0}}$ and $\{ y_T \}_{T \in \mathbb{Z}_{> 0}}$ are sequences of positive real numbers such that $\eta_T = x_T / y_T$ satisfies $\lim_{T \rightarrow \infty} T^{1 / 3} (\eta_T - \theta) = d$, for some real number $d$. Set $\mathcal{H} (x, y)$ and $\mathcal{F} (x, y)$ as in \eqref{processfluctuationslargeeta}, and for each index $j \in [1, m]$ define 
\begin{flalign}
\label{cmodelphasetransition}
c_j = - \displaystyle\frac{f_{\eta} d_j}{\chi} - \displaystyle\frac{ \kappa f_{\eta} d}{2 (\kappa - 1) \chi \Lambda}; \qquad \textbf{\emph{c}} = (c_1, c_2, \ldots , c_m). 
\end{flalign} 

\noindent Then, for any real number $s \in \mathbb{R}$, we have that 
\begin{flalign*}
\displaystyle\lim_{T \rightarrow \infty} \mathbb{P} \left[ \displaystyle\frac{\mathcal{H} (x, y) T - \mathfrak{H} (xT, yT)}{\mathcal{F} (x, y) T^{1 / 3}} \le s \right] = F_{\BBP; \textbf{\emph{c}}} (s). 
\end{flalign*} 
}

\item{Assume that $x$ and $y$ are positive real numbers such that $x \in \big( \Lambda^{-1} \theta y, \theta y \big)$, and that $b_j = b$ for all indices $j \in [1, m]$. Set  
\begin{flalign}
\label{modelfluctuationslargeetaleft}
\mathcal{H}' (x, y) & = b y  - \Lambda^{-1} b x ; \qquad \mathcal{F}' (x, y) = \chi^{1 / 2} \left(  y - \theta^{-1} x \right)^{1 / 2}. 
\end{flalign}

\noindent Then, for any real number $s \in \mathbb{R}$, we have that 
\begin{flalign*}
\displaystyle\lim_{T \rightarrow \infty} \mathbb{P} \left[ \displaystyle\frac{\mathcal{H}' (x, y) T - \mathfrak{H} (xT, yT)}{\mathcal{F}' (x, y) T^{1 / 2}} \le s \right] = G_m (s). 
\end{flalign*}
}
\end{enumerate}
	
\end{thm} 

\begin{rem}

We expect Theorem \ref{asymmetriclimit} to hold for all $\eta < \theta$ and Theorem \ref{hlimit} to hold for all $x / y \in (0, \theta)$ but will not pursue this potential improvement further. However, in Appendix \ref{Asymptotics2}, we outline a different method to establish asymptotics through a comparison with Schur measures. It seems likely that one can use this method to remove the limitation on $\eta$ in the Gaussian regime.
\end{rem}

What enables the proofs of Theorem \ref{asymmetriclimit} and Theorem \ref{hlimit} is the integrability of the \emph{inhomogeneous stochastic higher spin vertex models} \cite{HSVMRSF}, which we will define in Section \ref{GeneralVertexModels}. These are a family of quite general vertex models in the positive quadrant that exhibit remarkable integrability properties and degenerate to many known models in the KPZ universality class; in particular, they degenerate to the stochastic six-vertex model and the ASEP (both with generalized step Bernoulli initial data), facts which will be discussed further in Section \ref{Spin12} and Section \ref{GeneralizedSpecialization}. 

These inhomogeneous stochastic higher spin vertex models were studied at length in the paper \cite{HSVMRSF}, in which multi-fold contour integral identities were obtained for certain observables ($q$-moments) of these models. In particular, it will be possible to directly degenerate these identities to the stochastic six-vertex model and ASEP with generalized step Bernoulli initial data; this will be done in Section \ref{ContourMoments}. It is known \cite{PTQT, MP, FEF, SSVM, DDQA} how to use such identities to produce Fredholm determinant identities that are directly amenable to asymptotic analysis, which we will do in Section \ref{DeterminantVertical}. We especially emphasize here the work \cite{PTQT} of Barraquand, who used similar Fredholm determinant identities to derive Baik-Ben-Arous-P\'{e}ch\'{e} phase transitions (of the type given in Theorem \ref{asymmetriclimit} and Theorem \ref{hlimit}) in the context of the $q$-TASEP with several slow particles. 

In Section \ref{TransitionsModel}, we will begin the asymptotic analysis of these Fredholm determinant identities and explain how they can be used for the proofs of Theorem \ref{asymmetriclimit} and Theorem \ref{hlimit}. In Section \ref{RightKernel}, Section \ref{VertexNear}, and Section \ref{Fluctuations12}, we will complete this analysis in the first, second, and third regimes, respectively, of Theorem \ref{asymmetriclimit} and Theorem \ref{hlimit}.

\subsection*{Acknowledgements}

We heartily thank Vadim Gorin and Guillaume Barraquand for very valuable conversations. The work of Amol Aggarwal was partially funded by the Eric Cooper and Naomi Siegel Graduate Student Fellowship I and the NSF Graduate Research Fellowship under grant number DGE1144152. The work of Alexei Borodin was partially supported by the NSF grants DMS-1056390 and DMS-1607901.

\section{Stochastic Higher Spin Vertex Models}

\label{GeneralVertexModels}

Both the ASEP and stochastic six-vertex model are degenerations of a larger class of vertex models called the \emph{inhomogeneous stochastic higher spin vertex models}, which were recently introduced by Borodin and Petrov in \cite{HSVMRSF}. These models are in a sense the original source of integrability for the ASEP, the stochastic six-vertex model, and in fact most models proven to be in the Kardar-Parisi-Zhang (KPZ) universality class. 

For that reason, we will begin our discussion by defining these vertex models. Similar to the stochastic six-vertex model, these models will take place on \emph{directed path ensembles}. We will first carefully define what we mean by a directed path ensemble in Section \ref{PathEnsembles}, and then we will define the stochastic higher spin vertex model in Section \ref{ProbabilityMeasures}. In Section \ref{InteractingParticles}, we will re-interpret these models as interacting particle systems, which will provide a useful framework for discussing observables in Section \ref{ObservablesVertical}.

\subsection{Directed Path Ensembles}

\label{PathEnsembles}

For the purpose of this paper, a \emph{directed path} is a collection of \emph{vertices}, which are lattice points in the non-negative quadrant $\mathbb{Z}_{\ge 0}^2$, connected by \emph{directed edges} (which we may also refer to as \emph{arrows}). A directed edge can connect a vertex $(i, j)$ to either $(i + 1, j)$ or $(i, j + 1)$, if $(i, j) \in \mathbb{Z}_{> 0}^2$; we also allow directed edges to connect $(k, 0)$ to $(k, 1)$ or $(0, k)$ to $(1, k)$, for any positive integer $k$. Thus, directed edges connect adjacent vertices, always point either up or to the right, and do not lie on the $x$-axis or $y$-axis. 

A \emph{directed path ensemble} is a collection of paths with the following two properties.

\begin{itemize}
\item{ Each path must contain an edge connecting $(0, k)$ to $(1, k)$ or $(k, 0)$ to $(k, 1)$ for some $k > 0$; stated alternatively, every path ``emanates'' from either the $x$-axis or the $y$-axis.}

\item{ No two distinct paths can share a horizontal edge; however, in contrast with the six-vertex case, they may share vertical edges.}

\end{itemize}

\noindent  An example of a directed path ensemble was previously shown in Figure \ref{figurevertexwedge}. See also Figure \ref{shift2generalizedvertex} in Section \ref{Spin12} below for more examples. 

Associated with each $(x, y) \in \mathbb{Z}_{> 0 }^2$ in a path ensemble is an \emph{arrow configuration}, which is a quadruple $(i_1, j_1; i_2, j_2) = (i_1, j_1; i_2, j_2)_{(x, y)}$ of non-negative integers. Here, $i_1$ denotes the number of directed edges from $(x, y - 1)$ to $(x, y)$; equivalently, $i_1$ denotes the number of vertical incoming arrows at $(x, y)$. Similarly, $j_1$ denotes the number of horizontal incoming arrows; $i_2$ denotes the number of vertical outgoing arrows (that is, the number of directed edges from $(x, y)$ to $(x, y +1)$); and $j_2$ denotes the number of horizontal outgoing arrows. Thus $j_1, j_2 \in \{ 0, 1 \}$ at every vertex in a path ensemble, since no two paths share a horizontal edge. An example of an arrow configuration (which cannot be a vertex in a directed path ensemble, since $j_1, j_2 \notin \{ 0, 1 \}$) is depicted in Figure \ref{arrows}. 

Assigning values $j_1$ to vertices on the line $(1, y)$ and values $i_1$ to vertices on the line $(x, 1)$ can be viewed as imposing boundary conditions on the vertex model. If $j_1 = 1$ at $(1, k)$ and $i_1 = 0$ at $(k, 1)$ for each $k > 0$, then all paths enter through the $y$-axis, and every vertex on the positive $y$-axis is an entrance site for some path. We will refer to this particular assignment as \emph{step initial data}; it was depicted in Figure \ref{figurevertexwedge}, and it is also depicted on the left side of Figure \ref{shift2generalizedvertex}. In general, we will refer to any assignment of $i_1$ to $\mathbb{Z}_{> 0} \times \{ 1 \}$ and $j_1$ to $\{ 1 \} \times \mathbb{Z}_{> 0}$ as \emph{initial data}, which can be deterministic (like step) or random. 
	
Observe that, at any vertex in the positive quadrant, the total number of incoming arrows is equal to the total number of outgoing arrows; that is, $i_1 + j_1 = i_2 + j_2$. This is sometimes referred to as \emph{arrow conservation} (or \emph{spin conservation}). Any arrow configuration to all vertices of $\mathbb{Z}_{> 0}^2$ that satisfies arrow conservation corresponds to a unique directed path ensemble, in which paths can share both vertical and horizontal edges. 

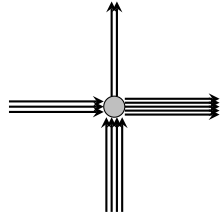
\begin{figure}

\begin{center} 

\begin{tikzpicture}[
      >=stealth,
      scale = .7
			]
			
			\draw[->,black, thick] (-.05, .2) -- (-.05, 2);
			\draw[->,black, thick] (.05, .2) -- (.05, 2);
			\draw[->,black, thick] (-.15,-2) -- (-.15, -.2);
			\draw[->,black, thick] (-.05,-2) -- (-.05, -.2);
			\draw[->,black, thick] (.05,-2) -- (.05, -.2);
			\draw[->,black, thick] (.15,-2) -- (.15, -.2);
			\draw[->,black, thick] (-2, -.1) -- (-.2, -.1);
			\draw[->,black, thick] (-2, 0) -- (-.2, 0);
			\draw[->,black, thick] (-2,.1) -- (-.2, .1);
			\draw[->,black, thick] (.2, -.15) -- (2, -.15); 
			\draw[->,black, thick] (.2, -.075) -- (2, -.075); 
			\draw[->,black, thick] (.2, 0) -- (2, 0); 
			\draw[->,black, thick] (.2, .075) -- (2, .075); 
			\draw[->,black, thick] (.2, .15) -- (2, .15);

			\filldraw[fill=gray!50!white, draw=black] (0, 0) circle [radius=.2];

\end{tikzpicture}

\end{center}

\caption{\label{arrows} Above is a vertex at which $(i_1, j_1; i_2, j_2) = (4, 3; 2, 5)$.} 
\end{figure}

\subsection{Probability Measures on Path Ensembles}

\label{ProbabilityMeasures}

The definition of the stochastic higher spin vertex models will closely resemble the definition of the stochastic six-vertex model given in Section \ref{StochasticVertex}. Specifically, we will first define probability measures $\mathbb{P}_n$ on the set of directed path ensembles whose vertices are all contained in triangles of the form $\mathbb{T}_n = \{ (x, y) \in \mathbb{Z}_{\ge 0}^2: x + y \le n \}$, and then we will take a limit as $n$ tends to infinity to obtain the vertex models in infinite volume. The first two measures $\mathbb{P}_0$ and $\mathbb{P}_1$ are both supported by the empty ensembles. 

For each positive integer $n \ge 1$, we will define $\mathbb{P}_{n + 1}$ from $\mathbb{P}_n$ through the following Markovian update rules. Use $\mathbb{P}_n$ to sample a directed path ensemble $\mathcal{E}_n$ on $\mathbb{T}_n$. This yields arrow configurations for all vertices in the triangle $\mathbb{T}_{n - 1}$. To extend this to a path ensemble on $\mathbb{T}_{n + 1}$, we must prescribe arrow configurations to all vertices $(x, y)$ on the complement $\mathbb{T}_n \setminus\mathbb{T}_{n - 1}$, which is the diagonal $\mathbb{D}_n = \{ (x, y) \in \mathbb{Z}_{> 0}^2: x + y = n \}$. Since any incoming arrow to $\mathbb{D}_n$ is an outgoing arrow from $\mathbb{D}_{n - 1}$, $\mathcal{E}_n$ and the initial data prescribe the first two coordinates, $i_1$ and $j_1$, of the arrow configuration to each $(x, y) \in \mathbb{D}_n$. Thus, it remains to explain how to assign the second two coordinates ($i_2$ and $j_2$) to any vertex on $\mathbb{D}_n$, given the first two coordinates. 

This is done by producing $(i_2, j_2)_{(x, y)}$ from $(i_1, j_1)_{(x, y)}$ according to the transition probabilities 
\begin{flalign}
\begin{aligned}
\label{configurationprobabilities} 
& \mathbb{P}_n \big[ (i_2, j_2) = (k, 0) \big| (i_1, j_1) = (k, 0) \big] = \displaystyle\frac{1 - q^k s_x \xi_x u_y}{1 - s_x \xi_x u_y}, \\
& \mathbb{P}_n \big[ (i_2, j_2) = (k - 1, 1) \big| (i_1, j_1) = (k, 0) \big] = \displaystyle\frac{(q^k - 1) s_x \xi_x u_y}{1 - s_x \xi_x u_y}, \\
& \mathbb{P}_n \big[ (i_2, j_2) = (k + 1, 0) \big| (i_1, j_1) = (k, 1) \big] = \displaystyle\frac{1 - q^k s_x^2}{1 - s_x \xi_x u_y},  \\
& \mathbb{P}_n \big[ (i_2, j_2) = (k, 1) \big| (i_1, j_1) = (k, 1) \big] = \displaystyle\frac{q^k s_x^2 - s_x \xi_x u_y}{1 - s_x \xi_x u_y},
\end{aligned}   
\end{flalign}

\noindent for any non-negative integer $k$. We also set $\mathbb{P}_n [(i_2, j_2) | (i_1, j_1)] = 0$ for all $(i_1, j_1; i_2, j_2)$ not of the above form. In the above, $q \in \mathbb{C}$ is a complex number and $U = (u_1, u_2, \ldots ) \subset \mathbb{C}$, $\Xi = (\xi_1, \xi_2, \ldots ) \subset \mathbb{C}$, and $S = (s_1, s_2, \ldots ) \subset \mathbb{C}$ are infinite sets of complex numbers, chosen to ensure that all of the above probabilities are non-negative. This can be arranged for instance when $q \in (0, 1)$, $U \subset (-\infty, 0]$, and $S, \Xi \subset [0, 1]$, although there are also other suitable choices. 

Choosing $(i_2, j_2)$ according to the above transition probabilities yields a random directed path ensemble $\mathcal{E}_{n + 1}$, now defined on $\mathbb{T}_{n + 1}$; the probability distribution of $\mathcal{E}_{n + 1}$ is then denoted by $\mathbb{P}_{n + 1}$. We define $\mathbb{P} = \lim_{n \rightarrow \infty} \mathbb{P}_n$. Then, $\mathbb{P}$ is a probability measure on the set of directed path ensembles that is dependent on the complex parameters $q$, $U$, $\Xi$, and $S$. The variables $U = (u_1, u_2, \ldots )$ are occasionally referred to as \emph{spectral parameters}, the variables $\Xi = (\xi_1, \xi_2, \ldots )$ as \emph{spacial inhomogeneity parameters}, and the variables $S = (s_1, s_2, \ldots )$ as \emph{spin parameters}. 

If there exists a positive integer $I$ such that $s^2 q^I = 1$ for all $s \in S$, then $\mathbb{P}[(i_2, j_2) = (I + 1, 0) | (i_1, j_1) = (I, 1)] = 0$; thus, the number of vertical incoming or outgoing arrows at any vertex remains less than $I + 1$. In this case, the vertex model is said to have \emph{spin $I / 2$}; if no such $I$ exists, then the spin of the model is said to be infinite.

\subsection{Vertex Models and Interacting Particle Systems}

\label{InteractingParticles}

Let $\mathbb{P}^{(T)}$ denote the restriction of the random path ensemble with step initial data (given by the measure $\mathbb{P}$ from the previous section) to the strip $\mathbb{Z}_{> 0} \times [0, T]$. Assume that all $T$ paths in this restriction almost surely exit the strip $\mathbb{Z}_{> 0} \times [0, T]$ through its top boundary; this will be the case, for instance, if $\mathbb{P}_n \big[ (i_2, j_2) = (0, k) \big| (i_1, j_1) = (0, k) \big] < 1$ for each $n$ and $k$. 
	
We will use the probability measure $\mathbb{P}^{(T)}$ to produce a discrete-time interacting particle system on $\mathbb{Z}_{> 0}$, defined up to time $T - 1$, as follows. Sample a line ensemble $\mathcal{E}$ randomly under $\mathbb{P}^{(T)}$, and consider the arrow configuration it associates to some vertex $(p, t) \in \mathbb{Z}_{> 0} \times [1, T]$. We will place $k$ particles at position $p$ and time $t - 1$ if and only if $i_1 = k $ at the vertex $(p, t)$. Therefore, the paths in the path ensemble $\mathcal{E}$ correspond to space-time trajectories of the particles. 

Let us introduce notation for particle positions. A \emph{non-negative signature} $\lambda = (\lambda_1, \lambda_2, \ldots , \lambda_n)$ of \emph{length} $n$ is a non-increasing sequence of $n$ integers $\lambda_1 \ge \lambda_2 \ge \cdots \ge \lambda_n \ge 0$. We denote the set of non-negative signatures of length $n$ by $\Sign_n^+$, and the set of all non-negative signatures by $\Sign^+ = \bigcup_{N = 0}^{\infty} \Sign_n^+$. For any signature $\lambda$ and integer $j$, let $m_j (\lambda)$ denote the number of indices $i$ for which $\lambda_i = j$; that is, $m_j (\lambda)$ is the multiplicity of $j$ in $\lambda$. 

We can associate a configuration of $n$ particles in $\mathbb{Z}_{\ge 0}$ with a signature of length $n$ as follows. A signature $\lambda = (\lambda_1, \lambda_2, \ldots , \lambda_n)$ is associated with the particle configuration that has $m_j (\lambda)$ particles at position $j$, for each non-negative integer $j$. Stated alternatively, $\lambda$ is the ordered set of positions in the particle configuration. If $\mathcal{E}$ is a directed line ensemble on $\mathbb{Z}_{\ge 0} \times [0, n]$, let $p_n (\mathcal{E}) \in \Sign_n^+$ denote the signature associated with the particle configuration produced from $\mathcal{E}$ at time $n$. 

Then, $\mathbb{P}^{(n)}$ induces a probability measure on $\Sign_n^+$,  defined as follows. 

\begin{definition}

\label{measuresignaturesvertexmodel} 

For any positive integer $n$, let $\textbf{M}_n$ denote the measure on $\Sign_n^+$ (dependent on the parameters $U$, $S$, and $\Xi$) defined by setting $\textbf{M}_n (\lambda) = \mathbb{P}^{(n)} [p_n (\mathcal{E}) = \lambda]$, for each $\lambda \in \Sign_n^+$. 
\end{definition}

If a model has spin $I / 2$, then it exhibits a type of \emph{weak exclusion principle}, which states that at most $I$ particles can occupy a given location (or equivalently that $m_j \big( p(\mathcal{E}) \big) \le I$ for each $j \in p(\mathcal{E}) $). We will be interested in the case when $I = 1$, which yields the stochastic six-vertex model.

\section{The Spin \texorpdfstring{$1/2$}{1/2} Vertex Models}

\label{Spin12}

We will be particularly interested in the stochastic higher spin vertex model in the \emph{spin $1 / 2$} setting, when all of the $s_k$ are equal to $s = q^{-1 / 2}$. Under this degeneration, the model becomes the stochastic six-vertex model. In Section \ref{StochasticSixVertexModel}, we we will explain this degeneration in more detail. Then, in Section \ref{AsymmetricExclusion}, we will discuss the ASEP and state how the stochastic six-vertex model degenerates to it.

\subsection{The Stochastic Six-Vertex Model}

\label{StochasticSixVertexModel}

The goal of this section is twofold. First, we explain how the stochastic six-vertex model can be recovered directly from a degeneration of the stochastic higher spin vertex model. Second, we formally define and give examples of initial data for the stochastic six-vertex model.

\subsubsection{The Stochastic Six-Vertex Model Through Higher Spin Vertex Models}

\label{VertexHigherSpinVertex}

Let us examine what happens to the probabilities \eqref{configurationprobabilities} when $u_k = u$, $\xi_k = 1$, and $s_k = q^{-1 / 2}$ for all $k \ge 0$. Since $\mathbb{P}_n [(i_2, j_2) = (2, 0) | (i_1, j_1) = (1, 1)] = (1 - q s^2) / (1 - su) = 0$, we have that $i_1, i_2 \in \{ 0, 1 \}$ at each vertex. Thus no two paths can share a vertical edge, meaning that $\mathbb{P}$ is supported on the set of six-vertex directed path ensembles defined in Section \ref{StochasticVertex}. 

Now let us re-parameterize. Let $\delta_1, \delta_2 \in (0, 1)$ satisfy $\delta_1 < \delta_2$, and define $\kappa = (1 - \delta_1) / (1 - \delta_2) > 1$. Denote
\begin{flalign*}
q = \displaystyle\frac{\delta_1}{\delta_2} < 1; \qquad u = \kappa s = \displaystyle\frac{1 - \delta_1}{1 - \delta_2} \sqrt{\displaystyle\frac{\delta_2}{\delta_1}} > 1.
\end{flalign*} 

\noindent Substituting this specialization into the other probabilities in \eqref{configurationprobabilities}, we find that
\begin{flalign}
\begin{aligned} 
\label{sixvertexprobabilities} 
& \qquad \mathbb{P}_n \big[ (i_2, j_2) = (0, 0) \big| (i_1, j_1) = (0, 0) \big] = 1 =  \mathbb{P}_n \big[ (i_2, j_2) = (1, 1) \big| (i_1, j_1) = (1, 1) \big],  \\
& \mathbb{P}_n \big[ (i_2, j_2) = (1, 0) \big| (i_1, j_1) = (1, 0) \big] = \delta_1; \qquad \mathbb{P}_n \big[ (i_2, j_2) = (0, 1) \big| (i_1, j_1) = (1, 0) \big] = 1 - \delta_1,  \\
& \mathbb{P}_n \big[ (i_2, j_2) = (0, 1) \big| (i_1, j_1) = (0, 1) \big] = \delta_2; \qquad \mathbb{P}_n \big[ (i_2, j_2) = (1, 0) \big| (i_1, j_1) = (0, 1) \big] = 1 - \delta_2. 
\end{aligned}
\end{flalign}

\noindent Since the probabilities \eqref{sixvertexprobabilities} coincide with the probabilities depicted in Figure \ref{sixvertexfigure}, it follows that the measure $\mathbb{P}$ of Section \ref{ProbabilityMeasures} (under the above specialization) is equal to the stochastic six-vertex measure $\mathcal{P} (\delta_1, \delta_2)$ defined in Section \ref{StochasticVertex}.

\subsubsection{Initial Data for the Stochastic Six-Vertex Model}

\label{InitialDataVertex}

We start with the following definition. 

\begin{definition}

\label{initialxy} 

Let $\varphi = \big( \varphi (1), \varphi (2), \ldots \big)$, where $\varphi (k) = \big( \varphi_k^{(x)}, \varphi_k^{(y)} \big) \in \mathbb{Z}_{\ge 0} \times \{ 0, 1 \}$ for each $k$, be a stochastic process. The stochastic higher spin vertex model with \emph{initial data $\varphi$} is the stochastic higher spin vertex model (as defined in Section \ref{ProbabilityMeasures} that evolves according to the transition probabilities given by \eqref{configurationprobabilities}) whose boundary condition is defined by setting $j_1 = \varphi_k^{(y)}$ at $(1, k)$ and $i_1 = \varphi_k^{(x)}$ at $(k, 1)$, for each $k$. 
\end{definition}

\begin{exa}

\label{doublesided}

\emph{Step Bernoulli initial data} with \emph{parameter} $b$ arises when the $\varphi (i)$ are all mutually independent, $\varphi_i^{(x)} \equiv 0$, and the $\varphi_i^{(y)}$ are independent $0-1$ Bernoulli random variables with mean $b$. 

Under this initial data, paths can only enter from the $y$-axis, and each vertex on the positive $y$-axis is an entrance point for a path with probability $b$. The case of step initial data is recovered when $b = 1$. 
\end{exa} 

\noindent In the example above, the $\varphi (i)$ are mutually independent, but they need not be in general. For instance, this is the case for the following example, which introduces a class of initial data that generalizes step Bernoulli initial data.

\begin{definition}

\label{generalized}

If $m$ is a positive integer and $b_1, b_2, \ldots , b_m \in [0, 1]$ are positive real numbers, then \emph{step $(b_1, b_2, \ldots , b_m)$-Bernoulli initial data} is defined as follows. First, we set $\varphi_i^{(x)} \equiv 0$, so that no paths enter through the $x$-axis. 

Next, to define $\varphi_i^{(y)}$, consider the stochastic higher spin vertex model (with step initial data), restricted to the strip $[0, m + 1] \times [0, \infty)$, in which the transition probabilities given by \eqref{configurationprobabilities} are replaced with the probabilities 
\begin{flalign}
\begin{aligned}
\label{generalizedprobabilities} 
& \mathbb{P}_n \big[ (i_2, j_2) = (k, 0) \big| (i_1, j_1) = (k, 0) \big] = 1 - (1 - q^k) b_x, \\
& \mathbb{P}_n \big[ (i_2, j_2) = (k - 1, 1) \big| (i_1, j_1) = (k, 0) \big] = (1 - q^k) b_x, \\
& \mathbb{P}_n \big[ (i_2, j_2) = (k + 1, 0) \big| (i_1, j_1) = (k, 1) \big] = 1 - b_x,  \\
& \mathbb{P}_n \big[ (i_2, j_2) = (k, 1) \big| (i_1, j_1) = (k, 1) \big] = b_x,  
\end{aligned} 
\end{flalign}

\noindent at the vertex $(x, y)$ for all integers $x \in [1, m]$ and $y \in [1, \infty)$. Then, define the random variable $\varphi_i^{(y)}$ to be the value of $j_2$ at the vertex $(m, i)$, for each positive integer $i$. 

\end{definition}

\begin{rem}

\label{m1generalized} 

Observe that this coincides with step $b_1$-Bernoulli initial data when $m = 1$. Indeed, in that case, $j_1 = 1$ deterministically at each vertex on the positive $y$-axis due to the step initial data. Therefore, the top two probabilities in \eqref{generalizedprobabilities} have no effect, and the latter two probabilities imply that $j_2$ is a Bernoulli $0-1$ random variable with mean $b_1$ at each vertex on the positive $y$-axis. 
\end{rem}

\begin{rem}

\label{columnsgeneralizedinitial} 

Running a stochastic six-vertex model with this initial data can be viewed as running a stochastic higher spin vertex model with step initial data, whose transition probabilities in the first $m$ columns match those given by \eqref{generalizedprobabilities}, and whose probabilities in all other columns match those of \eqref{sixvertexprobabilities}. This is depicted on the left side of Figure \ref{shift2generalizedvertex} in the case $m = 2$. 

On the right side of Figure \ref{shift2generalizedvertex} is the same model ``shifted'' to the left by $2$ with all arrows originally in the first two columns removed; this is how a stochastic six-vertex model with generalized step Bernoulli initial data might look. 

\end{rem}

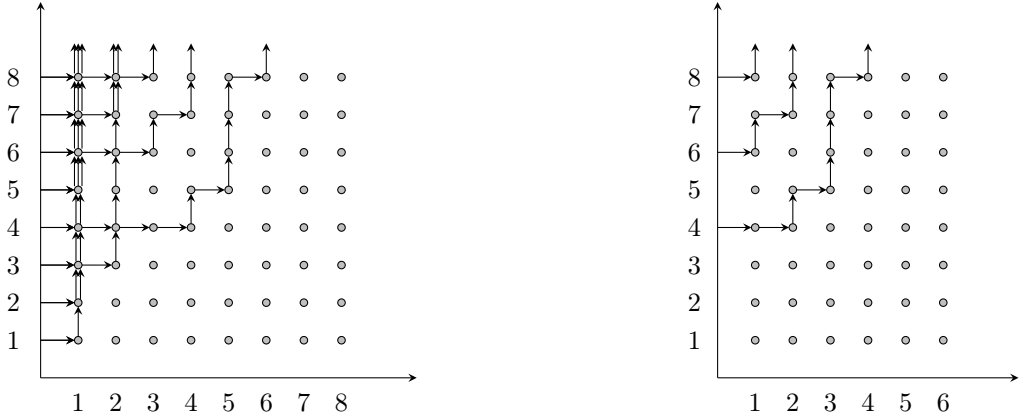
\begin{figure} 

\begin{center}

\begin{tikzpicture}[
      >=stealth,
			scale = 1
			]

			\draw[->, black	] (0, 0) -- (5, 0);
			\draw[->, black] (0, 0) -- (0, 5);
			
			\draw[->, black] (0, .5) -- (.45, .5) node[black, left=17] {1};
			\draw[->, black] (0, 1) -- (.45, 1) node[black, left=17] {2};
			\draw[->, black] (0, 1.5) -- (.45, 1.5) node[black, left=17] {3};
			\draw[->, black] (0, 2) -- (.45, 2) node[black, left=17] {4};
			\draw[->, black] (0, 2.5) -- (.45, 2.5) node[black, left=17] {5};
			\draw[->, black] (0, 3) -- (.45, 3) node[black, left=17] {6};
			\draw[->, black] (0, 3.5) -- (.45, 3.5) node[black, left=17] {7};
			\draw[->, black] (0, 4) -- (.45, 4) node[black, left=17] {8};

			\draw[->, black] (.55, 1.5) -- (.95, 1.5);
			\draw[->, black] (.55, 2) -- (.95, 2);
			\draw[->, black] (.55, 3) -- (.95, 3);
			\draw[->, black] (.55, 3.5) -- (.95, 3.5);
			\draw[->, black] (.55, 4) -- (.95, 4);

			\draw[->, black] (1.05, 2) -- (1.45, 2);
			\draw[->, black] (1.05, 3) -- (1.45, 3);
			\draw[->, black] (1.05, 4) -- (1.45, 4);
			
			\draw[->, black] (1.55, 2) -- (1.95, 2);
			\draw[->, black] (1.55, 3.5) -- (1.95, 3.5);
			
			\draw[->, black] (2.05, 2.5) -- (2.45, 2.5);
			
			\draw[->, black] (2.55, 4) -- (2.95, 4);
			
			\draw[->, black] (0, .5) -- (.45, .5);
			\draw[->, black] (0, 1) -- (.45, 1);
			\draw[->, black] (0, 1.5) -- (.45, 1.5);
			\draw[->, black] (0, 2) -- (.45, 2);
			\draw[->, black] (0, 2.5) -- (.45, 2.5);
			\draw[->, black] (0, 3) -- (.45, 3);
			\draw[->, black] (0, 3.5) -- (.45, 3.5);
			\draw[->, black] (0, 4) -- (.45, 4);

			\draw[->, black] (.5, .55) -- (.5, .95);
			\draw[->, black] (.47, 1.05) -- (.47, 1.45);
			\draw[->, black] (.53, 1.05) -- (.53, 1.45);
			\draw[->, black] (.47, 1.55) -- (.47, 1.95);
			\draw[->, black] (.53, 1.55) -- (.53, 1.95);
			\draw[->, black] (.47, 2.05) -- (.47, 2.45);
			\draw[->, black] (.53, 2.05) -- (.53, 2.45);
			\draw[->, black] (.55, 2.55) -- (.55, 2.95);
			\draw[->, black] (.45, 2.55) -- (.45, 2.95);
			\draw[->, black] (.5, 2.55) -- (.5, 2.95);
			\draw[->, black] (.55, 3.05 ) -- (.55, 3.45);
			\draw[->, black] (.45, 3.05) -- (.45, 3.45);
			\draw[->, black] (.5, 3.05) -- (.5, 3.45);
			\draw[->, black] (.55, 3.55) -- (.55, 3.95);
			\draw[->, black] (.45, 3.55) -- (.45, 3.95);
			\draw[->, black] (.5, 3.55) -- (.5, 3.95);
			\draw[->, black] (.55, 4.05 ) -- (.55, 4.45);
			\draw[->, black] (.45, 4.05) -- (.45, 4.45);
			\draw[->, black] (.5, 4.05) -- (.5, 4.45);

			\draw[->, black] (1, 1.55) -- (1, 1.95);
			\draw[->, black] (1, 2.05) -- (1, 2.45);
			\draw[->, black] (1, 2.55) -- (1, 2.95);
			\draw[->, black] (1, 3.05 ) -- (1, 3.45);
			\draw[->, black] (.97, 3.55) -- (.97, 3.95);
			\draw[->, black] (1.03, 3.55) -- (1.03, 3.95);
			\draw[->, black] (.97, 4.05) -- (.97, 4.45);
			\draw[->, black] (1.03, 4.05) -- (1.03, 4.45);

			\draw[->, black] (1.5, 3.05 ) -- (1.5, 3.45);
			\draw[->, black] (1.5, 4.05) -- (1.5, 4.45); 
			
			\draw[->, black] (2, 2.05 ) -- (2, 2.45);
			\draw[->, black] (2, 3.55) -- (2, 3.95);
			\draw[->, black] (2, 4.05) -- (2, 4.45);

			\draw[->, black] (2.5, 2.55 ) -- (2.5, 2.95);
			\draw[->, black] (2.5, 3.05) -- (2.5, 3.45);
			\draw[->, black] (2.5, 3.55) -- (2.5, 3.95);
			
			\draw[->, black] (3, 4.05) -- (3, 4.45);

			\filldraw[fill=gray!50!white, draw=black] (.5, .5) circle [radius=.05] node[black, below = 17] {1};
			\filldraw[fill=gray!50!white, draw=black] (.5, 1) circle [radius=.05];
			\filldraw[fill=gray!50!white, draw=black] (.5, 1.5) circle [radius=.05];
			\filldraw[fill=gray!50!white, draw=black] (.5, 2) circle [radius=.05];
			\filldraw[fill=gray!50!white, draw=black] (.5, 2.5) circle [radius=.05];
			\filldraw[fill=gray!50!white, draw=black] (.5, 3) circle [radius=.05];
			\filldraw[fill=gray!50!white, draw=black] (.5, 3.5) circle [radius=.05];
			\filldraw[fill=gray!50!white, draw=black] (.5, 4) circle [radius=.05];

			\filldraw[fill=gray!50!white, draw=black] (1, .5) circle [radius=.05] node[black, below = 17] {2};
			\filldraw[fill=gray!50!white, draw=black] (1, 1) circle [radius=.05];
			\filldraw[fill=gray!50!white, draw=black] (1, 1.5) circle [radius=.05];
			\filldraw[fill=gray!50!white, draw=black] (1, 2) circle [radius=.05];
			\filldraw[fill=gray!50!white, draw=black] (1, 2.5) circle [radius=.05];
			\filldraw[fill=gray!50!white, draw=black] (1, 3) circle [radius=.05];
			\filldraw[fill=gray!50!white, draw=black] (1, 3.5) circle [radius=.05];
			\filldraw[fill=gray!50!white, draw=black] (1, 4) circle [radius=.05];
			
			\filldraw[fill=gray!50!white, draw=black] (1.5, .5) circle [radius=.05] node[black, below = 17] {3};
			\filldraw[fill=gray!50!white, draw=black] (1.5, 1) circle [radius=.05];
			\filldraw[fill=gray!50!white, draw=black] (1.5, 1.5) circle [radius=.05];
			\filldraw[fill=gray!50!white, draw=black] (1.5, 2) circle [radius=.05];
			\filldraw[fill=gray!50!white, draw=black] (1.5, 2.5) circle [radius=.05];
			\filldraw[fill=gray!50!white, draw=black] (1.5, 3) circle [radius=.05];
			\filldraw[fill=gray!50!white, draw=black] (1.5, 3.5) circle [radius=.05];
			\filldraw[fill=gray!50!white, draw=black] (1.5, 4) circle [radius=.05];

			\filldraw[fill=gray!50!white, draw=black] (2, .5) circle [radius=.05] node[black, below = 17] {4};
			\filldraw[fill=gray!50!white, draw=black] (2, 1) circle [radius=.05];
			\filldraw[fill=gray!50!white, draw=black] (2, 1.5) circle [radius=.05];
			\filldraw[fill=gray!50!white, draw=black] (2, 2) circle [radius=.05];
			\filldraw[fill=gray!50!white, draw=black] (2, 2.5) circle [radius=.05];
			\filldraw[fill=gray!50!white, draw=black] (2, 3) circle [radius=.05];
			\filldraw[fill=gray!50!white, draw=black] (2, 3.5) circle [radius=.05];
			\filldraw[fill=gray!50!white, draw=black] (2, 4) circle [radius=.05];
			
			\filldraw[fill=gray!50!white, draw=black] (2.5, .5) circle [radius=.05] node[black, below = 17] {5};
			\filldraw[fill=gray!50!white, draw=black] (2.5, 1) circle [radius=.05];
			\filldraw[fill=gray!50!white, draw=black] (2.5, 1.5) circle [radius=.05];
			\filldraw[fill=gray!50!white, draw=black] (2.5, 2) circle [radius=.05];
			\filldraw[fill=gray!50!white, draw=black] (2.5, 2.5) circle [radius=.05];
			\filldraw[fill=gray!50!white, draw=black] (2.5, 3) circle [radius=.05];
			\filldraw[fill=gray!50!white, draw=black] (2.5, 3.5) circle [radius=.05];
			\filldraw[fill=gray!50!white, draw=black] (2.5, 4) circle [radius=.05];

			\filldraw[fill=gray!50!white, draw=black] (3, .5) circle [radius=.05] node[black, below = 17] {6};
			\filldraw[fill=gray!50!white, draw=black] (3, 1) circle [radius=.05];
			\filldraw[fill=gray!50!white, draw=black] (3, 1.5) circle [radius=.05];
			\filldraw[fill=gray!50!white, draw=black] (3, 2) circle [radius=.05];
			\filldraw[fill=gray!50!white, draw=black] (3, 2.5) circle [radius=.05];
			\filldraw[fill=gray!50!white, draw=black] (3, 3) circle [radius=.05];
			\filldraw[fill=gray!50!white, draw=black] (3, 3.5) circle [radius=.05];
			\filldraw[fill=gray!50!white, draw=black] (3, 4) circle [radius=.05];
			
			\filldraw[fill=gray!50!white, draw=black] (3.5, .5) circle [radius=.05] node[black, below = 17] {7};
			\filldraw[fill=gray!50!white, draw=black] (3.5, 1) circle [radius=.05];
			\filldraw[fill=gray!50!white, draw=black] (3.5, 1.5) circle [radius=.05];
			\filldraw[fill=gray!50!white, draw=black] (3.5, 2) circle [radius=.05];
			\filldraw[fill=gray!50!white, draw=black] (3.5, 2.5) circle [radius=.05];
			\filldraw[fill=gray!50!white, draw=black] (3.5, 3) circle [radius=.05];
			\filldraw[fill=gray!50!white, draw=black] (3.5, 3.5) circle [radius=.05];
			\filldraw[fill=gray!50!white, draw=black] (3.5, 4) circle [radius=.05];
			
			\filldraw[fill=gray!50!white, draw=black] (4, .5) circle [radius=.05] node[black, below = 17] {8};
			\filldraw[fill=gray!50!white, draw=black] (4, 1) circle [radius=.05];
			\filldraw[fill=gray!50!white, draw=black] (4, 1.5) circle [radius=.05];
			\filldraw[fill=gray!50!white, draw=black] (4, 2) circle [radius=.05];
			\filldraw[fill=gray!50!white, draw=black] (4, 2.5) circle [radius=.05];
			\filldraw[fill=gray!50!white, draw=black] (4, 3) circle [radius=.05];
			\filldraw[fill=gray!50!white, draw=black] (4, 3.5) circle [radius=.05];
			\filldraw[fill=gray!50!white, draw=black] (4, 4) circle [radius=.05];

			\draw[->, black	] (9, 0) -- (13, 0);
			\draw[->, black] (9, 0) -- (9, 5);

			\draw[->, black] (9, 2) -- (9.45, 2);
			\draw[->, black] (9, 3) -- (9.45, 3);
			\draw[->, black] (9, 4) -- (9.45, 4);
			
			\draw[->, black] (9.55, 2) -- (9.95, 2);
			\draw[->, black] (9.55, 3.5) -- (9.95, 3.5);
			
			\draw[->, black] (10.05, 2.5) -- (10.45, 2.5);
			
			\draw[->, black] (10.55, 4) -- (10.95, 4);

			\draw[->, black] (9.5, 3.05 ) -- (9.5, 3.45);
			\draw[->, black] (9.5, 4.05) -- (9.5, 4.45);
			
			\draw[->, black] (10, 2.05 ) -- (10, 2.45);
			\draw[->, black] (10, 3.55) -- (10, 3.95);
			\draw[->, black] (10, 4.05) -- (10, 4.45);

			\draw[->, black] (10.5, 2.55 ) -- (10.5, 2.95);
			\draw[->, black] (10.5, 3.05) -- (10.5, 3.45);
			\draw[->, black] (10.5, 3.55) -- (10.5, 3.95);
			
			\draw[->, black] (11, 4.05) -- (11, 4.45);

			\filldraw[fill=gray!50!white, draw=black] (9.5, .5) circle [radius=.05] node[black, below = 17] {1}  node[black, left = 17] {1};
			\filldraw[fill=gray!50!white, draw=black] (9.5, 1) circle [radius=.05] node[black, left = 17] {2};
			\filldraw[fill=gray!50!white, draw=black] (9.5, 1.5) circle [radius=.05] node[black, left = 17] {3};
			\filldraw[fill=gray!50!white, draw=black] (9.5, 2) circle [radius=.05] node[black, left = 17] {4};
			\filldraw[fill=gray!50!white, draw=black] (9.5, 2.5) circle [radius=.05] node[black, left = 17] {5};
			\filldraw[fill=gray!50!white, draw=black] (9.5, 3) circle [radius=.05] node[black, left = 17] {6};
			\filldraw[fill=gray!50!white, draw=black] (9.5, 3.5) circle [radius=.05] node[black, left = 17] {7};
			\filldraw[fill=gray!50!white, draw=black] (9.5, 4) circle [radius=.05] node[black, left = 17] {8};

			\filldraw[fill=gray!50!white, draw=black] (10, .5) circle [radius=.05] node[black, below = 17] {2};
			\filldraw[fill=gray!50!white, draw=black] (10, 1) circle [radius=.05];
			\filldraw[fill=gray!50!white, draw=black] (10, 1.5) circle [radius=.05];
			\filldraw[fill=gray!50!white, draw=black] (10, 2) circle [radius=.05];
			\filldraw[fill=gray!50!white, draw=black] (10, 2.5) circle [radius=.05];
			\filldraw[fill=gray!50!white, draw=black] (10, 3) circle [radius=.05];
			\filldraw[fill=gray!50!white, draw=black] (10, 3.5) circle [radius=.05];
			\filldraw[fill=gray!50!white, draw=black] (10, 4) circle [radius=.05];
			
			\filldraw[fill=gray!50!white, draw=black] (10.5, .5) circle [radius=.05] node[black, below = 17] {3};
			\filldraw[fill=gray!50!white, draw=black] (10.5, 1) circle [radius=.05];
			\filldraw[fill=gray!50!white, draw=black] (10.5, 1.5) circle [radius=.05];
			\filldraw[fill=gray!50!white, draw=black] (10.5, 2) circle [radius=.05];
			\filldraw[fill=gray!50!white, draw=black] (10.5, 2.5) circle [radius=.05];
			\filldraw[fill=gray!50!white, draw=black] (10.5, 3) circle [radius=.05];
			\filldraw[fill=gray!50!white, draw=black] (10.5, 3.5) circle [radius=.05];
			\filldraw[fill=gray!50!white, draw=black] (10.5, 4) circle [radius=.05];
			
			\filldraw[fill=gray!50!white, draw=black] (11, .5) circle [radius=.05] node[black, below = 17] {4};
			\filldraw[fill=gray!50!white, draw=black] (11, 1) circle [radius=.05];
			\filldraw[fill=gray!50!white, draw=black] (11, 1.5) circle [radius=.05];
			\filldraw[fill=gray!50!white, draw=black] (11, 2) circle [radius=.05];
			\filldraw[fill=gray!50!white, draw=black] (11, 2.5) circle [radius=.05];
			\filldraw[fill=gray!50!white, draw=black] (11, 3) circle [radius=.05];
			\filldraw[fill=gray!50!white, draw=black] (11, 3.5) circle [radius=.05];
			\filldraw[fill=gray!50!white, draw=black] (11, 4) circle [radius=.05];
			
			\filldraw[fill=gray!50!white, draw=black] (11.5, .5) circle [radius=.05] node[black, below = 17] {5};
			\filldraw[fill=gray!50!white, draw=black] (11.5, 1) circle [radius=.05];
			\filldraw[fill=gray!50!white, draw=black] (11.5, 1.5) circle [radius=.05];
			\filldraw[fill=gray!50!white, draw=black] (11.5, 2) circle [radius=.05];
			\filldraw[fill=gray!50!white, draw=black] (11.5, 2.5) circle [radius=.05];
			\filldraw[fill=gray!50!white, draw=black] (11.5, 3) circle [radius=.05];
			\filldraw[fill=gray!50!white, draw=black] (11.5, 3.5) circle [radius=.05];
			\filldraw[fill=gray!50!white, draw=black] (11.5, 4) circle [radius=.05];
			
			\filldraw[fill=gray!50!white, draw=black] (12, .5) circle [radius=.05] node[black, below = 17] {6};
			\filldraw[fill=gray!50!white, draw=black] (12, 1) circle [radius=.05];
			\filldraw[fill=gray!50!white, draw=black] (12, 1.5) circle [radius=.05];
			\filldraw[fill=gray!50!white, draw=black] (12, 2) circle [radius=.05];
			\filldraw[fill=gray!50!white, draw=black] (12, 2.5) circle [radius=.05];
			\filldraw[fill=gray!50!white, draw=black] (12, 3) circle [radius=.05];
			\filldraw[fill=gray!50!white, draw=black] (12, 3.5) circle [radius=.05];
			\filldraw[fill=gray!50!white, draw=black] (12, 4) circle [radius=.05];

\end{tikzpicture}

\end{center}

\caption{\label{shift2generalizedvertex} Shown above and to the left is the stochastic higher spin vertex model specialized as in Example \ref{generalized}, with $m = 2$. Shown to the right is this model shifted to the left by $2$, yielding a stochastic six-vertex model with generalized step Bernoulli initial data.  }

\end{figure}

\subsection{The ASEP}

\label{AsymmetricExclusion}

In this section we discuss the asymmetric simple exclusion process (ASEP). Our goals are again two-fold. First, we formally define initial data for the ASEP; then, we explain a limit degeneration that maps the stochastic six-vertex model to the ASEP. The latter point will be particularly useful to us later, in Section \ref{DeterminantVertical}, where it will allow us to produce identities for the ASEP from identities for the stochastic six-vertex model.

\subsubsection{Initial Data for the ASEP}

\label{InitialDataProcess}

 Recall from Section \ref{AsymmetricExclusions} that the ASEP is a continuous time Markov process $\{ \eta_t (i) \}_{i \in \mathbb{Z}, t \in \mathbb{R}_{\ge 0}} \subset \{ 0, 1 \}^{\mathbb{Z}} \times \mathbb{R}_{\ge 0}$, whose dynamics can be defined as follows. 

For each pair of consecutive integers $(i, i + 1) \in \mathbb{Z}^2$, the variables $\eta_t (i)$ and $\eta_t (i + 1)$ are interchanged (meaning we set $\eta_t (i) = \eta_{t^-} (i + 1)$ and $\eta_t (i + 1) = \eta_{t^-} (i)$, but leave all other $\eta_t (j)$ fixed) at exponential rate $L \ge 0$ if $\big( \eta_{t^-} (i), \eta_{t^-} (i + 1) \big) = (0, 1)$; similarly, the variables $\eta_t (i)$ and $\eta_t (i + 1)$ are interchanged at exponential rate $R \ge 0$ if $\big( \eta_{t^-} (i), \eta_{t^-} (i + 1) \big) = (1, 0)$. The exponential clocks corresponding to each pair $(i, i + 1)$ are mutually independent. Here, $t^-$ refers to the time infinitesimally before time $t$. 

The random variables $\eta_t (i)$ can be viewed as indicators for the event that a particle is at position $i$ at time $t$. Therefore, the dynamics at ASEP are equivalent to particles independently attempting to jump left at rate $L$ and right at rate $R$, subject to the exclusion restriction. Specifically, if the destination of the jump is unoccupied, the jump is performed; otherwise, it is not. Throughout, we will assume that $R > L$, so that particles drift to the right on average. 

Now we can define initial data for the ASEP. 

\begin{definition}

\label{initialprocess}

Let $\varphi = (\varphi (1), \varphi (2), \ldots )$ be a stochastic process, where $\varphi (i) = \big( \varphi_i^{(x)}, \varphi_i^{(y)} \big) \in \{ 0, 1 \} \times \{ 0, 1 \}$. Define the \emph{ASEP with initial data $\varphi$} to be the ASEP with $\eta_0 (i) = \varphi_i^{(x)}$ if $i$ is positive, and $\eta_0 (i) = \varphi_{1 - i}^{(y)}$ if $i$ is non-positive. 

\end{definition}

\begin{exa}
If $b \in [0, 1]$, then \emph{step $b$-Bernoulli initial data} arises when the $\varphi (i)$ are mutually independent, $\varphi_i^{(x)} \equiv 0$, and the $\varphi_i^{(y)}$ are independent $0-1$ Bernoulli random variables with mean $b$, for each $i$. That is, particles are initially placed at or to the left of $0$ with probability $b$, and no particles are placed to the right of $0$; all placements are independent. 
\end{exa}

\subsubsection{Degeneration of the Stochastic Six-Vertex Model to the ASEP}

\label{CurrentConverge} 

In this section we explain a certain limit degeneration under which the stochastic six-vertex model converges to the ASEP. 

To briefly provide a heuristic as to why such a convergence should exist, consider the stochastic six-vertex model $\mathcal{P} (\delta_1, \delta_2)$ in which $\delta_1 = \delta_2 = 0$ (see Section \ref{VertexHigherSpinVertex} for definitions). In this case, all paths will almost surely alternate between turning one space up and one space right. In terms of the interacting particle system interpretation discussed in Section \ref{InteractingParticles}, each particle almost surely jumps one space to the right at each time step; thus, if we translate each particle $t$ spaces to the left at time $t$ (we refer to this as \emph{offsetting by the diagonal}), then the particle system will not evolve over time. 

Now assume that $\delta_1 = \varepsilon L$ and $\delta_2 = \varepsilon R$, for some positive real numbers $R, L > 0$ and some very small number $\varepsilon > 0$. Then, paths will ``usually'' turn up and right; thus, after offsetting by the diagonal, particles will usually not move. However, very rarely, an offset particle will jump to the left (with probability $\varepsilon L$) and to the right (with probability $\varepsilon R$), subject to the restriction that a particle cannot jump to an occupied location. After rescaling time by $\varepsilon^{-1}$, and taking the limit as $\varepsilon$ tends to $0$, this coincides with the dynamics of the ASEP. 

Thus, after offsetting by the diagonal, one expects that the stochastic six-vertex model $\mathcal{P} (\varepsilon L, \varepsilon R)$ should converge to the ASEP. This heuristic was provided in Section 2.2 of \cite{SSVM} (see also Section 6.5 of \cite{HSVMRSF}) and later established in Theorem 3 and Corollary 4 of \cite{CSSVMEP}. 

The following proposition, which appears as Corollary 4 of \cite{CSSVMEP}, states this heuristic precisely in terms of convergence of currents. In what follows, we recall the definitions of the currents $J$ of the ASEP and $\mathfrak{H}$ of the stochastic six-vertex model from Section \ref{AsymmetricExclusions} and Section \ref{StochasticVertex}. 

\begin{prop}[{\cite[Corollary 4]{CSSVMEP}}]

\label{currentmodelprocess}

Fix real numbers $R, L \ge 0$ and a stochastic process $\varphi = \big( \varphi (1), \varphi (2), \ldots \big)$, with $\varphi (i) = \big( \varphi_i^{(x)}, \varphi_i^{(y)} \big) \in \{ 0, 1 \} \times \{ 0, 1 \}$. Let $\varepsilon > 0$ be a real number, and denote $\delta_1 = \delta_{1; \varepsilon} = \varepsilon L$ and $\delta_2 = \delta_{2; \varepsilon} = \varepsilon R$; assume that $\delta_1, \delta_2 \in (0, 1)$. Also fix $r \in \mathbb{R}$, $t \in \mathbb{R}_{> 0}$, and $x \in \mathbb{Z}$.

Let $p_{\varepsilon} (x; r) = \mathbb{P}_V \big[ \mathfrak{H} \big( x + \lfloor \varepsilon^{-1} t \rfloor, \lfloor \varepsilon^{-1} t \rfloor \big) \ge r \big]$, where the probability $\mathbb{P}_V$ is under the stochastic six-vertex model $\mathcal{P} (\delta_1, \delta_2)$ with initial data $\varphi$. Furthermore, let $p (x; r) = \mathbb{P}_A \big[ J_t \big( x \big) \ge r \big]$, where the probability $\mathbb{P}_A$ is under the ASEP with left jump rate $L$, right jump rate $R$, and initial data $\varphi$. 

Then, $\lim_{\varepsilon \rightarrow 0} p_{\varepsilon} (x; r) = p (x; r)$. 

\end{prop}

\section{Observables for Models With Generalized Step Bernoulli Initial Data}

\label{ObservablesVertical}

We now turn to the analysis of the ASEP and stochastic six-vertex model with generalized step Bernoulli initial data. Our results in this section are primarily algebraic; our goal is to establish Theorem \ref{hdeterminant} and Theorem \ref{hdeterminantprocess}, which are Fredholm determinant identities for the $q$-Laplace transform of the current of the stochastic six-vertex model and ASEP, respectively, with generalized step Bernoulli initial data. These identities will be amenable to asymptotic analysis and will therefore enable the proofs of the phase transition results Theorem \ref{asymmetriclimit} and Theorem \ref{hlimit} later in the paper. 

The proofs of these Fredholm determinant identities rely on the integrability of the stochastic six-vertex model under generalized step Bernoulli initial data. In particular, we will first show how to degenerate the stochastic higher spin vertex model with step initial data to the stochastic six-vertex model with generalized step Bernoulli initial data. Using the results of \cite{HSVMRSF}, this then yields contour integral identities for $q$-moments of the current of this model, which will give rise to the Fredholm determinant identities Theorem \ref{hdeterminant} and Theorem \ref{hdeterminantprocess}.

\subsection{Contour Integral Identities for \texorpdfstring{$q$}{q}-Moments of the Current}

\label{ContourCurrentVertical} 

Our goal in this section is to establish Proposition \ref{manydeformationexpectationn}, which is a multi-fold contour integral identity for $q$-moments of the current of the stochastic six-vertex model with step Bernoulli initial data. 

Obtaining such identities for the stochastic higher spin vertex models with step initial data has been investigated at length in \cite{HSVMRSF}; see, for example, Theorem 9.8 of that article. We will show that, under a certain specialization, the stochastic higher spin vertex model with step initial data degenerates to the stochastic six-vertex model with different initial data, namely, generalized step Bernoulli initial data. Then, it will be possible to degenerate the results of \cite{HSVMRSF} to obtain the contour integral identities relevant for the proofs of Theorem \ref{hdeterminant} and Theorem \ref{hdeterminantprocess}.

\subsubsection{A Specialization of the Stochastic Higher Spin Vertex Model}

\label{GeneralizedSpecialization}

Consider the stochastic higher spin vertex model as defined in Section \ref{ProbabilityMeasures}, with the following specialization of parameters. Let $m \ge 1$ be some positive integer and $b_1, b_2, \ldots , b_m \in (0, 1)$ be positive real numbers. Let $\delta_1 < \delta_2$ be positive real numbers in the interval $(0, 1)$, and let 
\begin{flalign}
\label{stochasticparameters}
q = \displaystyle\frac{\delta_1}{\delta_2} < 1; \qquad \kappa = \displaystyle\frac{1 - \delta_1}{1 - \delta_2} > 1; \qquad s = q^{-1 / 2}; \qquad u = \kappa s; \qquad \beta_i = \displaystyle\frac{b_i}{1 - b_i},
\end{flalign}

\noindent for each $i \in [1, m]$, as in Section \ref{StochasticSixVertexModel}. Now, set $u_j = u$ for each $j \ge 1$; $s_j = s$ and $\xi_j = 1$ for each $j \ge m + 1$; $s_i = a_i$ for each $i \in [1, m]$; and $\xi_i = - \beta_i / a_i u$ for each $i \in [1, m]$. 

For $x > m$, the probabilities in \eqref{configurationprobabilities} match with those of the stochastic six-vertex model given by \eqref{sixvertexprobabilities}. If $x \in [1, m]$, then the probabilities become 
\begin{flalign}
\begin{aligned}
\label{verticalconfigurationprobabilities} 
& \mathbb{P}_n \big[ (i_2, j_2) = (k, 0) \big| (i_1, j_1) = (k, 0) \big] = \displaystyle\frac{1 + q^k \beta_x}{1 + \beta_x},  \\
& \mathbb{P}_n \big[ (i_2, j_2) = (k - 1, 1) \big| (i_1, j_1) = (k, 0) \big] = \displaystyle\frac{(1 - q^k) \beta_x}{1 + \beta_x},  \\
& \mathbb{P}_n \big[ (i_2, j_2) = (k + 1, 0) \big| (i_1, j_1) = (k, 1) \big] = \displaystyle\frac{1 - q^k a_x^2}{1 + \beta_x},  \\
& \mathbb{P}_n \big[ (i_2, j_2) = (k, 1) \big| (i_1, j_1) = (k, 1) \big] = \displaystyle\frac{q^k a_x^2 + \beta_x}{1 + \beta_x}.  
\end{aligned}
\end{flalign}

\noindent Taking the limit as $a_i$ tends to $0$ for each $i \in [1, m]$ and recalling that $\beta_i = b_i / (1 - b_i)$ for each $i$ yields the probabilities in \eqref{generalizedprobabilities}. 

In particular, let us define $\varphi_1^{(y)}, \varphi_2^{(y)}, \ldots $ by setting $\varphi_k^{(y)} = 1$ if $j_2 = 1$ at the vertex $(k, m)$ and $\varphi_k^{(y)} = 0$ otherwise; equivalently, $\varphi_k^{(y)}$ is the indicator that there exists some path containing an arrow from $(m, k)$ to $(m + 1, k)$. Then the initial data defined by $\varphi = \big( \varphi (1), \varphi (2), \ldots \big)$, where $\varphi (i) = \big( \varphi_i^{(x)}, \varphi_i^{(y)} \big) = \big( 0, \varphi_i^{(y)} \big)$, coincides with the step $(b_1, b_2, \ldots , b_m)$-Bernoulli initial data given by Definition \ref{generalized}. 

Thus, if we translate the vertex model to the left by $m$ (so that the line $y = m$ is shifted onto the $y$-axis) and then ignore all arrows to the left of the $y$-axis, we obtain the stochastic six-vertex model with step $(b_1, b_2, \ldots , b_m)$-Bernoulli initial data; see Remark \ref{columnsgeneralizedinitial} and Figure \ref{shift2generalizedvertex}.

\subsubsection{Observables for the Stochastic Six-Vertex Model}

\label{ContourMoments}

In this section, we will establish contour integral identities for $q$-moments of the current of the stochastic six-vertex model with generalized step Bernoulli initial data. We will begin with Proposition \ref{firstheightgeneral}, which gives an identity for $q$-moments of the height function of the stochastic higher spin vertex model in a reasonably general setting. However, we first require some notation. 

In what follows, $n$ will be a positive integer; $q \in (0, 1)$ will be a positive real number; $U = (u_1, u_2, \ldots , u_n) \subset \mathbb{C}$ will be a finite set of non-zero complex numbers; and $\Xi = (\xi_1, \xi_2, \ldots ) \subset \mathbb{C}$ and $S = (s_1, s_2, \ldots ) \subset \mathbb{C}$ will be infinite sets of non-zero complex numbers. 

The following definition is a restriction on our parameters so that the contours we consider exist. 

\begin{definition}

\label{spacedparameters}

We call the quadruple $(U; \Xi, S)$ \emph{suitably spaced} if the following two conditions are satisfied. 

\begin{itemize}

\item{The elements of $U$ are real and sufficiently close together so that they all have the same sign and $q \max_{1 \le i \le n} |u_i| < \min_{1 \le i \le n} |u_i|$.}

\item{No number of the form $s_i \xi_i$ is contained in the interval $\big( \min_{1 \le i \le n} u_i^{-1}, \max_{1 \le i \le n} u_i^{-1} \big)$.}

\end{itemize}

\end{definition}

\noindent The following defines the contours we will use in this section. 

\begin{definition}

\label{firstcontours}

Suppose that $(U; \Xi, S)$ is suitably spaced. Let us define the set of $k$ positively oriented contours $\gamma_1 (U; \Xi, S), \gamma_2 (U; \Xi, S), \ldots , \gamma_k (U; \Xi, S)$ as follows. Each contour $\gamma_i = \gamma_i (U; \Xi, S)$ will be the disjoint union of two positively oriented circles $c_i^{(1)}$ and $c_i^{(2)}$. 

Here, $c^{(1)} = c_1^{(1)} = c_2^{(1)} = \cdots = c_k^{(1)}$ are all the same circle that contain each of the $u_i^{-1}$ and leave outside $\xi_1 s_1, \xi_2 s_2, \ldots $. We also assume that $c^{(1)}$ is sufficiently small so that its interior is disjoint with the image of its interior under multiplication by $q$. 

The circles $c_1^{(2)}, c_2^{(2)}, \ldots , c_k^{(2)}$ will be centered at $0$ and sufficiently small such that the following two properties hold. 

\begin{itemize}

\item{The circles are ``nested'' in the sense that, for each integer $i \in [1, k - 1]$, the circle $c_{i + 1}^{(2)}$ strictly contains $q^{-1} c_i^{(2)}$ (defined to be the image of $c_i^{(2)}$ under multiplication by $q^{-1}$).} 

\item{All circles are sufficiently small so that the interior of $c_k^{(2)}$ is disjoint with the interior of $q c^{(1)}$ and so that the interior of $c_k^{(2)}$ does not contain $s_1 \xi_1, s_2 \xi_2, \ldots $. }

\end{itemize}

\end{definition}

The fact that $(U; \Xi, S)$ is suitably spaced ensures the existence of such contours. See Figure \ref{firstcontoursfigure} for an example.

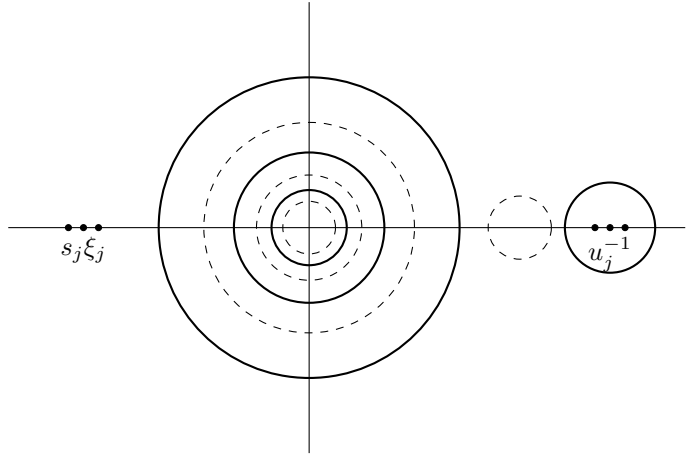
\begin{figure}

\begin{center}

\begin{tikzpicture}[
      >=stealth,
			]

			\draw[-, black] (-4, 0) -- (5, 0); 
			\draw[-, black] (0, -3) -- (0, 3);

			\draw[black, fill] (4, 0) circle [radius=.04] node [black, below=0] {$u_j^{-1}$};
			\draw[black, fill] (4.2, 0) circle [radius=.04];
			\draw[black, fill] (3.8, 0) circle [radius=.04];
			\draw[black, fill] (-3, 0) circle [radius=.04] node [black, below=0] {$s_j \xi_j$};
			\draw[black, fill] (-3.2, 0) circle [radius=.04];
			\draw[black, fill] (-2.8, 0) circle [radius=.04];

			\draw[black, thick] (0, 0) circle [radius=.5]; 
			\draw[black, thick] (0, 0) circle [radius=1]; 
			\draw[black, thick] (0, 0) circle [radius=2];
			
			\draw[black, dashed] (0, 0) circle [radius=.35]; 
			\draw[black, dashed] (0, 0) circle [radius=.7];
			\draw[black, dashed] (0, 0) circle [radius=1.4]; 
			 
			\draw[black, thick] (4, 0) circle [radius=.6];
			\draw[black, dashed] (2.8, 0) circle [radius=.42];

\end{tikzpicture}
	
\end{center}

\caption{\label{firstcontoursfigure} Shown above in solid is an example of possible contours $\gamma_1 (U; \Xi, S)$, $\gamma_2 (U; \Xi, S)$, and $\gamma_3 (U; \Xi, S)$. The dashed curves are the images of the solid curves under multiplication by $q$.}
\end{figure}

Now we can state Proposition 9.5 of \cite{HSVMRSF}. In what follows, $\mathfrak{h}_{\lambda} (x)$ denotes the number of indices $i$ for which $\lambda_i \ge x$, for any $\lambda \in \Sign^+$. 

\begin{prop}[{\cite[Proposition 9.5]{HSVMRSF}}]
\label{firstheightgeneral} 

Let $k$, $x$, and $t$ be positive integers; $q \in (0, 1)$ be a positive real number; $U = (u_1, u_2, \ldots , u_t)$ be a set of $t$ positive real numbers; and $S = (s_0, s_1, s_2, \ldots )$ and $\Xi = (\xi_0, \xi_1, \ldots ) \in (0, \infty)$ be infinite sets of real numbers numbers, such that $(U; \Xi, S)$ is suitably spaced and $|s_j \xi_j u_i - s_j^2| < |1 - s_j \xi_j u_i|$ for each $i, j$	. 

Under the measure $\textbf{\emph{M}} = \textbf{\emph{M}}_t$ given by Definition \ref{measuresignaturesvertexmodel}, with spectral parameters $U$, spin parameters $S = (s_1, s_2, \ldots )$, and spectral inhomogeneity parameters $\Xi = (\xi_1, \xi_2, \ldots )$, we have that 
\begin{flalign}
\label{firstheightgeneralidentity}
\mathbb{E}_{\textbf{\emph{M}}} [q^{k \mathfrak{h}_{\lambda} (x)}] & = \displaystyle\frac{q^{\binom{k}{2}}}{(2 \pi \textbf{\emph{i}})^k} \displaystyle\oint \cdots \displaystyle\oint  \displaystyle\prod_{1 \le i < j \le k} \displaystyle\frac{w_i - w_j}{w_i - q w_j} \displaystyle\prod_{i = 1}^k \left( \displaystyle\prod_{j = 1}^{x - 1} \displaystyle\frac{s_j \xi_j - s_j^2 w_i}{s_j \xi_j - w_i} \displaystyle\prod_{j = 1}^t \displaystyle\frac{1 - q u_j w_i}{1 - u_j w_i} \right) \displaystyle\frac{d w_i}{w_i}, 
\end{flalign}

\noindent where each $w_i$ is integrated along the contour $\gamma_i = \gamma_i (U; \Xi, S)$ of Definition \ref{firstcontours}. 
\end{prop}

In fact, as stated above, Proposition \ref{firstheightgeneral} does not precisely coincide with Proposition 9.5 of \cite{HSVMRSF}; the latter was originally stated with more stringent analytic constraints on the parameters $U$, $S$, and $\Xi$. However, Lemma 9.1 of \cite{HSVMRSF} states that the left side of \eqref{firstheightgeneralidentity} is a rational function in all of these parameters. Thus, Proposition \ref{firstheightgeneral} above follows from Proposition 9.5 of \cite{HSVMRSF} and a standard (see, for instance, Section 10 of \cite{HSVMRSF}) analytic continuation argument, which we will not detail further. 

 Using Proposition \ref{firstheightgeneral}, we can obtain $q$-moments of the current for the stochastic six-vertex model with generalized step Bernoulli initial data; recall Definition \ref{generalized} for a description of this initial data. In what follows, for any $(x, t) \in \mathbb{Z}_{\ge 0}^2$, $\mathfrak{h}_t (x)$ denotes the \emph{height function} of a stochastic higher spin vertex model, that is, the number of integers $k \ge x$ such that $j_2 = 1$ at $(k, t)$; stated alternatively, $\mathfrak{h}_t (x)$ denotes the number of paths that are strictly to the right of $x - 1$ at time $t$. Observe that $\mathfrak{h}_t (x) \le t$ and is thus finite. 

\begin{prop}
\label{manydeformationexpectationn}

Fix $m, t \in \mathbb{Z}_{> 0}$; $\delta_1, \delta_2 \in (0, 1)$; and $b_1, b_2, \ldots , b_m \in (0, 1)$. Consider the stochastic six-vertex model with left jump probability $\delta_1$, right jump probability $\delta_2$, and step $(b_1, b_2, \ldots , b_m)$-Bernoulli initial data. Denoting $q$, $u$, $\kappa$, $s$, and $\beta_i$ as in \eqref{stochasticparameters}, we have that 
\begin{flalign}
\begin{aligned}
\label{heightqgeneralized}
\mathbb{E} [q^{k \mathfrak{h}_t (x)}] = \displaystyle\frac{q^{\binom{k}{2}}}{(2 \pi \textbf{\emph{i}})^k} \displaystyle\oint \cdots & \displaystyle\oint \displaystyle\prod_{1 \le i < j \le k} \displaystyle\frac{y_i - y_j}{y_i - q y_j} \displaystyle\prod_{i = 1}^k  \left( \displaystyle\frac{1 + y_i}{1 + q^{-1} y_i} \right)^t \left( \displaystyle\frac{1 + q^{-1} \kappa^{-1} y}{1 + \kappa^{-1} y_i}  \right)^{x - 1}  \\
& \quad \times \displaystyle\prod_{i = 1}^k \left( \displaystyle\prod_{j = 1}^m \displaystyle\frac{1}{1 - q^{-1} \beta_j^{-1} y_i}  \right) \displaystyle\frac{d y_i}{y_i}, 
\end{aligned}
\end{flalign}

\noindent where each $y_i$ is integrated along the contour $\gamma_i (\widehat{U}; \widehat{\Xi}, \widehat{S})$ of Definition \ref{firstcontours}. Here, we have denoted $\widehat{U} = (-q^{-1}, -q^{-1}, \ldots )$, $\widehat{S} = (1, 1, \ldots )$, and $\widehat{\Xi} = (q \beta_1, q \beta_2, \ldots , q \beta_m, - \kappa, - \kappa, \ldots )$. 

\end{prop}

\begin{proof} 

We will apply Proposition \ref{firstheightgeneral} to the stochastic higher spin vertex model, specialized as in Section \ref{GeneralizedSpecialization}, run for $t$ discrete time steps. That is, set $u_i = u$ for each $i \in [1, t]$; $\xi_i = 1$ and $s_i = s$ if $i > m$; $s_i = a_i \in (-1, 0)$ to be some small negative real number if $i \in [1, m]$; and $\xi_i = -\beta_i / a_i u$ if $i \in [1, m]$. Also set $\xi_0 = 1$ and $s_0 = s$. 

Now, applying Proposition \ref{firstheightgeneral}, we obtain  
\begin{flalign*}
\mathbb{E}_{\textbf{M}} [q^{k \mathfrak{h}_{\lambda} (x + m)}] = \displaystyle\frac{q^{\binom{k}{2}}}{(2 \pi \textbf{i})^k} \displaystyle\oint \cdots & \displaystyle\oint \displaystyle\prod_{1 \le i < j \le k} \displaystyle\frac{w_i - w_j}{w_i - q w_j} \displaystyle\prod_{i = 1}^k \left( \displaystyle\frac{1 - q u w_i}{1 - u w_i} \right)^t \left( \displaystyle\frac{s - s^2 w_i}{s - w_i}  \right)^{x - 1} \\
& \quad \times \displaystyle\prod_{i = 1}^k \left( \displaystyle\prod_{j = 1}^m \displaystyle\frac{\beta_j + a_j^2 u w_i}{\beta_j + u w_i} \right) \displaystyle\frac{d w_i}{w_i} , 
\end{flalign*}

\noindent where each $w_i$ is integrated along the contour $\gamma_i = \gamma (U; \Xi, S)$. Let us take the limit as each $a_j$ tends to $0$ (as in Section \ref{GeneralizedSpecialization}) and perform the change of variables sending $w_i$ to $- y_i / q u = - s y_i / \kappa$. We obtain 
\begin{flalign}
\begin{aligned}
\label{heightqgeneralizedm}
\mathbb{E}_{\textbf{M}} [q^{k \mathfrak{h}_{\lambda} (x + m)}] = \displaystyle\frac{q^{\binom{k}{2}}}{(2 \pi \textbf{i})^k} \displaystyle\oint \cdots & \displaystyle\oint \displaystyle\prod_{1 \le i < j \le k} \displaystyle\frac{y_i - y_j}{y_i - q y_j} \displaystyle\prod_{i = 1}^k  \left( \displaystyle\frac{1 + y_i}{1 + q^{-1} y_i} \right)^t \left( \displaystyle\frac{1 + q^{-1} \kappa^{-1} y}{1 + \kappa^{-1} y_i}  \right)^{x - 1}  \\
& \quad \times \displaystyle\prod_{i = 1}^k \left( \displaystyle\prod_{j = 1}^m \displaystyle\frac{1}{1 - q^{-1} \beta_j^{-1} y_i}  \right) \displaystyle\frac{d y_i}{y_i}, 
\end{aligned}
\end{flalign}

\noindent where each $y_i$ is integrated along the contour $\gamma_i (\widehat{U}; \widehat{\Xi}, \widehat{S})$

In the above, $\textbf{M}$ refers to the measure on particle configurations induced by the stochastic higher spin vertex model with spectral parameters $(u, u, \ldots , u)$, inhomogeneity parameters $(-\beta_1 / a_1 u, - \beta_2 / a_2 u, \ldots , - \beta_m / a_m u, 1, 1, \ldots )$, and spin parameters $(a_1, a_2, \ldots , a_m, s, s, \ldots )$, where the first $m$ spins $a_1, a_2, \ldots , a_m$ all tend to $0$. Recall from the explanation in Section \ref{GeneralizedSpecialization} that this stochastic higher spin vertex model in fact coincides with the stochastic six-vertex model with step $(b_1, b_2, \ldots , b_m)$-Bernoulli initial data, shifted to the right $m$ spaces. Hence, $\mathbb{E} [q^{k \mathfrak{h}_t (x)}] = \mathbb{E}_{\textbf{M}} [q^{k \mathfrak{h}_{\lambda} (x + m)}]$, where the former expectation is with respect to the stochastic six-vertex model with step $(b_1, b_2, \ldots , b_m)$-Bernoulli initial data. Thus, \eqref{heightqgeneralized} follows from \eqref{heightqgeneralizedm}. 
\end{proof}

\subsection{Fredholm Determinants for Models With Generalized Step Bernoulli Initial Data} 

\label{DeterminantVertical}

The goal of this section is to establish Fredholm determinant identities for the $q$-Laplace transform of the current of the stochastic six-vertex model (Theorem \ref{hdeterminant}) and the ASEP (Theorem \ref{hdeterminantprocess}), both with generalized step Bernoulli initial data; we refer to Appendix \ref{Determinants1} for the definition and properties of Fredholm determinants. Given Proposition \ref{manydeformationexpectationn}, the derivation of these two theorems will largely follow Section 5 of \cite{DDQA}.

\subsubsection{A General Fredholm Determinant Identity}

\label{DeterminantsGeneral} 

The goal of this section is to state Lemma \ref{momentdeterminant}, which produces Fredholm determinant identities from $q$-moment identities of the type given in Proposition \ref{manydeformationexpectationn}; this lemma is similar to Proposition 3.6 of \cite{DDQA}, and it can be proven analogously. 

The following appears as Definition 5.1 of \cite{DDQA}. 

\begin{definition}[{\cite[Definition 5.1]{DDQA}}]

\label{kmoment}

Let $a \in \mathbb{C}$ be a non-zero complex number and $q \in (0, 1)$ be a positive real number. Let $f$ be a meromorphic function that has a pole at $a$, but no other poles in a neighborhood of the segment joining $0$ and $a$. Define $\textbf{m}_0 = 1$ and, for each positive integer $k$, define 
\begin{flalign*}
\textbf{m}_k = \textbf{m}_{k, f} =  \displaystyle\frac{q^{\binom{k}{2}}}{(2 \pi \textbf{i})^k} \displaystyle\oint \cdots \displaystyle\oint \displaystyle\prod_{1 \le i < j \le k} \displaystyle\frac{z_i - z_j}{z_i - q z_j} \displaystyle\prod_{i = 1}^k f(z_i) z_i^{-1} d z_i ,
\end{flalign*}

\noindent where each $z_i$ is integrated along a positively oriented contour $C_i$ that satisfies the following two properties. First, each contour contains $0$ and $a$, but does not contain any other poles of $f$. Second, for all integers $i < j \in [1, k]$, the interior of $C_i$ does not intersect the contour $q C_j$. 

\end{definition}

\begin{rem}

The equality \eqref{heightqgeneralized} gives an example for $\textbf{m}_k$ if we set $\textbf{m}_k = \mathbb{E} [q^{k \mathfrak{h}_t (x)}]$ and $a = -q$. 

\end{rem}

\noindent We also require the following types of contours; they appeared previously in \cite{DDQA} as Definition 3.5 (in the case $\delta = 1 / 2$). 

\begin{definition}

\label{contourslong}

Let $R, d, \delta > 0$ be positive real numbers with $d, \delta < 1$ and $R > 1$. Let $D_{R, d, \delta} \subset \mathbb{C}$ denote the contour in the complex plane, with nondecreasing imaginary part, formed by taking the union of intervals 
\begin{flalign*}
(R - \textbf{i} \infty, R - \textbf{i} d ) \cup [R - \textbf{i} d , \delta - \textbf{i} d ] \cup [\delta - \textbf{i} d, \delta + \textbf{i} d] \cup [\delta + \textbf{i} d, R + \textbf{i} d ] \cup (R + \textbf{i} d, R + \textbf{i} \infty). 
\end{flalign*}

\noindent Furthermore, let $k > R$ be a positive integer. Let $z \in \mathbb{C}$ be the complex number satisfying $\Re z = R$, $\Im z > 0$, and $|z - \delta| = k$. Let $I$ denote the minor arc of the circle in the complex half-plane, centered at $\delta$, with radius $R$, connecting $z$ to its conjugate $\overline{z}$. Then let $D_{R, d, \delta; k} \subset \mathbb{C}$ denote the negatively oriented contour in the complex plane, formed by the union 
\begin{flalign*}
(\overline{z}, R - \textbf{i} d ) \cup [R - \textbf{i} d , \delta - \textbf{i} d ] \cup [\delta - \textbf{i} d, \delta + \textbf{i} d] \cup [\delta + \textbf{i} d, R + \textbf{i} d] \cup (R + \textbf{i} d, z) \cup I. 
\end{flalign*}
\end{definition}

Observe that the contour $D_{R, d, \delta; k}$ approximates the contour $D_{R, d, \delta}$ as $k$ tends to $\infty$. Examples of these contours are given in Figure \ref{contourslongfigure}. 

\begin{figure}

\begin{center}

\begin{tikzpicture}[
      >=stealth,
			]
			
			\draw[-, black] (-7, 0) -- (-1, 0); 
			\draw[-, black] (-6, -3) -- (-6, 3); 
			\draw[-, black] (7.5, 0) -- (1, 0); 
			\draw[-, black] (2, -3) -- (2, 3); 
			
			\draw[->, black, thick] (-2.2, -2.5) -- (-2.2, -1.4) node[black, below = 10, right = 0] {$D_{R, r, \delta}$};
			\draw[-, black, thick] (-2.2, -1.45) -- (-2.2, -.6);
			\draw[-, black,  thick] (-5.2, -.6) -- (-2.2, -.6);
			\draw[-, black,  thick] (-5.2, -.6) -- (-5.2, .6);
			\draw [<->, black, dashed] (-4.5, 0) -- (-4.5, .6) node[black, right=5, below = 0] {$d$}; 
			\draw [->, black, dashed] (-5.6, -.4) -- (-6, -.4); 
			\draw [->, black, dashed] (-5.6, -.4) -- (-5.2, -.4); 
			\draw [-, white] (-5.595, -.44) -- (-5.605, -.4) node[black, below = 0] {$\delta$}; 
			\draw[-, black,  thick] (-5.2, .6) -- (-2.2, .6);
			\draw[->, black,  thick] (-2.2, .6) -- (-2.2, 2.5);
			\draw[->, black,  dashed] (-4.4, 1.5) -- (-6, 1.5);
			\draw[->, black,  dashed] (-4.4, 1.5) -- (-2.2, 1.5); 
			\draw[->, white] (-4.37, 1.5) -- (-4.36, 1.5) node[black, above = 0] {$R$};

			\draw[->, black,  thick] (5.8, -2.5) -- (5.8, -1.4)  node[black, below = 10, right = 18] {$D_{R, r, \delta; k}$};
			\draw[-, black,  thick] (5.8, -1.45) -- (5.8, -.6);
			\draw[-, black,  thick] (2.8, -.6) -- (5.8, -.6);
			\draw[->, black,  dashed] (3.6, 1.5) -- (2, 1.5);
			\draw[->, black,  dashed] (3.6, 1.5) -- (5.8, 1.5);
			\draw[->, white] (3.63, 1.5) -- (3.64, 1.5) node[black, above = 0] {$R$};
			\draw[-, black,  thick] (2.8, -.6) -- (2.8, .6);
			\draw[-, black,  thick] (2.8, .6) -- (5.8, .6);
			\draw[-, black,  thick] (5.8, 2.5) -- (5.8, .6);
			\draw [<->, black, dashed] (3.5, 0) -- (3.5, .6) node[black, right=5, below = 0] {$d$}; 
			\draw [->, black, dashed] (2.4, -.4) -- (2, -.4); 
			\draw [->, black, dashed] (2.4, -.4) -- (2.8, -.4); 
			\draw [-, white] (2.395, -.4) -- (2.405, -.4) node[black, below = 0] {$\delta$};

			\draw[black, thick] (5.8, -2.5) arc (-45:45:3.55);
			
			\filldraw[fill=white, draw=black] (-5, 0) circle [radius=.07];
			\filldraw[fill=white, draw=black] (-4, 0) circle [radius=.07];
			\filldraw[fill=white, draw=black] (-3, 0) circle [radius=.07];
			\filldraw[fill=white, draw=black] (-2, 0) circle [radius=.07];

			\filldraw[fill=white, draw=black] (3, 0) circle [radius=.07];
			\filldraw[fill=white, draw=black] (4, 0) circle [radius=.07];
			\filldraw[fill=white, draw=black] (5, 0) circle [radius=.07];
			\filldraw[fill=white, draw=black] (6, 0) circle [radius=.07] node[black, below = 2] {$k$};
			\filldraw[fill=white, draw=black] (7, 0) circle [radius=.07];

\end{tikzpicture}

\end{center}

\caption{\label{contourslongfigure} An example of $D_{R, r, \delta}$ is shown to the left, and an example of $D_{R, r, \delta; k}$ is shown to the right; the circles depict the locations of positive integers. }

\end{figure}
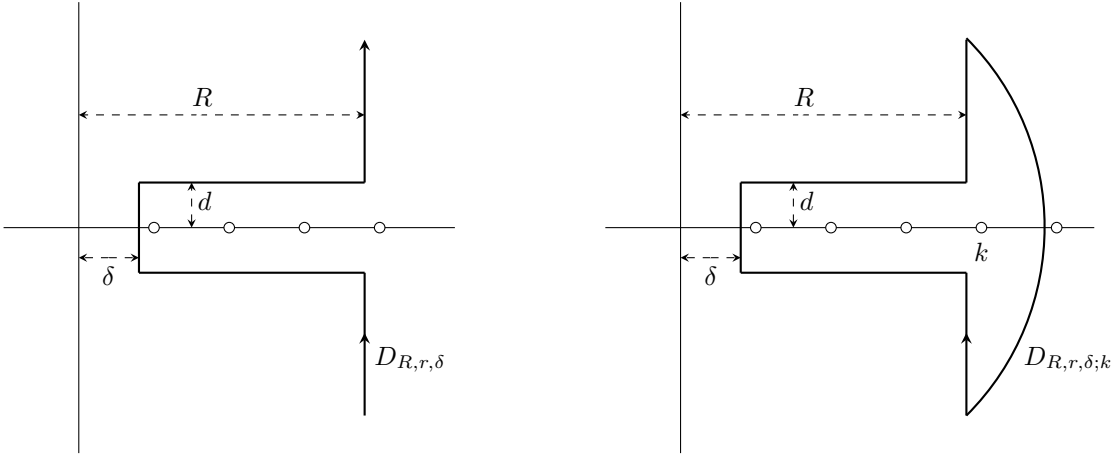

The following proposition, which is similar to Proposition 3.6 of \cite{DDQA}, states that one has a Fredholm determinant identity for a certain generating series of the sequence $\{ \textbf{m}_0, \textbf{m}_1, \ldots \}$. 

\begin{lem}

\label{momentdeterminant}

Let $a \in \mathbb{C}$ be a complex number, and let $f$ be a meromorphic function satisfying the properties of Definition \ref{kmoment}. Assume that $f$ is of the form $f(z) = g(z) / g(qz)$, for some function $g$. Let $C \subset \mathbb{C}$ be a compact contour containing $0$ and $a$ in its interior, but no other pole of $f$; moreover, assume that the interior of $qC$ is contained inside $C$. 

Suppose that there exists some constant $B > 0$ such that $f(q^n w) < B$ and $|q^n w - w'| > B^{-1}$ for all $w, w' \in C$ and integers $n \ge 1$. Also assume that there exist positive real numbers $d, \delta \in (0, 1)$ and $R > 1$ such that 
\begin{flalign}
\label{convergencedeterminant2} 	
\displaystyle\inf_{\substack{w, w' \in C \\ k \in (2R, \infty) \cap \mathbb{Z} \\ s \in D_{R, d, \delta; k}}} |q^s w - w'| > 0; \qquad \displaystyle\sup_{\substack{ w \in C \\ k \in (2R, \infty) \cap \mathbb{Z} \\ s \in D_{R, d, \delta; k}}} \displaystyle\frac{g(w)}{g(q^s w)} < \infty. 
\end{flalign}

\noindent Then, for any complex number $\zeta \in \mathbb{C} \setminus\mathbb{R}_{\ge 0}$ satisfying $|\zeta| < B^{-1} (1 - q)^{-1}$, we have that 
\begin{flalign}
\label{generaldeterminant} 
\displaystyle\sum_{k = 0}^{\infty} \displaystyle\frac{\textbf{\emph{m}}_k \zeta^k}{(q; q)_k} = \det \big( \Id + K_{\zeta} \big)_{L^2 (C)}, 
\end{flalign}

\noindent where the kernel $K$ is defined through 
\begin{flalign*}
K_{\zeta} (w, w') = \displaystyle\frac{1}{2 \textbf{\emph{i}}} \displaystyle\int_{D_{R, d, \delta}} \displaystyle\frac{g(w)}{g(q^r w)} \displaystyle\frac{(-\zeta)^r dr}{\sin (\pi r) \big( q^r w - w' \big)}. 
\end{flalign*}

\noindent Furthermore, the right side of \eqref{generaldeterminant} is analytic in $\zeta \in \mathbb{C} \setminus\mathbb{R}_{\ge 0}$. 
\end{lem}

The proof of this lemma is very similar to that of Proposition 3.6 of \cite{DDQA} and is thus omitted.

\subsubsection{Fredholm Determinants for the Stochastic Six-Vertex Model and ASEP} 

\label{DeterminantsSpin}

In this section, we will establish the Fredholm determinant identities Theorem \ref{hdeterminant} and Theorem \ref{hdeterminantprocess}. As stated previously, these two identities will follow from Proposition \ref{manydeformationexpectationn} and Lemma \ref{momentdeterminant}. 

\begin{thm}
\label{hdeterminant}

Let $m$ be a non-negative integer and $t$ be a positive integer. Fix $0 < \delta_1 < \delta_2 < 1$ and $b_1, b_2, \ldots , b_m \in (0, 1)$. Denoting $q$, $u$, $\kappa$, $s$, and $\beta_i$ as in \eqref{stochasticparameters}, let 
\begin{flalign}
\label{gv}
g_V (z; x, t) = \displaystyle\frac{(\kappa^{-1} z + q)^{x - 1}}{(z + q)^t} \displaystyle\prod_{j = 1}^m \displaystyle\frac{1}{(q^{-1} \beta_j^{-1} z; q)_{\infty}}. 
\end{flalign}

\noindent Let $\mathscr{C}_V$ be a positively oriented circle in the complex plane, centered at some non-positive real number, containing $-q$ and $0$, but leaving outside $- 1$ and $q \beta_j$ for each integer $j \in [1, m]$. 

Then, there exist a sufficiently large real number $R > 1$, a real number $\delta \in (0, 1)$ sufficiently close to $1$, and a sufficiently small real number $d \in (0, 1)$ such that 
\begin{flalign}
\begin{aligned}
\label{inequalityrddelta}
q^R < & \displaystyle\min_{w, w' \in \mathscr{C}_V} | w^{-1} w'|; \qquad q^{1 - \delta} > \displaystyle\max_{w \in \mathscr{C}_V} \{ |  \kappa^{-1} w |, | qw | \} ; \\
& \qquad \qquad \big| \Im q^{id} \big| < \displaystyle\max_{\substack{w, w' \in \mathscr{C}_V \\ \Re w^{-1} w' \ge 0 \\ |w^{-1} w' | \le q^{\delta}}} \left| \Im w^{-1} w' \right|. 
\end{aligned}
\end{flalign}

\noindent Consider the stochastic six-vertex model with jump probabilities $\delta_1$ and $\delta_2$, and with $(b_1, b_2, \ldots , b_m)$-step Bernoulli initial data (as defined in Section \ref{StochasticSixVertexModel}). For any $\zeta \in \mathbb{C} \setminus\mathbb{R}_{\ge 0}$, we have that 
\begin{flalign}
\label{hdeterminantidentity2}
\mathbb{E} \left[ \displaystyle\frac{1}{\big( \zeta q^{\mathfrak{h}_t (x)} ; q \big)_{\infty}} \right] = \det \big( \Id + V_{\zeta} \big)_{L^2 (\mathscr{C}_V)}, 
\end{flalign}

\noindent where $V_{\zeta}$ is defined by 

\begin{flalign*}
V_{\zeta} (w, w') = \displaystyle\frac{1}{2 \textbf{\emph{i}}} \displaystyle\int_{D_{R, d, \delta}} \displaystyle\frac{g_V (w; x, t)}{g_V (q^r w; x, t)} \displaystyle\frac{(-\zeta)^r dr}{\sin (\pi r) (q^r w - w')}. 
\end{flalign*}	

\end{thm} 

\begin{proof}

Let us begin with the first statement. The existence of $R$ follows from the compactness of $\mathscr{C}_V$ and the fact that $0$ does not lie on the contour $\mathscr{C}_V$. The existence of $\delta$ follows from the compactness $\mathscr{C}_V$ and the fact that $\mathscr{C}_V$ and does not intersect the circle centered at $0$ with radius $\kappa$ or radius $q^{-1}$. The existence of $d$ follows from the compactness of $\mathscr{C}_V$ and the fact that $\Im w^{-1} w' \ne 0$ whenever $\Re w^{-1} w' \ge 0$ and $|w^{-1} w'| < q^{\delta}$. Thus, $R$, $\delta$, and $d$ satisfying \eqref{inequalityrddelta} exist. 

To establish the Fredholm determinant identity \eqref{hdeterminantidentity2}, we will apply Lemma \ref{momentdeterminant}, with $\textbf{m}_k = \mathbb{E} [q^{k \mathfrak{h}_t (x)}]$ and $a = -q$; Proposition \ref{manydeformationexpectationn} shows that $\textbf{m}_k$ has the form given by Definition \ref{kmoment}, with $f(z) = g_V (z) / g_V (qz)$. To apply Lemma \ref{momentdeterminant}, we must verify several conditions on the contours $\mathscr{C}_V$ and $D_{R, d, \delta}$. 

First, we require that $\mathscr{C}_V$ contains $0$, $-q$, and its image under multiplication by $q$, but no other pole of $f$; these statements all quickly follow from the definition of $\mathscr{C}_V$. 

Second, we require the existence of some positive real number $B > 0$ such that $|q^n w - w'| > B^{-1}$ and $f (q^n w) < B$ for all integers $n \ge 1$ and all $w, w' \in \mathscr{C}_V$. By compactness of $\mathscr{C}_V$, the existence of such a $B$ would follow if $|q^n w - w'| > 0$ and $f (q^n w) < \infty$ for all $w, w' \in \mathcal{C}$ and $n \ge 1$. The first inequality holds by convexity of $\mathscr{C}_V$. To establish the second inequality, recall that the only poles of $f$ are at $-q$ and the $q \beta_j$. The $q \beta_j$ are outside the contour $\mathscr{C}_V$, so no element of the form $q^n w$, with $w \in \mathscr{C}_V$, is equal to some $q \beta_j$. Furthermore, since no non-positive power of $q$ lies on $\mathscr{C}_V$, we have that $q^n w \ne q^{-1}$ for any $w \in \mathscr{C}_V$. Thus, the second inequality holds since $q^n w$ is not a pole of $f$ for any $w \in \mathscr{C}_V$. 

Third, we require that 
\begin{flalign}	
\label{convergencevertexg}
\displaystyle\inf_{\substack{w, w' \in \mathscr{C}_V \\ k \in (2R, \infty) \cap \mathbb{Z} \\ s \in D_{R, d, \delta; k}}} |q^s w - w'| > 0; \qquad \displaystyle\sup_{\substack{ w \in \mathscr{C}_V \\ k \in (2R, \infty) \cap \mathbb{Z} \\ s \in D_{R, d, \delta; k}}} \displaystyle\frac{g_V (w)}{g_V (q^s w)} < \infty. 
\end{flalign}

To see that the first inequality in \eqref{convergencevertexg} holds, observe that if $s \in D_{R, d, \delta; k}$ then either $\Re s \ge R$, $\Re s = \delta$ and $\Im s \in (-d, d)$, or $\Im s = d$ and $\Re s \in (\delta, R)$. The first inequality in \eqref{inequalityrddelta} implies that $q^s \ne w^{-1} w'$ in the first case. Furthermore, the third inequality in \eqref{inequalityrddelta} implies that $q^s \ne w^{-1} w'$ in either the second or third case. Therefore, $q^s \ne w^{-1} w'$, for any $w, w' \in \mathscr{C}_V$ and $s \in D_{R, d, \delta; k}$ with $k > 2 R$; thus, the first inequality in \eqref{convergencevertexg} follows from compactness of $\mathscr{C}_V$. 

Now let us establish the second inequality in \eqref{convergencevertexg}. Again, since $\mathscr{C}_V$ is compact, it suffices to show that no pole of $g_V$ is on $\mathscr{C}_V$, and that no element of the form $q^s w$ (with $w \in \mathscr{C}_V$ and $s \in D_{R, d, \delta; k}$) is a zero of $g_V$. The first statement holds by the definition of $\mathscr{C}_V$. To establish the second statement, observe that the zeroes of $g_V$ are $- q \kappa$ and $-1$. If $q^s w = -q \kappa$, then we must have that $|w| = \kappa |q^{1 - s}| \le q^{1 - \delta} \kappa$, which cannot hold due to the second inequality in \eqref{inequalityrddelta}; by similar reasoning, $q^s w \ne -1$. Thus, the second statement also holds. 

Hence, the contours $\mathscr{C}_V$ and $D_{R, d, \delta}$ satisfy the conditions of Lemma \ref{momentdeterminant}. Applying this lemma to $\textbf{m}_k = \mathbb{E} [q^{k \mathfrak{h}_t (x)}]$, we deduce that 
\begin{flalign}
\label{determinant1vertex}
\displaystyle\sum_{k = 0}^{\infty} \displaystyle\frac{\zeta^k \mathbb{E} [q^{k \mathfrak{h}_t (x)}]}{(q; q)_k} = \det \big( \Id + V_{\zeta} \big)_{L^2 (\mathscr{C}_V)}, 
\end{flalign}

\noindent for any $\zeta \in \mathbb{C} \setminus\mathbb{R}_{\ge 0}$ satisfying $|\zeta| < (1 - q)^{-1} B^{-1}$. Since $q \in (0, 1)$, we have that $\mathbb{E} [q^{k \mathfrak{h}_t (x)}] < 1$. Hence, if $|\zeta| < 1$, then the left side of \eqref{determinant1vertex} converges absolutely, meaning that we can commute the sum with the expectation. Therefore, if $|\zeta| < 1$, we have that 
\begin{flalign*}
\displaystyle\sum_{k = 0}^{\infty} \displaystyle\frac{\zeta^k \mathbb{E} [q^{k \mathfrak{h}_t (x)}]}{(q; q)_k} & = \mathbb{E} \left[ \displaystyle\sum_{k = 0}^{\infty} \displaystyle\frac{ \big( \zeta q^{\mathfrak{h}_t (x)} \big)^k}{(q; q)_k} \right] = \mathbb{E} \left[ \displaystyle\frac{1}{\big( \zeta q^{\mathfrak{h}_t (x)}; q \big)_{\infty}} \right], 
\end{flalign*}

\noindent where we have used the $q$-binomial identity in the second equality. Substituting this into \eqref{determinant1vertex} yields 
\begin{flalign}
\label{determinant2vertex}
\mathbb{E} \left[ \displaystyle\frac{1}{\big( \zeta q^{\mathfrak{h}_t (x)}; q \big)_{\infty}} \right] = \det \big( \Id + V_{\zeta} \big)_{L^2 (\mathscr{C}_V)},
\end{flalign}

\noindent for any $\zeta \in \mathbb{C} \setminus\mathbb{R}_{\ge 0}$ such that $|\zeta| < \min \{1 , (1 - q)^{-1} B^{-1} \}$. 

Finally, for any real number $z \ge 0$, the function $(z \zeta; q)_{\infty}^{-1}$ is analytic in $\zeta$ in the domain $\mathbb{C} \setminus\mathbb{R}_{\ge 0}$. Since the left side of \eqref{determinant2vertex} is a locally uniformly convergent sum of such expressions, we deduce that the left side of \eqref{determinant2vertex} is analytic in $\zeta$ in the domain $\mathbb{C} \setminus\mathbb{R}_{\ge 0}$. Furthermore, Proposition \ref{momentdeterminant} states that the right side of \eqref{determinant2vertex} is also analytic in $\zeta$ in the domain $\mathbb{C} \setminus\mathbb{R}_{\ge 0}$. Uniqueness of analytic continuation concludes the proof.  
\end{proof}

Degenerating the stochastic six-vertex model to the ASEP as explained in Section \ref{CurrentConverge} yields the following Fredholm determinant identity for the ASEP with generalized step Bernoulli initial data. 

\begin{thm}
\label{hdeterminantprocess}

Let $m$ be a non-negative integer and $t$ be a positive real number. Fix $0 < L < R$ and $b_1, b_2, \ldots , b_m \in (0, 1)$. Denote $q = L / R < 1$ and $\beta_i = b_i / (1 - b_i)$ for each integer $i \in [1, m]$. Define
\begin{flalign}
\label{hdeterminantidentity}
g_A (z; x, T) = (z + q)^{x - 1} \exp \left( \displaystyle\frac{q T (R - L)}{z + q} \right) \displaystyle\prod_{j = 1}^m \displaystyle\frac{1}{(q^{-1} \beta_j^{-1} z; q)_{\infty}}. 
\end{flalign}

\noindent Let $\mathscr{C}_A$ be a positively oriented circle in the complex plane, centered at a non-positive real number, containing $-q$ and $0$, but leaving outside $- 1$ and $q \beta_j$ for each integer $j \in [1, m]$. There exist a sufficiently large real number $R > 1$, a real number $\delta \in (0, 1)$ sufficiently close to $1$, and a sufficiently small real number $d \in (0, 1)$ such that 
\begin{flalign}
\label{inequalityrddeltaa}
q^R < \displaystyle\min_{w, w' \in \mathscr{C}_A} | w^{-1} w'|; \qquad q^{1 - \delta} > \displaystyle\max_{w \in \mathscr{C}_A} |w| ; \qquad \big| \Im q^{id} \big| < \displaystyle\max_{\substack{w, w' \in \mathscr{C}_A \\ \Re w^{-1} w' \ge 0 \\ |w^{-1} w' | \le q^{\delta}}} \left| \Im w^{-1} w' \right|. 
\end{flalign}

\noindent Consider the ASEP with left jump rate $L$ and right jump rate $R$, with step $(b_1, b_2, \ldots , b_m)$-Bernoulli initial data (as defined in Section \ref{InitialDataVertex}). For any $\zeta \in \mathbb{C} \setminus\mathbb{R}_{\ge 0}$, we have that 
\begin{flalign}
\label{hdeterminantidentity3}
\mathbb{E} \left[ \displaystyle\frac{1}{\big( \zeta q^{\mathfrak{h}_t (x)} ; q \big)_{\infty}} \right] = \det \big( \Id + A_{\zeta} \big)_{L^2 (\mathscr{C}_A)}, 
\end{flalign}

\noindent where $\mathfrak{h}_t (x) = J_t (x - 1)$, and $A_{\zeta}$ is defined by 

\begin{flalign*}
A_{\zeta} (w, w') = \displaystyle\frac{1}{2 \textbf{\emph{i}}} \displaystyle\int_{D_{R, d, \delta}} \displaystyle\frac{g_A (w; x, T)}{g_A (q^r w; x, T)} \displaystyle\frac{(-\zeta)^r dr}{\sin (\pi r) (q^r w - w')} . 
\end{flalign*}

\end{thm} 

\begin{proof}

The proof of the first statement is similar to that of the first statement in Theorem \ref{hdeterminant} (in fact, it is the $\kappa = 1$ degeneration of that statement), and is thus omitted. 

For the proof of the second part of the theorem, we apply the degeneration of the stochastic six-vertex model to the ASEP explained in Section \ref{CurrentConverge}. Specifically, set $\delta_1 = \varepsilon L$, $\delta_2 = \varepsilon R$, $t = \varepsilon^{-1} T$, and let $\varepsilon$ tend to $0$. Then, the distribution of $\mathfrak{h}_t (x + t) = \mathfrak{H} (x + t - 1, t)$ under the the stochastic six-vertex model with jump rates $\delta_1$ and $\delta_2$ (as defined in Section \ref{StochasticSixVertexModel}) converges to the distribution of $J_t (x)$ under the ASEP with left jump rate $L$ and right jump rate $R$, due to Proposition \ref{currentmodelprocess}. 

Therefore, $\mathbb{E}_V \big[ (\zeta q^{\mathfrak{h}_t (x + t)}; q)_{\infty}^{-1} \big]$ converges to $\mathbb{E}_A \big[ (\zeta q^{\mathfrak{h}_t (x)}; q)_{\infty}^{-1} \big]$, since both $(\zeta q^{\mathfrak{h}_t (x + t)}; q)_{\infty}^{-1}$ and $(\zeta q^{\mathfrak{h}_t (x)}; q)_{\infty}^{-1}$ are uniformly bounded; here, the first expectation is with respect to the stochastic six-vertex model and the second expectation is with respect to the ASEP. Thus, the left side of \eqref{hdeterminantidentity2} (with $x$ replaced by $x + t$) converges to the left side of \eqref{hdeterminantidentity3}. 

Moreover, it is quickly verified that 
\begin{flalign*}
\displaystyle\lim_{\varepsilon \rightarrow 0} g_V (z; x + t, t) & = e^{T (L - R)} g_A (z; x, T). 
\end{flalign*}

\noindent Hence, $V_{\zeta} (w, w')$ tends to $A_{\zeta} (w, w')$ as $\varepsilon$ tends to $0$, for all $w, w' \in \mathscr{C}_A$. 

Now, as $\varepsilon$ tends to $0$, $\kappa$ tends to $1$. Therefore, if $\varepsilon$ is sufficiently small (in comparison to $R$, $\delta$, and $d$), the estimates \eqref{inequalityrddelta} are implied by the estimates \eqref{inequalityrddeltaa}. Hence, we may take the contour $\mathscr{C}_V$ from Theorem \ref{hdeterminant} to be the contour $\mathscr{C}_A$. 

Since $\mathscr{C}_A$ is compact, the fact that $\lim_{\varepsilon \rightarrow 0} V_{\zeta} (w, w') = A_{\zeta} (w, w')$ implies that the right side of \eqref{hdeterminantidentity2} converges to the right side of \eqref{hdeterminantidentity3} as $\varepsilon$ tends to $0$ (for instance, by Corollary \ref{determinantlimitkernels}). 

Therefore, \eqref{hdeterminantidentity3} is a limit degeneration of \eqref{hdeterminantidentity2}, and the proof is complete. 
\end{proof}

\begin{rem}

\label{limitprocessequation}

Taking the \emph{weakly asymmetric limit} \cite{SEPS} of Theorem \ref{hdeterminantprocess} might provide an alternative proof of Theorem 1.10 of \cite{FEF}, which is a Fredholm determinant identity for the Laplace transform of the free energy of the continuum directed random polymer with an $m$-spiked perturbation. 
\end{rem}

\section{Preliminary Remarks for the Asymptotic Analysis} 

\label{TransitionsModel}

In this section we begin the analysis of the Fredholm determinant identities given by Theorem \ref{hdeterminant} and Theorem \ref{hdeterminantprocess}. Our goal is twofold. First, in Section \ref{ModifiedDeterminants}, we will rewrite the kernels $V_{\zeta}$ and $A_{\zeta}$ (defined in Theorem \ref{hdeterminant} and Theorem \ref{hdeterminantprocess}, respectively) in a way that will be useful for asymptotic analysis later. Then, in Section \ref{ApplicationsDeterminantModelProcess}, we will state the asymptotics (see Proposition \ref{processdeterminant} and Proposition \ref{modeldeterminant}) for the corresponding Fredholm determinants and also show how these asymptotics imply Theorem \ref{hlimit} and Theorem \ref{asymmetriclimit}. The proofs of Proposition \ref{processdeterminant} and Proposition \ref{modeldeterminant} will be the topics of later sections of the paper.

\subsection{Reformulations of Theorem \ref{hdeterminant} and Theorem \ref{hdeterminantprocess}}

\label{ModifiedDeterminants}

Theorem \ref{hdeterminant} and Theorem \ref{hdeterminantprocess} are of a similar form. Specifically, they both state that 
\begin{flalign}
\label{hboth}
\mathbb{E}_K \left[ \displaystyle\frac{1}{\big( \zeta q^{\mathfrak{h}_T (x)}; q \big)_{\infty}} \right] = \det \big( \Id + K_{\zeta} \big)_{L^2 (\mathscr{C}_K)}, 
\end{flalign}

\noindent where $\mathfrak{h}_T (x) = \mathfrak{H} (x - 1, T)$ in the case of the stochastic six-vertex model, and $\mathfrak{h}_T (x) = J_T (x - 1)$ in the case of the ASEP. In \eqref{hboth}, the notation $K$ either stands for $V$ (for the stochastic six-vertex model) or $A$ (for the ASEP); the kernels $V_{\zeta}$ and $A_{\zeta}$ are defined in Theorem \ref{hdeterminant} and Theorem \ref{hdeterminantprocess}, respectively. If $K = V$, then the expectation on the left side of \eqref{hboth} is with respect to the stochastic six-vertex model with step $(b_1, b_2, \ldots , b_m)$-Bernoulli initial data, run to some integer time $T$. If $K = A$, then the expectation on the left side of \eqref{hboth} is with respect to the ASEP with step $(b_1, b_2, \ldots , b_m)$-Bernoulli initial data, run up to some real time $T$. 

Recall from Theorem \ref{hdeterminant} and Theorem \ref{hdeterminantprocess} that 
\begin{flalign*}
K_{\zeta} (w, w') = \displaystyle\frac{1}{2 \textbf{i}} \displaystyle\int_{D_{R, d, \delta}} \displaystyle\frac{g_K (w; x, T)}{g_K (q^r w; x, t)} \displaystyle\frac{(-\zeta)^r dr}{\sin (\pi r) \big( q^r w - w' \big)}, 
\end{flalign*}
	
\noindent for any $w, w' \in \mathscr{C}_K$; the functions $g_V$ and $g_A$ are given by \eqref{gv} and \eqref{hdeterminantidentity}, respectively. 

Let us set $\zeta = - q^p$, for some real number $p$, and change variables from $r$ to $v = q^r w$. To understand how this affects the above identity for $K_{\zeta}$, we must analyze the contour for $v$. 

To that end, first observe that if $d$ is sufficiently small (which we will always assume to be the case), the contour $D_{R, d, \delta}$ can be written as the union $D_{R, d, \delta} = I' \cup \bigcup_{k \in \mathbb{Z} \setminus\{ -1 \}} I_k$. Here, $I_k$ is defined to be the interval
\begin{flalign*}
I_k = \Big[ R + \textbf{i} \big(d + 2 \pi |\log q|^{-1} k \big), R + \textbf{i} \big(d + 2 \pi |\log q|^{-1} (k + 1) \big) \Big],
\end{flalign*} 

\noindent for each integer $k$, and $I'$ is defined to be the piecewise linear curve
\begin{flalign*}
I' = \Big[ R + \textbf{i} \big(d - 2 \pi |\log q|^{-1} \big), R - \textbf{i} d \Big] \cup \big[ R -  \textbf{i} d, \delta - \textbf{i} d \big] \cup \big[ \delta - \textbf{i} d, \delta + \textbf{i} d \big] \cup [\delta + \textbf{i} d, R + \textbf{i} d]. 
\end{flalign*}

Now, for each integer $k$, the map from $r$ to $q^r w$ is a bijection from $I_k$ to $-Q_{R, d, \delta; w}^{(1)}$, where the negative refers to reversal of orientation and the contour $Q_{R, d, \delta; w}^{(1)}$ is a positively oriented circle of radius $q^R$, centered at $0$; here, the negative orientation is due to the fact that $\log q < 0$. 

Furthermore, the map from $r$ to $q^r w$ is a bijection from $I'$ to the contour $- Q_{R, d, \delta; w}$, where $Q_{R, d, \delta; w}$ is defined to be the union of four curves $\mathcal{J}_1 \cup \mathcal{J}_2 \cup \mathcal{J}_3 \cup \mathcal{J}_4$, which are defined are as follows. The curve $\mathcal{J}_1$ is the major arc of the circle centered at $0$, with radius $q^R |w|$, connecting $q^{R - id} w$ to $q^{R + id } w$. The curve $\mathcal{J}_2$ is the line segment in the complex plane connecting $q^{R + id} w$ to $q^{\delta + id} w$; the curve $\mathcal{J}_3$ is the minor arc of the circle centered at $0$, with radius $q^{\delta} |w|$, connecting $q^{\delta + id} w$ to $q^{\delta - id} w$; and the curve $\mathcal{J}_4$ is the line segment in the complex plane connecting $q^{\delta - id} w$ to $q^{R - id} w$. We refer to Figure \ref{contoursgammaq} for an example of the contour $Q_{R, d, \delta; w}$. 

Thus, since $dv = v \log q dr$ and $(-\zeta)^r = v^p w^{-p}$, we obtain that
\begin{flalign*}
K_{\zeta} (w, w') & = - \displaystyle\frac{1}{2 \textbf{i} \log q} \displaystyle\sum_{j \ne 0} \displaystyle\oint_{Q_{R, d, \delta; w}^{(1)}} \displaystyle\frac{g_K (w; x, T)}{g_K (v; x, t)}  \displaystyle\frac{v^{p - 1} w^{-p} }{\sin \big( \frac{\pi}{\log q} ( \log v - \log w + 2 \pi \textbf{i} j ) \big)} \displaystyle\frac{dv}{v - w'}  \nonumber \\
& \qquad - \displaystyle\frac{1}{2 \textbf{i} \log q} \displaystyle\oint_{Q_{R, d, \delta; w}} \displaystyle\frac{g_K (w; x, T)}{g_K (v; x, T)}  \displaystyle\frac{v^{p - 1} w^{-p} 	}{\sin \big( \frac{\pi}{\log q} ( \log v - \log w ) \big) } \displaystyle\frac{dv}{v - w'} , 
\end{flalign*}

\noindent where we have taken the principal branch of the logarithm. 

We can rewrite $K_{\zeta} = K_{\zeta} (w, w')$ as 
\begin{flalign}
\label{kvw}
\begin{aligned}
K_{\zeta} & = - \displaystyle\frac{1}{2 \textbf{i} \log q} \displaystyle\sum_{j = -\infty}^{\infty} \displaystyle\oint_{Q_{R, d, \delta; w}} \displaystyle\frac{g_K (w; x, T)}{g_K (v; x, t)}  \displaystyle\frac{v^{p - 1} w^{-p}}{\sin \big( \frac{\pi}{\log q} ( \log v - \log w + 2 \pi \textbf{i} j ) \big)}\displaystyle\frac{dv}{ v - w'}  \\
& \quad - \displaystyle\frac{1}{2 \textbf{i} \log q} \displaystyle\sum_{j \ne 0} \displaystyle\oint_{Q_{R, d, \delta; w}^{(2)} } \displaystyle\frac{g_K (w; x, T)}{g_K (v; x, T)}  \displaystyle\frac{v^{p - 1} w^{-p}} {\sin \big( \frac{\pi}{\log q} ( \log v - \log w + 2 \pi \textbf{i} j) \big)} \displaystyle\frac{dv}{ v - w'} , 
\end{aligned}
\end{flalign}

\noindent where the contour $Q_{R, d, \delta; w}^{(2)} = Q_{R, d, \delta; w}^{(1)} - Q_{R, d, \delta; w}$ is the union of four curves $\mathcal{C}_1 \cup \mathcal{C}_2 \cup \mathcal{C}_3 \cup \mathcal{C}_4$, which are defined as follows. The curve $\mathcal{C}_1$ is the line segment in the complex plane connecting $q^{R - id} w$ to $q^{\delta - id} w$. The curve $\mathcal{C}_2$ is the minor arc of the circle centered at $0$, with radius $q^{\delta} |w|$, connecting $q^{\delta - id} w$ to $q^{\delta + id} w$. The curve $\mathcal{C}_3$ is the line segment connecting $q^{\delta + id} w$ to $q^{R + id} w$. The curve $\mathcal{C}_4$ is the minor arc of the circle centered at $0$, with radius $q^R |w|$, connecting $q^{R + id} w$ to $q^{R - id} w$. We refer to Figure \ref{contoursgammaq} for an example of the contour $Q_{R, d, \delta; w}^{(2)}$. 

\begin{figure}

\begin{center}

\begin{tikzpicture}[
      >=stealth,
			scale=.8
			]

			\draw[-, black] (5, 0) -- (-1, 0); 
			\draw[-, black] (0, -1) -- (0, 3);

			\draw[->, black,  thick] (.57923,.40558) -- (1.44808, 1.01395);
			\draw[-, black,  thick] (1.44808, 1.01395) -- (2.31692, 1.62232);
			\draw[-, black,  thick] (.40558, .57923) -- (1.62232, 2.31692);
			
			\draw[black, thick] (2, 2) arc (45:55:2.8284);
			\draw[black, thick] (2, 2) arc (45:35:2.8284);
			\draw[black, thick] (.40558, .57923) arc (55:180:.7071);
			\draw[black, thick] (.57923, .40558) arc (35:-180:.7071) node [black, below = 24, right = 15] {$Q_{R, d, \delta; w}$};

			\draw[<->, black, dashed] (0, 0) -- (-.5, .5) node[black, left = 10, above = 0] {$q^R |w|$};
			\draw[<->, black, dashed] (0, 0) -- (2, 2) node[black, below = 4, left = 18] {$q^{\delta} |w|$};

			\draw[black, fill] (2.5, 2.5) circle [radius=.04] node [black, above = 2] {$w$};

			\draw[-, black] (15, 0) -- (9, 0); 
			\draw[-, black] (10, -1) -- (10, 3);

			\draw[-, black,  thick] (10.57923,.40558) -- (11.44808, 1.01395);
			\draw[->, black,  thick] (12.31692, 1.62232) -- (11.44808, 1.01395);
			\draw[-, black,  thick] (10.40558, .57923) -- (11.62232, 2.31692);
			
			\draw[black, thick] (12, 2) arc (45:55:2.8284);
			\draw[black, thick] (12, 2) arc (45:35:2.8284);
			\draw[black, thick] (10.5, .5) arc (45:55:.7071);
			\draw[black, thick] (10.5, .5) arc (45:35:.7071) node [black, below = 1, right = 4] {$Q_{R, d, \delta; w}^{(2)}$};

			\draw[black, fill] (12.5, 2.5) circle [radius=.04] node [black, above = 2] {$w$};

\end{tikzpicture}

\end{center}

\caption{\label{contoursgammaq} Shown to the left is the contour $Q_{R, d, \delta; w}$. Shown to the right is the contour $Q_{R, d, \delta; w}^{(2)}$. } 
\end{figure}
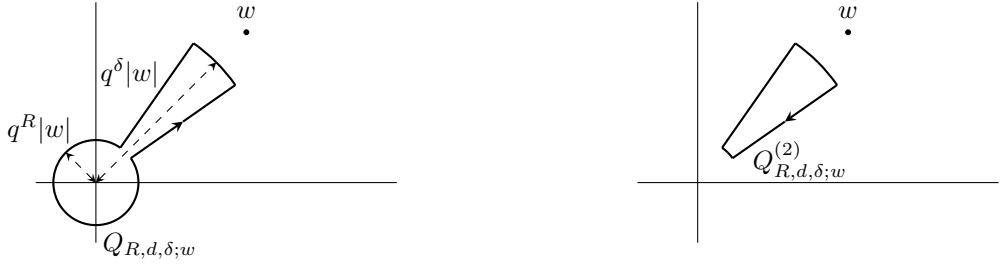

We claim that each summand of the second sum on the right side of \eqref{kvw} is equal to $0$. To verify this, we analyze the poles of the corresponding integrands. When $K = V$, these poles are at $v \in \{0, w', -q \kappa \}$ and, when $K = A$, these poles are at $v \in \{ 0, w', -q \}$. Observe that there is no pole at $v = w$ (or at $v \in w q^{\mathbb{Z}}$) since $j$ is non-zero in each summand. 

The pole $0$ is not contained inside $Q_{R, d, \delta; w}^{(2)}$. Furthermore, in the case $K = V$, the inequalities \eqref{inequalityrddelta} imply that the poles at $v \in \{ w', -q \kappa \}$ are not contained inside $Q_{R, d, \delta; w}^{(2)}$; similarly, in the case $K = A$, the inequalities \eqref{inequalityrddeltaa} imply that the poles at $v \in \{ w', - q \}$ are not contained inside $Q_{R, d, \delta; w}^{(2)}$. Thus, no poles of the integrand of the second sum on the right side of \eqref{kvw} are contained in the interior of $Q_{R, d, \delta; w}^{(2)}$, from which we deduce that the corresponding integrals are equal to $0$. 

Next, we would like to deform the contour $Q_{R, d, \delta; w}$ on the right side of \eqref{kvw} to one (to be chosen later) that will be better suited for asymptotic analysis. To that end, recall that a contour $\gamma \subset \mathbb{C}$ is called \emph{star-shaped} (with respect to the origin) if, for each real number $a \in \mathbb{R}$, there exists exactly one complex number $z_a \in \gamma$ such that $z_a / |z_a| = e^{ \textbf{i} a}$. 

The following proposition is a reformulation of Theorem \ref{hdeterminant} and Theorem \ref{hdeterminantprocess}, specialized to the case $\zeta = - q^p$, but with more general, star-shaped contours. As above, the notation $K \in \{ V, A \}$ will be used to discuss the stochastic six-vertex model and ASEP simultaneously. 

\begin{prop}

\label{asymptoticheightvertex}

Fix $\delta_1, \delta_2 \in (0, 1)$; $R > L > 0$; $m \in \mathbb{Z}_{\ge 1}$; $b_1, b_2, \ldots , b_m \in (0, 1)$; $x \in \mathbb{Z}$; and $p \in \mathbb{R}$. Denote $\beta_j = b_j / (1 - b_j)$ for each $j \in [1, m]$. 

Let $\Gamma_V \subset \mathbb{C}$ be a positively oriented, star-shaped contour in the complex plane containing $0$, but leaving outside $-q \kappa$ and all $q \beta_j$. Let $\mathcal{C}_V \subset \mathbb{C}$ be a positively oriented, star-shaped contour contained inside $q^{-1} \Gamma_V$; that contains $0$, $-q$, and $\Gamma_V$; but that leaves outside all $q \beta_j$. Let $\Gamma_A \subset \mathbb{C}$ be a positively oriented, star-shaped contour containing $0$, but leaving outside $-q$ and all $q \beta_j$. Furthermore, let $\mathcal{C}_A \subset \mathbb{C}$ be a positively oriented, star-shaped contour contained inside $q^{-1} \Gamma_A$; that contains $0$, $-q$, and $\Gamma_A$; but that leaves outside all $q \beta_j$.

Let $\mathbb{E}_V$ denote the expectation with respect to the stochastic six-vertex model with left jump probability $\delta_1$, right jump probability $\delta_2$, and step $(b_1, b_2, \ldots , b_m)$-Bernoulli initial data. Similarly, let $\mathbb{E}_A$ denote the expectation with respect to the ASEP with left jump rate $L$, right jump rate $R$, and step $(b_1, b_2, \ldots , b_m)$-Bernoulli initial data. 

Then, for each $K \in \{V, A \}$, we have that 
\begin{flalign}
\label{limitheightmodel}
\mathbb{E}_K \left[ \displaystyle\frac{1}{\big( -q^{\mathfrak{h}_T (x) + p}; q \big)_{\infty}} \right] = \det \big( \Id + K^{(p)} \big)_{L^2 (\mathcal{C}_K)},
\end{flalign}

\noindent where 
\begin{flalign}
\label{kvw1}
K^{(p)} (w, w') & = \displaystyle\frac{1}{2 \textbf{\emph{i}} \log q} \displaystyle\sum_{j = -\infty}^{\infty} \displaystyle\oint_{\Gamma_K} \displaystyle\frac{g_K (w; x, T)}{g_K (v; x, t)} \displaystyle\frac{v^{p - 1} w^{-p}}{\sin \big( \frac{\pi}{\log q} ( \log v - \log w + 2 \pi \textbf{\emph{i}} j ) \big)} \displaystyle\frac{dv}{ w' - v },
\end{flalign}

\noindent and $g_V$ and $g_A$ are defined in \eqref{gv} and \eqref{hdeterminantidentity}, respectively. 
\end{prop}

\begin{proof}

In view of \eqref{kvw} and the fact that the second sum on the right side of this identity is equal to $0$, there are two differences between \eqref{hboth} and \eqref{limitheightmodel}. The first is that the contour along which the Fredholm determinants are taken is different; in the former it is $\mathscr{C}_K$ and in the latter it is $\mathcal{C}_K$. The second is that the contour of integration for $v$ in \eqref{kvw} is different from that in \eqref{kvw1}; in the former it is $Q_{R, d, \delta; w}$ and in the latter it is $\Gamma_K$. Thus, it suffices (see, for instance, Remark \ref{determinantcontourdeform}) to show that we can continuously and simultaneously deform $Q_{R, d, \delta; w}$ and $\mathscr{C}_K$ to $\Gamma_K$ and $\mathcal{C}_K$, respectively, without crossing any poles of any integrand on the right side of \eqref{kvw1}. 

We will only do this in the case $K = V$; the case $K = A$ is entirely analogous. To that end, observe that the poles for $v$ of the integrands in the first sum on the right side of \eqref{kvw} are at $v \in \{ 0, w', -q \kappa \} \cup \{ q^j w \}_{j \in \mathbb{Z}}$; of these, the poles at $v \in \{ 0, q w, q^2 w, \ldots \}$ are contained inside the contour $Q_{R, d, \delta; w}$ and the others are left outside. The poles for $w$ of the same integrands are at $w \in \{ -q, 0, q \beta_1, q \beta_2, \ldots , q \beta_m \}$; of these, the poles at $w \in \{ 0, -q \}$ are contained inside the contour $\mathscr{C}_V$ and the others are left outside. 

Now, recall that $\mathcal{C}_V$ contains $0$, $-q$, and $\Gamma_V$, but leaves outside all $q \beta_j$; similarly, $\Gamma_V$ contains $0$ and $q \mathcal{C}_V$, but leaves outside $-q \kappa$. Thus, $\mathcal{C}_V$ and $\mathscr{C}_V$ contain and leave outside the same poles of $K^{(p)}$; similarly, $Q_{R, d, \delta; w}$ and $\Gamma_V$ contain and leave outside the same poles of the integrand on the right side of \eqref{kvw1}. Therefore, since $\mathcal{C}_V$ and $\Gamma_V$ are star-shaped, it follows by moving along rays through the origin that $Q_{R, d, \delta; w}$ and $\mathscr{C}_V$ can be continuously and simultaneously deformed to $\Gamma_V$ and $\mathcal{C}_V$, respectively, without crossing any poles of any integrand on the right side of \eqref{kvw1}. Thus, \eqref{limitheightmodel} follows from \eqref{hboth} and \eqref{kvw}. 
\end{proof}

Observe that $\big( -q^{\mathfrak{h}_T (x) + p}; q \big)_{\infty}^{-1}$ from \eqref{limitheightmodel} tends to $0$ when $p$ is much less than than $-\mathfrak{h}_T (x)$, and that it tends to $1$ when $p$ is much greater than than $-\mathfrak{h}_T (x)$. Therefore, one might expect the left side of \eqref{limitheightmodel} to approximate the quantity $\mathbb{E}_{\mathbb{P}^K} [\textbf{1}_{\mathfrak{h}_T (x) \ge -p}] = \mathbb{P}^K [\mathfrak{h}_T (x) \ge -p]$. Hence, in order to analyze this probability, one requires asymptotics of the determinant on the right side of \eqref{limitheightmodel}. 

In the next section, we will clarify this heuristic and explain how to reduce Theorem \ref{asymmetriclimit} and Theorem \ref{hlimit} to an asymptotic analysis of Fredholm determinants of the type given on the right side of \eqref{limitheightmodel}.

\subsection{Applications of Proposition \ref{asymptoticheightvertex}}

\label{ApplicationsDeterminantModelProcess}

The following two results provide the large $T$ asymptotics for the determinant on the right side of \eqref{limitheightmodel}. Proposition \ref{processdeterminant} addresses the case $K = A$ (ASEP), and Proposition \ref{modeldeterminant} addresses the case $K = V$ (stochastic six-vertex model). In what follows, we recall that the integer $x$ appears in the definition of the function $g_K$ (of \eqref{gv} and \eqref{hdeterminantidentity}) that is used to define the kernel $K^{(p)}$ (of \eqref{kvw1}). 

\begin{prop}

\label{processdeterminant}

Adopt the notation of Theorem \ref{asymmetriclimit} and Proposition \ref{asymptoticheightvertex}. 

\begin{enumerate}

\item{Let $\eta \in (\theta, 1)$. For each positive real number $T$, let $x = x(T) = \lfloor \eta T \rfloor + 1 $ and $p_T = s f_{\eta} T^{1 / 3} - m_{\eta} T$. We have that 
\begin{flalign*}
\displaystyle\lim_{T \rightarrow \infty} \det \big( \Id + A_{\zeta}^{(p_T)} \big)_{L^2 (\mathcal{C}_A)} = F_{\TW} (s). 
\end{flalign*}}

\item{Assume that $\{ \eta = \eta_T \}_{T \in \mathbb{Z}_{> 0}}$ is a sequence of real numbers such that $\lim_{T \rightarrow \infty} T^{1 / 3} \big( \eta_T - \theta \big) = d$.  For each positive real number $T$, let $x = x(T) = \lfloor \eta T \rfloor + 1$ and $p_T = s f_{\eta} T^{1 / 3} - m_{\eta} T $. We have that 
\begin{flalign*}
\displaystyle\lim_{T \rightarrow \infty} \det \big( \Id + A_{\zeta}^{(p_T)}  \big)_{L^2 (\mathcal{C}_A)} = F_{\BBP; \textbf{\emph{c}}} (s). 
\end{flalign*}}

\item{Let $\eta \in (-b, \theta)$, and assume that $b_j = b$ for all indices $j \in [1, m]$. For each positive real number $T$, let $x = x(T) = \lfloor \eta T \rfloor + 1$ and $p_T = s f_{\eta}' T^{1 / 3} - m_{\eta}' T$. We have that 
\begin{flalign*}
\displaystyle\lim_{T \rightarrow \infty} \det \big( \Id + A_{\zeta}^{(p_T)} \big)_{L^2 (\mathcal{C}_A)} = G_m (s). 
\end{flalign*}}

\end{enumerate}

\end{prop}

\begin{prop}

\label{modeldeterminant}

Recall the definitions of $\chi$, $\kappa$, $\Lambda$, and $\theta$ given in \eqref{processlocationtransition}, and also the definition of $\textbf{\emph{c}} = (c_1, c_2, \ldots , c_m)$ given in \eqref{cmodelphasetransition}. Adopt the notation of Proposition \ref{asymptoticheightvertex}, and further define
\begin{flalign}
\label{modelfluctuationslargeetadeterminant}
m_{\eta} = \displaystyle\frac{\big( \sqrt{\kappa} - \sqrt{\eta} \big)^2}{\kappa - 1}; \qquad f_{\eta} = \displaystyle\frac{(\sqrt{\kappa \eta} - 1)^{2 / 3} (\kappa - \sqrt{\kappa \eta})^{2 / 3} }{(\kappa - 1) \kappa^{1 / 6} \eta^{1 / 6}}, 
\end{flalign}
 
\noindent and 
\begin{flalign}
\label{modelfluctuationslargeetaleftdeterminant}
m_{\eta}' & = b  - \Lambda^{-1} b \eta ; \qquad f_{\eta}' = \chi^{1 / 2} (1 - \theta^{-1} \eta)^{1 / 2}. 
\end{flalign}

\begin{enumerate}

\item{Let $\eta \in (\theta, \kappa)$. For each positive integer $T$, let $x = x(T) = \lfloor \eta T \rfloor + 1$ and $p_T = s f_{\eta} T^{1 / 3} - m_{\eta} T$. We have that 
\begin{flalign*}
\displaystyle\lim_{T \rightarrow \infty} \det \big( \Id + V_{\zeta}^{(p_T)} \big)_{L^2 (\mathcal{C}_V)} = F_{\TW} (s). 
\end{flalign*}}

\item{Assume that $\{ \eta = \eta_T \}_{T \in \mathbb{Z}_{> 0}}$ is a sequence of real numbers such that $\lim_{T \rightarrow \infty} T^{1 / 3} \big( \eta_T - \theta \big) = d$. For each positive integer $T$, let $x = x(T) = \lfloor \eta T \rfloor + 1$ and $p_T = s f_{\eta} T^{1 / 3} - m_{\eta} T $. We have that 
\begin{flalign*}
\displaystyle\lim_{T \rightarrow \infty} \det \big( \Id + V_{\zeta}^{(p_T)} \big)_{L^2 (\mathcal{C}_V)} = F_{\BBP; \textbf{\emph{c}}} (s). 
\end{flalign*}}

\item{Let $\eta \in (\Lambda^{-1} \theta, \theta)$, and assume that $b_j = b$ for all indices $j \in [1, m]$. For each positive integer $T$, let $x = x(T) = \lfloor \eta T \rfloor + 1$ and $p_T = s f_{\eta}' T^{1 / 3} - m_{\eta}' T$. We have that 
\begin{flalign*}
\displaystyle\lim_{T \rightarrow \infty} \det \big( \Id + V_{\zeta}^{(p_T)} \big)_{L^2 (\mathcal{C}_V)} = G_m (s). 
\end{flalign*}}

\end{enumerate}

\end{prop}

Now let us explain how to deduce Theorem \ref{asymmetriclimit} and Theorem \ref{hlimit} from Proposition \ref{processdeterminant} and Proposition \ref{modeldeterminant}, respectively. Both of these reductions use the following lemma, which is Lemma 4.1.39 of \cite{MP}. 

\begin{lem}[{\cite[Lemma 4.1.39]{MP}}] 

\label{convergenceprobabilityexpectationfunction}

Let $g_1, g_2, \ldots : \mathbb{R} \rightarrow [0, 1]$ be decreasing functions, such that $\lim_{x \rightarrow -\infty} g_j (x) = 1$ and $\lim_{x \rightarrow \infty} g_j (x) = 0$ for each positive integer $j$. Assume moreover that, for each $\varepsilon > 0$, the $\{ g_j (x) \}_{j \in \mathbb{Z}_{> 0}}$ converge uniformly to $\textbf{\emph{1}}_{x \le 0}$ on $x \in \mathbb{R} \setminus[-\varepsilon, \varepsilon]$, as $j$ tends to $\infty$. 

Suppose that $F$ is a continuous probability distribution and that $X_1, X_2, \ldots $ are random variables such that $\lim_{n \rightarrow \infty} \mathbb{E} \big[ g_n (X_n - s) \big] = F(s)$ for each real number $s \in \mathbb{R}$. Then, $\lim_{n \rightarrow \infty} \mathbb{P} [X_n \le s] = F (s)$ for each $s \in \mathbb{R}$. 
\end{lem}

\noindent Using Lemma \ref{convergenceprobabilityexpectationfunction}, we can establish Theorem \ref{asymmetriclimit} from Proposition \ref{processdeterminant}. 

\begin{proof}[Proof of Theorem \ref{asymmetriclimit} Assuming Proposition \ref{processdeterminant}] 

Let $g_N (x) = (- q^{- N^{1 / 3} f_{\eta} x}; q)_{\infty}^{-1}$, for each $N > 0$, and suppose that $\eta \in (\theta, 1)$. From Proposition \ref{asymptoticheightvertex} and the first part of Proposition \ref{processdeterminant}, it follows that 
\begin{flalign}
\label{hsprocess}
\displaystyle\lim_{T \rightarrow \infty} \mathbb{E} \bigg[ g_T \Big( T^{-1 / 3} f_{\eta}^{-1} \big( m_{\eta} T - \mathfrak{h}_t ( \lfloor \eta T \rfloor + 1 ) \big) - s \Big) \bigg] & = \displaystyle\lim_{T \rightarrow \infty} \det \big( \Id + A_{\zeta}^{(p_T)} \big) = F_{\TW} (s). 
\end{flalign}

\noindent Now, observe that the sequence of functions $\{ g_N \}_{N \ge 0}$ satisfies the conditions of Lemma \ref{convergenceprobabilityexpectationfunction}. Applying this lemma with $X_T = T^{-1 / 3} f_{\eta}^{-1} \big( m_{\eta} T - \mathfrak{h}_t ( \lfloor \eta T \rfloor + 1) \big)$ and using \eqref{hsprocess}, we find that 
\begin{flalign*}
\displaystyle\lim_{T \rightarrow \infty} \mathbb{P} \left[ m_{\eta} T - \mathfrak{h}_t ( \lfloor \eta T \rfloor + 1 ) \le s f_{\eta} T^{1 / 3} \right] = F_{\TW} (s). 
\end{flalign*}

\noindent From this, we deduce the first part of Theorem \ref{asymmetriclimit}, since $\mathfrak{h}_T (\lfloor \eta T \rfloor + 1 ) = J_T (\lfloor \eta T \rfloor) = J_T (\eta T)$. 

The second and third parts follow similarly. 
\end{proof}

\noindent We omit the verification of Theorem \ref{hlimit} from Proposition \ref{modeldeterminant} since it follows in a very similar way, after setting $\eta = x / y$ and observing that the height function $\mathfrak{H} (xT, yT)$ from Theorem \ref{hlimit} satisfies $\mathfrak{H} (xT, yT) = \mathfrak{h}_{yT} (xT - 1)$ in Proposition \ref{modeldeterminant}.  

In the remaining sections, we establish Proposition \ref{processdeterminant} and Proposition \ref{modeldeterminant}.

\section{Tracy-Widom Fluctuations}

\label{RightKernel}

In this section we establish the first parts of Proposition \ref{processdeterminant} and Proposition \ref{modeldeterminant} through a saddle point analysis of the Fredholm determinant $\det \big( \Id + K_{\zeta}^{(p_T)} \big)_{L^2 (\mathcal{C}_K)}$, defined in Proposition \ref{asymptoticheightvertex}. Similar analyses have been performed in several previous papers; in particular, we will largely follow the reasoning presented in Section 5 of \cite{SSVM}, although we will be more detailed here. 

To implement this analysis, we will first select the contours $\mathcal{C}_K$ and $\Gamma_K$ from Proposition \ref{asymptoticheightvertex} in a specific way, so as to ensure that the kernel $K_{\zeta}^{(p_T)} (w, w')$ decays exponentially (in $T$) for all $w, w' \in \mathcal{C}_K$ away from a certain point $\psi \in \mathcal{C}_K$; this choice of contours will be made in Section \ref{ckgammakcontours}. Once the contours have been fixed, we will be able to proceed with the asymptotic analysis through the Laplace method; this will be done in Section \ref{VertexRight}.

\subsection{Choosing the Contours \texorpdfstring{$\mathcal{C}_K$}{} and \texorpdfstring{$\Gamma_K$}{}}

\label{ckgammakcontours}

In this section we exhibit one choice of contours $\mathcal{C}_K$ and $\Gamma_K$ that will later lead to proofs of the first parts of Proposition \ref{processdeterminant} and Proposition \ref{modeldeterminant}. The choice of these contours will depend on whether $K = A$ or $K = V$. In Section \ref{ContoursVertexRight} we consider the case $K = V$, and in Section \ref{ContoursProcessRight} we consider the case $K = A$. In both cases, our definitions are based on the same observation, which we now describe.

\subsubsection{An Exponential Form for the Kernel \texorpdfstring{$K_{\zeta}^{(p_T)}$}{}}

\label{KExponentialRight}

For notational convenience, we set $x(T) = \eta T + 1$, as opposed to $x(T) = \lfloor \eta T \rfloor + 1$ (which is what was stated in Proposition \ref{processdeterminant} and Proposition \ref{modeldeterminant}); this will not affect the asymptotics. 

Now, let us rewrite the kernel $K (w, w') = K^{(p_T)} (w, w')$ given in \eqref{kvw1} through the identity 
\begin{flalign}
\begin{aligned}
\label{vptright} 
K (w, w') = \displaystyle\frac{1}{2 \textbf{i} \log q} \displaystyle\sum_{j \in \mathbb{Z}} & \displaystyle\oint_{\Gamma_V} \displaystyle\frac{ \exp \left( T \big( G_K (w) - G_K (v) \big) \right) }{\sin \big( \pi (\log q)^{-1} (2 \pi \textbf{i} j + \log v - \log w) \big) } \displaystyle\prod_{k = 1}^m \displaystyle\frac{(q^{-1} \beta_k^{-1} v; q)_{\infty}}{(q^{-1} \beta_k^{-1} w; q)_{\infty}} \\
& \qquad \times \left( \displaystyle\frac{v}{w} \right)^{s f_{\eta; K} T^{1 / 3}} \displaystyle\frac{dv}{v (w' - v)} , 
\end{aligned}
\end{flalign}

\noindent where $G_K (z)$ and $f_{\eta; K}$ depend on $K$. Explicitly, when $K = A$, 
\begin{flalign}
\label{ga13}
G_A (z) = \displaystyle\frac{q}{z + q} + \eta \log (z + q) + m_{\eta; A} \log z; \qquad f_{\eta; A} = \left( \displaystyle\frac{1 - \eta^2}{4} \right)^{2 / 3}, 
\end{flalign}

\noindent and when $K = V$,
\begin{flalign}
\label{gv13}
G_V (z) = \eta \log (\kappa^{-1} z + q) - \log (z + q) + m_{\eta; V} \log z; \quad f_{\eta; V} = \displaystyle\frac{(\sqrt{\kappa \eta} - 1)^{2 / 3} (\kappa - \sqrt{\kappa \eta})^{2 / 3} }{(\kappa - 1) \kappa^{1 / 6} \eta^{1 / 6}}. 
\end{flalign}

\noindent Here, 
\begin{flalign}
\label{m13}
m_{\eta; A} = \left( \displaystyle\frac{1 - \eta}{2} \right)^2; \qquad m_{\eta; V} = \displaystyle\frac{\big( \sqrt{\kappa} - \sqrt{\eta} \big)^2}{\kappa - 1}. 
\end{flalign}

Let us now locate the critical points of $G_K$; these will again depend on whether $K = A$ or $K = V$. We find that 	
\begin{flalign}
\begin{aligned}
\label{derivativega} 
\qquad \quad G_A' (z) = \left( \displaystyle\frac{\eta + 1}{2} \right)^2 \displaystyle\frac{(z - \psi_A)^2 }{z (z + q)^2},  \qquad \text{with} \qquad \psi_A =  \displaystyle\frac{q (1 - \eta)}{1 + \eta}; 
\end{aligned}
\end{flalign}

\begin{flalign}
\begin{aligned}
\label{derivativegv} 
G_V' (z) =  \displaystyle\frac{(\sqrt{\kappa \eta} - 1)^2}{\kappa - 1}  \displaystyle\frac{(z - \psi_V)^2 }{z (z + q \kappa) (z + q)}, \qquad \text{with} \qquad \psi_V = \displaystyle\frac{q (\kappa - \sqrt{\kappa \eta})}{\sqrt{\kappa \eta} - 1}. 
\end{aligned} 
\end{flalign} 

\noindent Thus, in both cases $K \in \{ V, A \}$, $\psi_K$ is a critical point of $G_K$. Furthermore, $G_K'' (\psi_K) = 0$ and 
\begin{flalign*}
\displaystyle\frac{G_A''' (\psi_K)}{2} = \displaystyle\frac{(\eta + 1)^5}{16 q^3 (1 - \eta)} = \left( \displaystyle\frac{f_{\eta; A}}{\psi_K } \right)^3; \qquad \displaystyle\frac{G_V''' (\psi_K)}{2} = \displaystyle\frac{(\sqrt{\kappa \eta} - 1)^5}{q^3 (\kappa - 1)^3 (\kappa - \sqrt{\kappa \eta}) \sqrt{\kappa \eta}} = \left( \displaystyle\frac{f_{\eta; V}}{\psi_K } \right)^3. 
\end{flalign*}

\noindent Thus, for both $K \in \{ V, A \}$, we have that $G_K''' (\psi_K) = \big( \psi_K^{-1} f_{\eta; K} \big)^3$, which implies from a Taylor expansion that	
\begin{flalign}
\begin{aligned}
\label{gzpsi}
G_K (z) - G_K (\psi_K) & = \displaystyle\frac{1}{3} \left( \displaystyle\frac{f_{\eta; K} (z - \psi_K)}{\psi_K} \right)^3 + R_K \left( \displaystyle\frac{f_{\eta; K} (z - \psi_K)}{\psi_K} \right) \\
& = \displaystyle\frac{1}{3} \left( \displaystyle\frac{f_{\eta; K} (z - \psi_K)}{\psi_K} \right)^3 + \mathcal{O} \big( |z - \psi_K|^4 \big), \qquad \qquad \text{as} \quad |z - \psi_K| \rightarrow 0. 
\end{aligned}
\end{flalign}

\noindent where 
\begin{flalign}
\label{r13}
R_K \left( \displaystyle\frac{f_{\eta; K} (z - \psi_K)}{\psi_K} \right) = G_K (z) - G_K (\psi_K) - \displaystyle\frac{1}{3} \left( \displaystyle\frac{f_{\eta; K} (z - \psi_K)}{\psi_K} \right)^3 . 
\end{flalign} 

We would now like to select contours $\mathcal{C}_K$ and $\Gamma_K$ such that $\Re \big( G_K (w) - G_K (v) \big) < 0$, for all $w \in \mathcal{C}_K$ and $v \in \Gamma_K$ away from $\psi_K$, subject to the restrictions stated in Proposition \ref{asymptoticheightvertex}. If we are able to do this, then the integrand on the right side of \eqref{vptright} will decay exponentially away from $\psi_K$. This will allow us to localize the integral around the critical point $\psi_K$ and use the estimate \eqref{gzpsi} to simplify the asymptotics.

\subsubsection{Choice of Contours \texorpdfstring{$\mathcal{C}_V$}{} and \texorpdfstring{$\Gamma_V$}{}}

\label{ContoursVertexRight}

In this section we show how to choose contours $\mathcal{C}_V$ and $\Gamma_V$ satisfying the conditions of Proposition \ref{asymptoticheightvertex} such that $\Re \big( G_V (w) - G_V (v) \big) < 0$ for $w \in \mathcal{C}_V$ and $v \in \Gamma_V$ both away from $\psi_V$. In what follows, we omit the subscript $V$ to simplify notation. 

Our choice for $\mathcal{C}$ and $\Gamma$ will be similar to the contours selected in Section 5 of \cite{SSVM}, in which $\mathcal{C}$ and $\Gamma$ are chosen to follow level lines of the equation $\Re G(z) = G(\psi)$. 

That work required some properties of these level lines that are summarized in the following proposition; see Section 5.1 of \cite{SSVM} and Figure \ref{l1l2l3}.

\begin{prop}[{\cite[Section 5.1]{SSVM}}]

\label{linesgv}

There exist three simple, closed curves, $\mathcal{L}_1 = \mathcal{L}_{1; V} $, $\mathcal{L}_2 = \mathcal{L}_{2; V} $, and $\mathcal{L}_3 = \mathcal{L}_{3; V}$, that all pass through $\psi$ and satisfy the following properties. 

\begin{enumerate}

\item{ \label{vzpsij} Any $z \in \mathbb{C}$ satisfying $\Re G(z) = G(\psi)$ lies on $\mathcal{L}_j$ for some $j \in \{1, 2, 3 \}$. }

\item{ \label{vzpsiall} Any complex number $z \in \mathcal{L}_1 \cup \mathcal{L}_2 \cup \mathcal{L}_3$ satisfies $\Re G(z) = G(\psi)$. }

\item{ \label{anglev} The level lines $\mathcal{L}_1$, $\mathcal{L}_2$, and $\mathcal{L}_3$ are all star-shaped. }

\item{ \label{containmentv} We have that $\mathcal{L}_1 \cap \mathcal{L}_2 = \mathcal{L}_2 \cap \mathcal{L}_3 = \mathcal{L}_1 \cap \mathcal{L}_3 = \{ \psi \}$. Furthermore, $\mathcal{L}_1 \setminus \{ \psi \}$ is contained in the interior of $\mathcal{L}_2$, and $\mathcal{L}_2 \setminus \{ \psi \}$ is contained in the interior of $\mathcal{L}_3$. }

\item{ \label{qkappaq0v} The interior of $\mathcal{L}_1$ contains $0$, but not $-q$; the interior of $\mathcal{L}_2$ contains $0$ and $-q$, but not $- q \kappa$; and the interior of $\mathcal{L}_3$ contains $0$, $-q$, and $-q \kappa$. }

\item{ \label{psianglesv} The level line $\mathcal{L}_1$ meets the positive real axis (at $\psi$) at angles $5 \pi / 6$ and $- 5 \pi /6$; the level line $\mathcal{L}_2$ meets the positive real axis (at $\psi$) at angles $\pi / 2$ and $- \pi / 2$; and the level line $\mathcal{L}_3$ meets the positive real axis (at $\psi$) at angles $\pi / 6$ and $- \pi /6$. }

\item{ \label{positiverealv} For all $z$ in the interior of $\mathcal{L}_2$ but strictly outside of $\mathcal{L}_1$, we have that $\Re \big( G(z) - G(\psi) \big) > 0$. }

\item{ \label{negativerealv} For all $z$ in the interior of $\mathcal{L}_3$ but strictly outside of $\mathcal{L}_2$, we have that $\Re \big( G(z) - G (\psi) \big) < 0$. }

\end{enumerate} 

\end{prop}

\begin{figure}

\begin{minipage}{0.45\linewidth}

\centering

\begin{tikzpicture}[
      >=stealth,
      auto,
      style={
        scale = 1
      }
			]

			\draw[<->, black	] (0, -3.5) -- (0, 3.5) node[black, above = 0] {$\Im z$};
			\draw[<->, black] (-4, 0) -- (2, 0) node[black, right = 0] {$\Re z$};
			
			\draw[->,black, thick] (1.1, 0) -- (1.3, .34641);
			\draw[-,black, thick] (1.1, 0) -- (1.3, -.34641);
			\draw[black, thick] (1.3, .34641) arc (15.4:180:1.34536) node [black, above = 42, right = 10 	] {$\mathcal{C}_V$}; 
			\draw[black, thick] (1.3, -.34641) arc (-15.4:-180:1.34536);
			
			\draw[->, black, thick] (1.02, 0) -- (.82, .34641); 
			\draw[-, black, thick] (1.02, 0) -- (.82, -.34641); 
			\draw[black, thick] (.82, .34641) arc (21.67:180:.9)  node [black, above = 12, right = 7 	] {$\Gamma_V$}; 
			\draw[black, thick] (.82, -.34641) arc (-21.67:-180:.9);

			\path[draw, dashed] (1.1, 0) -- (1.08, .02) -- (1.06, .03) -- (1.03, .04) -- (1, .05) -- (.98, .06) -- (.95, .08) -- (.93, .09) -- (.89, .1) -- (.86, .12) -- (.83, .13) -- (.8, .14) -- (.75, .16) -- (.71, .17) -- (.65, .19) -- (.61, .2) -- (.56, .21) -- (.5, .22) -- (.47, .23) -- (.43, .24) -- (.38, .24) -- (.35, .24) -- (.3, .24) -- (.27, .24) -- (.22, .24) -- (.18, .23) -- (.14, .22) -- (.08, .21) -- (.06, .19) -- (.04, .18) -- (0, .17) -- (-.02, .15) -- (-.04, .13) -- (-.06, .12) -- (-.07, .1) -- (-.08, .09) -- (-.08, .08) -- (-.09, .07) -- (-.1, .04) -- (-.11, .03) -- (-.11, 0);
			
			\path[draw, dashed] (1.1, 0) -- (1.08, -.02) -- (1.06, -.03) -- (1.03, -.04) -- (1, -.05) -- (.98, -.06) -- (.95, -.08) -- (.93, -.09) -- (.89, -.1) -- (.86, -.12) -- (.83, -.13) -- (.8, -.14) -- (.75, -.16) -- (.71, -.17) -- (.65, -.19) -- (.61, -.2) -- (.56, -.21) -- (.5, -.22) -- (.47, -.23) -- (.43, -.24) -- (.38, -.24) -- (.35, -.24) -- (.3, -.24) -- (.27, -.24) -- (.22, -.24) -- (.18, -.23) -- (.14, -.22) -- (.08, -.21) -- (.06, -.19) -- (.04, -.18) -- (0, -.17) -- (-.02, -.15) -- (-.04, -.13) -- (-.06, -.12) -- (-.07, -.1) -- (-.08, -.09) -- (-.08, -.08) -- (-.09, -.07) -- (-.1, -.04) -- (-.11, -.03) -- (-.11, 0);
			
			\draw[black, dashed] (0, 0) circle [radius=1.1];
			
			\path[draw, dashed] (1.1, 0) -- (1.23, .1) -- (1.4, .3) -- (1.57, .71) -- (1.6, 1) -- (1.57, 1.3) -- (1.4, 1.76) -- (1.23, 2.02) -- (1.07, 2.22) -- (.9, 2.37) -- (.74, 2.49) -- (.57, 2.58) -- (.4, 2.68) -- (.23, 2.75) -- (.07, 2.8) -- (-.1, 2.85) -- (-.27, 2.87) -- (-.43, 2.88) -- (-.6, 2.88) -- (-.77, 2.87) -- (-.94, 2.86) -- (-1.1, 2.83) -- (-1.27, 2.78) -- (-1.44, 2.74) -- (-1.6, 2.67) -- (-1.77, 2.59) -- (-1.94, 2.51) -- (-2.1, 2.39) -- (-2.27, 2.26) -- (-2.44, 2.12) -- (-2.6, 1.94) -- (-2.77, 1.7) -- (-2.94, 1.47) -- (-3.1, 1.13) -- (-3.27, .57) -- (-3.32, 0); 
			
			\path[draw, dashed] (1.1, 0) -- (1.23, -.1) -- (1.4, -.3) -- (1.57, -.71) -- (1.6, -1) -- (1.57, -1.3) -- (1.4, -1.76) -- (1.23, -2.02) -- (1.07, -2.22) -- (.9, -2.37) -- (.74, -2.49) -- (.57, -2.58) -- (.4, -2.68) -- (.23, -2.75) -- (.07, -2.8) -- (-.1, -2.85) -- (-.27, -2.87) -- (-.43, -2.88) -- (-.6, -2.88) -- (-.77, -2.87) -- (-.94, -2.86) -- (-1.1, -2.83) -- (-1.27, -2.78) -- (-1.44, -2.74) -- (-1.6, -2.67) -- (-1.77, -2.59) -- (-1.94, -2.51) -- (-2.1, -2.39) -- (-2.27, -2.26) -- (-2.44, -2.12) -- (-2.6, -1.94) -- (-2.77, -1.7) -- (-2.94, -1.47) -- (-3.1, -1.13) -- (-3.27, -.57) -- (-3.32, 0); 
			
			\filldraw[fill=black, draw=black] (-.5, 0) circle [radius=.03] node [black, below = 0] {$-q$};
			
			\filldraw[fill=black, draw=black] (-2, 0) circle [radius=.03] node [black, below = 0] {$-q \kappa$};
			
			\filldraw[fill=black, draw=black] (1.6, 0) circle [radius=.03];
			
			\filldraw[fill=black, draw=black] (1.7, 0) circle [radius=.03] node [black, right = 4, below = 0] {$q \beta_j$};
			
			\filldraw[fill=black, draw=black] (1.8, 0) circle [radius=.03];

\end{tikzpicture}

\end{minipage}
\qquad
\begin{minipage}{0.45\linewidth}

\centering

\begin{tikzpicture}[
      >=stealth,
			scale=.4
			]
			
			\draw[<->] (-4, 5) -- (4, 5);
			\draw[<->] (0, 1) -- (0, 9);
			
			\draw[->,black,very thick] (2, 1.535) -- (1, 3.268);
			\draw[-,black,very thick] (1, 3.268) -- (0, 5);
			\draw[-,black,very thick]  (1, 6.732) -- (2, 8.465) node [black, above = 0] {$\mathfrak{W}_{0, \infty}$};
			\draw[->,black,very thick] (0, 5) -- (1, 6.732);
			
			\draw[->,black,very thick] (-3, 1.535) -- (-2, 3.268);
			\draw[-,black,very thick] (-1, 5) -- (-2, 3.268);
			\draw[-,black,very thick] (-2, 6.732) -- (-3, 8.465) node [black, above = 0] {$\mathfrak{V}_{-1, \infty}$};
			\draw[->,black,very thick] (-1, 5) -- (-2, 6.732);
\end{tikzpicture}

\end{minipage}

\caption{\label{l1l2l3} Above and to the left, the three level lines $\mathcal{L}_1$, $\mathcal{L}_2$, and $\mathcal{L}_3$ are depicted as dashed curves; the contours $\Gamma_V$ and $\mathcal{C}_V$ are depicted as solid curves and are labeled. Above and to the right are the two contours $\mathfrak{W}_{0, \infty}$ and $\mathfrak{V}_{-1, \infty}$. }
\end{figure}
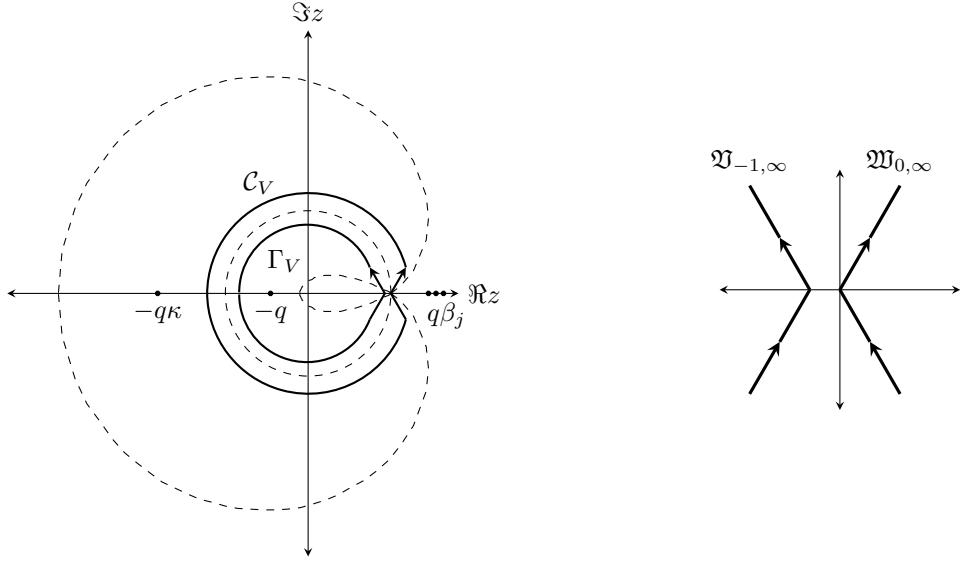

Now, let us explain how to select the contours $\mathcal{C}$ and $\Gamma$. They will each be the union of two contours, a ``small piecewise linear part'' near $\psi$, and a ``large curved part'' that closely follows the level line $\mathcal{L}_2$. 

Let us be more specific. The contours given by the following definition will be the linear parts of $\mathcal{C}$ and $\Gamma$.

\begin{definition}

\label{crgammar}

For a real number $r \in \mathbb{R}$ and a positive real number $\varepsilon > 0$ (possibly infinite), let $\mathfrak{W}_{r, \varepsilon}$ denote the piecewise linear curve in the complex plane that connects $r + \varepsilon e^{- \pi \textbf{i} / 3} $ to $r$ to $r + \varepsilon e^{\pi \textbf{i} / 3}$. Similarly, let $\mathfrak{V}_{r, \varepsilon}$ denote the piecewise linear curve in the complex plane that connects $r + \varepsilon e^{- 2 \pi \textbf{i} / 3}$ to $r$ to $r + \varepsilon^{2 \pi \textbf{i} / 3}$. 

\end{definition} 

\noindent Examples of these contours are given in Figure \ref{l1l2l3}. Now, let us make the following definitions; Definition \ref{linearvertexright} and Definition \ref{curvedvertexright} define the piecewise linear and curved parts of the contours $\mathcal{C}$ and $\Gamma$, respectively. Then, Definition \ref{vertexrightcontours} defines the contours $\mathcal{C}$ and $\Gamma$. 

\begin{definition}

\label{linearvertexright} 

Let $\mathcal{C}^{(1)} = \mathfrak{W}_{\psi, \varepsilon}$ and $\Gamma^{(1)} = \mathfrak{V}_{\psi - \psi f_{\eta}^{-1} T^{-1 / 3}, \varepsilon}$, where $\varepsilon$ is chosen to be sufficiently small (independent of $T$) so that the following four properties hold.

\begin{itemize}
\item{For all sufficiently large $T$, the two conjugate endpoints of $\mathcal{C}^{(1)}$ are strictly between the level lines $\mathcal{L}_2$ and $\mathcal{L}_3$, so that their distance from $\mathcal{L}_2$ and $\mathcal{L}_3$ is bounded away from $0$, independent of $T$.}

\item{For all sufficiently large $T$, the two conjugate endpoints of $\Gamma^{(1)}$ are strictly between $\mathcal{L}_1$ and $\mathcal{L}_2$, so that their distance from $\mathcal{L}_1$ and $\mathcal{L}_2$ is bounded away from $0$, independent of $T$.}

\item{We have that $|R\big( \psi^{-1} f_{\eta} (z - \psi) \big)| < |f_{\eta} (z - \psi) / 2 \psi|^3$, for all $z \in \mathcal{C}^{(1)} \cup \Gamma^{(1)}$, where $R$ was given in \eqref{r13}.}

\item{We have that $|v / w| \in (q^{1 / 2}, 1)$ for all $v \in \Gamma^{(1)}$ and $w \in \mathcal{C}^{(1)}$.}

\end{itemize}

\noindent Such a positive real number $\varepsilon$ is guaranteed to exist by part \ref{psianglesv} of Proposition \ref{linesgv} and the estimate \eqref{gzpsi}. 

\end{definition}

\begin{definition}

\label{curvedvertexright} 

Let $\mathcal{C}^{(2)}$ denote a positively oriented contour from the top endpoint $\psi + \varepsilon e^{\pi \textbf{i} / 3}$ of $\mathcal{C}^{(1)}$ to the bottom endpoint $\psi + \varepsilon e^{- \pi \textbf{i} / 3}$ of $\mathcal{C}^{(1)}$, and let $\Gamma^{(2)}$ denote a positively oriented contour from the top endpoint $\psi - \psi f_{\eta}^{-1} T^{-1 / 3} + \varepsilon e^{2 \pi \textbf{i} / 3}$ of $\Gamma^{(1)}$ to the bottom endpoint $\psi + \psi f_{\eta}^{-1} T^{-1 / 3} + \varepsilon e^{-2 \pi \textbf{i} / 3}$ of $\Gamma^{(1)}$, satisfying the following five properties. 

\begin{itemize} 

\item{The contour $\mathcal{C}^{(2)}$ remains strictly between the level lines $\mathcal{L}_2$ and $\mathcal{L}_3$, so that the distance from $\mathcal{C}^{(2)}$ to $\mathcal{L}_2$ and $\mathcal{L}_3$ remains bounded away from $0$, independent of $T$.}

\item{The contour $\Gamma^{(2)}$ remains strictly between the level lines $\mathcal{L}_1$ and $\mathcal{L}_2$, so that the distance from $\mathcal{C}^{(2)}$ to $\mathcal{L}_1$ and $\mathcal{L}_2$ remains bounded away from $0$, independent of $T$.}

\item{ The contour $\mathcal{C}^{(1)} \cup \mathcal{C}^{(2)}$ is star-shaped.}

\item{The contour $\Gamma^{(1)} \cup \Gamma^{(2)}$ is star-shaped and does not contain $-q \kappa$.}

\item{The contours $\mathcal{C}^{(2)}$ and $\Gamma^{(2)}$ are both sufficiently close to $\mathcal{L}_2$ so that the interior of $\Gamma^{(1)} \cup \Gamma^{(2)}$ contains the image of $\mathcal{C}^{(1)} \cup \mathcal{C}^{(2)}$ under multiplication by $q^{1 / 2}$. }

\end{itemize}

\noindent Such contours $\mathcal{C}^{(2)}$ and $\Gamma^{(2)}$ are guaranteed to exist by	 part \ref{anglev} and part \ref{qkappaq0v} of Proposition \ref{linesgv}. 

\end{definition} 

\begin{definition}

\label{vertexrightcontours}

Set $\mathcal{C} = \mathcal{C}^{(1)} \cup \mathcal{C}^{(2)}$ and $\Gamma = \Gamma^{(1)} \cup \Gamma^{(2)}$. 

\end{definition}

 Examples of the contours $\mathcal{C}$ and $\Gamma$ are depicted in Figure \ref{l1l2l3}. The following lemma states that $\mathcal{C}$ and $\Gamma$ satisfy their required conditions. 

\begin{lem}

\label{rightcvgammav}

The contour $\Gamma$ is positively oriented and star-shaped; it contains $0$, but leaves outside $-q \kappa$ and $q \beta_j$ for each $j \in [1, m]$. Furthermore, $\mathcal{C}$ is a positively oriented, star-shaped contour that is contained inside $q^{-1} \Gamma$; that contains $0$, $-q$ and $\Gamma$; but that leaves outside $q \beta_j$ for each $j \in [1, m]$. 
\end{lem}

We omit the proof of this lemma, since it can be quickly deduced from the definitions of $\mathcal{C}$ and $\Gamma$ and the fact that $q \beta_j > \psi$; the latter fact can be verified using the definition \eqref{derivativegv} of $\psi_K$ and the fact that $\eta > \theta$ (recall that $\theta$ was defined by \eqref{processlocationtransition}). 

In the next section we choose the contours $\Gamma_K$ and $\mathcal{C}_K$ in the case $K = A$.

\subsubsection{Choice of Contours \texorpdfstring{$\mathcal{C}_A$}{} and \texorpdfstring{$\Gamma_A$}{}}

\label{ContoursProcessRight}

Similar to in Section \ref{ContoursVertexRight}, our goal in this section is to choose contours $\mathcal{C}_A$ and $\Gamma_A$ satisfying the conditions of Proposition \ref{asymptoticheightvertex} such that $\Re \big( G_A (w) - G_A (v) \big) < 0$ for $w \in \mathcal{C}_A$ and $v \in \Gamma_A$ both away from $\psi_A$. In what follows, we omit the subscript $A$ to simplify notation. 

As in Section \ref{ContoursVertexRight}, we will take the contours $\Gamma$ and $\mathcal{C}$ to follow level lines of the equation $\Re G(z) = G(\psi)$. Again, we require some properties of these level lines that are summarized in the following proposition. Unlike Proposition \ref{linesgv}, this proposition does not seem to have appeared previously; we will establish it later, in Section \ref{PropositionLinesProcessRight}. 

\begin{prop}

\label{linesga}

There exist three simple, closed curves, $\mathcal{L}_1 = \mathcal{L}_{1; A} $, $\mathcal{L}_2 = \mathcal{L}_{2; A} $, and $\mathcal{L}_3 = \mathcal{L}_{3; A}$, that all pass through $\psi$ and satisfy properties \ref{vzpsij}, \ref{anglev}, \ref{containmentv}, \ref{psianglesv}, \ref{positiverealv}, and \ref{negativerealv} of Proposition \ref{linesgv}, with properties \ref{vzpsiall} and \ref{qkappaq0v} respectively replaced by the following. 

\begin{enumerate}

\item[2.]{ \label{azpsiall} Any complex number $z \in \mathcal{L}_1 \cup \mathcal{L}_2 \cup \mathcal{L}_3$, not equal to $q$, satisfies $\Re G(z) = G(\psi)$. }

\item[5.]{ \label{qkappaq0a} The interior of $\mathcal{L}_1$ contains $0$, but not $-q$; the interior of $\mathcal{L}_2$ contains $0$; and the interior of $\mathcal{L}_3$ contains $0$ and $-q$. Furthermore, $-q$ lies on the curve $\mathcal{L}_2$. }

\end{enumerate} 

\end{prop}

Thus, these contours are very similar to those depicted in Figure \ref{l1l2l3}, except that now $-q$ lies on $\mathcal{L}_2$ instead of in its interior (and $-q \kappa$ does not exist). 

Given Proposition \ref{linesga}, the selection of the contours $\mathcal{C}$ and $\Gamma$ will proceed as in Section \ref{ContoursVertexRight}. Again, these contours will each be be the union of two contours, a small piecewise linear part near $\psi$ and a large curved part that closely follows the level line $\mathcal{L}_2$. 

Specifically, define $\mathcal{C}^{(1)}$ and $\Gamma^{(1)}$ as in Definition \ref{linearvertexright}, where now $\psi_V$ is replaced with $\psi_A$, $f_{\eta; V}$ is replaced with $f_{\eta; A}$, $R_V$ is replaced by $R_A$, and $\mathcal{L}_{j; V}$ is replaced with $\mathcal{L}_{j; A}$ for each $j \in \{ 1, 2, 3 \}$. Similarly, define $\mathcal{C}^{(2)}$ and $\Gamma^{(2)}$, under the same replacements (and ignoring the requirement that $\Gamma^{(1)} \cup \Gamma^{(2)}$ must avoid $-q \kappa$, since the parameter $\kappa$ does not exist for the ASEP). As in Section \ref{ContoursVertexRight}, the existence of these contours is guaranteed by Proposition \ref{linesga} and the estimate \eqref{gzpsi}.  

Then, as in Definition \ref{vertexrightcontours}, define $\mathcal{C} = \mathcal{C}^{(1)} \cup \mathcal{C}^{(2)}$ and $\Gamma = \Gamma^{(1)} \cup \Gamma^{(2)}$.

The following lemma states that $\mathcal{C}$ and $\Gamma$ satisfy their required conditions; we omit its proof. 

\begin{lem}

\label{rightcagammaa}

The contour $\Gamma$ is positively oriented and star-shaped; it contains $0$, but leaves outside $-q$ and $q \beta_j$ for each $j \in [1, m]$. Furthermore, $\mathcal{C}$ is a positively oriented, star-shaped contour that is contained inside $q^{-1} \Gamma$; that contains $0$, $-q$ and $\Gamma$; but that leaves outside $q \beta_j$ for each $j \in [1, m]$. 
\end{lem}

\subsubsection{Proof of Proposition \ref{linesga}}

\label{PropositionLinesProcessRight}

The goal of this section is to establish Proposition \ref{linesga}, whose proof will be similar to that of Proposition \ref{linesgv} given in Section 5 of \cite{SSVM}. To do this, we require the following two lemmas, which provide some preliminary properties of the level lines of $\Re G_A (z) = G_A (\psi_A)$. As in Section \ref{ContoursProcessRight}, we assume that $K = A$ throughout and thus omit the subscript $A$ from notation. 

\begin{lem}

\label{realgarightbounded} 

The set of $z \in \mathbb{C}$ for which $\Re G(z) = G (\psi)$ is bounded.

\end{lem}

\begin{proof}
Observe that 
\begin{flalign}
\begin{aligned}
\label{largezgaright}
\Re G(z) & = \Re \left( \displaystyle\frac{q}{z + q} \right) + \eta \log |z + q| + m_{\eta} \log |z| = \left( \displaystyle\frac{1 + \eta}{2} \right)^2 \log |z| + \mathcal{O} \big( |z|^{-1} \big), 
\end{aligned}
\end{flalign}

\noindent where we have used the definition of $m_{\eta} = m_{\eta; A}$ given in \eqref{m13} to establish the estimate above. This implies that $\big| \Re G(z) \big|$ tends to $\infty$ as $|z|$ tends to $\infty$, so the set of $z$ satisfying $\Re G(z) = G (\psi)$ is bounded. 
\end{proof}

\begin{lem}

\label{sixzgazgapsi}

Let $\ell \subset \mathbb{C}$ be any line through $0$. There exist at most $6$ complex numbers $z \in \ell$ such that $\Re G(z) = G(\psi)$. Moreover, if $\ell$ is the real axis, then there exist only three such $z$; one is $z = \psi$, one is in the interval $(-q, 0)$, and one is in the interval $(-\infty, -q)$. 

\end{lem}

\begin{proof}

Let us establish the second statement first, so suppose that $\ell$ is the real axis. Then, from \eqref{derivativega}, we deduce that $G' (z) > 0$ for $z > 0$, and that $G' (z) < 0$ for $z < 0$. Due to the singularity of $G (z)$ at $z = - q$, this implies that there are at most three real numbers $z$ for which $\Re G(z) = G (\psi)$, and that at most one of them is positive; the one positive value of $z$ is $z = \psi$. 

Furthermore, observe that $G(z)$ tends to $\infty$ when $z$ tends to $-\infty$ (see \eqref{largezgaright}) or when $z$ tends to $-q$ from the right. Moreover, $G(z)$ tends to $-\infty$ when $z$ tends to $0$ or when $z$ tends to $-q$ from the left. By the continuity of $G$ on $(-\infty, -q) \cup (-q, 0)$, it follows that there exists one value of $z = z_1 \in (- q, 0)$ and one value of $z = z_2 \in (- \infty, - q)$ such that $\Re G(z_1) = G(\psi) = \Re G (z_2)$. This establishes the lemma in the case when $\ell = \mathbb{R}$.

Now assume that $\ell \ne \mathbb{R}$, and let $z_0 \in \ell$ be a non-zero complex number on $\ell$. It suffices to show that there exist at most six real numbers $\omega \in \mathbb{R}$ such that $\Re G(\omega z_0) = G (\psi)$. To that end, we differentiate $\Re G(\omega z_0)$ with respect to $\omega$ to find that 
\begin{flalign}
\begin{aligned}
\label{gaderivativeomegaright}
\displaystyle\frac{\partial}{\partial \omega} \big( \Re G (\omega z_0) - G (\psi) \big) & = \Re \Bigg( \displaystyle\frac{\partial}{\partial \omega} \bigg( \displaystyle\frac{q}{\omega z_0 + q} + \eta \log (\omega z_0 + q) + m_{\eta} \log (\omega z_0) \bigg) \Bigg)  \\
&= \left( \displaystyle\frac{(\eta + 1)^2}{4 \omega |\omega + q z_0^{-1}|^4} \right) \Re (\omega - z_0^{-1} \psi)^2 (\omega + q \overline{z_0^{-1}})^2. 
\end{aligned}
\end{flalign}

\noindent Fixing $z_0 \in \ell$, the right side of \eqref{gaderivativeomegaright} is a rational function in $\omega$ whose numerator is a polynomial of degree four; in particular, it can be zero for at most four distinct values of $\omega$. 

Furthermore, \eqref{gaderivativeomegaright} shows that the only real value of $\omega$ for which the derivative of $ \Re \big( G (\omega z_0) - G (\psi) \big)$ is singular is $\omega = 0$; this is due to the fact that $\omega + q z_0^{-1} \ne 0$, which holds since $\ell \ne \mathbb{R}$. Thus there are at most five real numbers $\omega$, at which the derivative $ \partial \big( \Re \big( G (\omega z_0) - G (\psi) \big) \big) / \partial \omega$ can change sign. It follows that there exist at most six real numbers $\omega$ for which $\Re G (\omega z_0) = G (\psi)$. 
\end{proof}

\noindent Now we can establish Proposition \ref{linesga}.

\begin{proof}[Proof of Proposition \ref{linesga}] 

Let $S = \mathbb{C} \setminus \{ \psi, -q \}$. Then $\Re G(z)$ and $\Re \big( G' (z) \big)$ are smooth on $S \setminus \{ 0 \}$. Furthermore, from the explicit form \eqref{derivativega} of $G' (z)$, we deduce that $G' (z)$ is non-zero on $S \setminus \{ 0 \}$, which implies that $\Re G (z)$ has no critical points on $S \setminus \{ 0 \}$. Thus, the implicit function theorem shows that the set of $z \in S$ satisfying $\Re G(z) = G (\psi)$ is a one-dimensional submanifold of $S$. In particular, it is a disjoint union of connected components that are each either homeomorphic to a line or a circle. 

By Lemma \ref{realgarightbounded}, all of these connected components are bounded. Therefore, the closure $\overline{\mathcal{M}}$ (in $\mathbb{C}$) of any such component $\mathcal{M}$ is either homeomorphic to a circle or has endpoints $-q$ and $\psi$ and is homeomorphic to a compact interval. If $\overline{\mathcal{M}}$ is homeomorphic to a circle, then the maximum principle for harmonic functions shows that one of the singularities $0$ or $-q$ of $\Re G(z)$ must lie on or in the interior of $\overline{\mathcal{M}}$ (since $\Re G(z)$ is harmonic for $z \in \mathbb{C} \setminus \{ -q, 0 \}$). 

Now, due to \eqref{gzpsi}, six (not necessarily distinct) connected components exit from $\psi$, at angles $\pi / 6$, $\pi / 2$, $5 \pi / 6$, $- 5 \pi / 6$, $- \pi / 2$, and $- \pi / 6$. Furthermore, due to the explicit form \eqref{ga13} of $G$, we deduce that two components also exit from $-q$ at angles $\pi / 2$ and $- \pi / 2$. Lemma \ref{sixzgazgapsi} also shows that there exist $z_1 \in (-q, 0)$ and $z_2 \in (-\infty, -q)$ such that $\Re G(z_1) = G(\psi) = \Re G(z_2)$; this implies that one of the components passes through $z_1$ and one of the components passes through $z_2$. By the second part of Lemma \ref{sixzgazgapsi}, we also find that there are no other real numbers $z \notin \{ z_1, z_2, \psi \}$ such that $\Re G(z) = G(\psi)$. 

These facts imply that there are four disjoint connected components $\mathcal{M}_1, \mathcal{M}_2, \mathcal{M}_2', \mathcal{M}_3 \subset S$ (see Figure \ref{m1m2m3}) such that the set of $z \in S$ satisfying $\Re G(z) = G (\psi)$ is equal to the union $\mathcal{M}_1 \cup \mathcal{M}_2 \cup \mathcal{M}_2' \cup \mathcal{M}_3$; these components can be described as follows. First, we have that $\psi \in \overline{\mathcal{M}_1}, \overline{\mathcal{M}_2}, \overline{\mathcal{M}_2'}, \overline{\mathcal{M}_3}$. Furthermore, $z_1 \in \mathcal{M}_1$, and $\overline{\mathcal{M}_1}$ is homeomorphic to a circle; similarly, $z_2 \in \mathcal{M}_3$, and $\overline{\mathcal{M}_3}$ is homeomorphic to a circle. 

Moreover, $\overline{\mathcal{M}_2}$ and $\overline{\mathcal{M}_2'}$ are both homeomorphic to compact intervals; we also have that $-q \in \overline{\mathcal{M}_2} \cap \overline{\mathcal{M}_2'}$, that $\mathcal{M}_2$ is in the upper half plane, and that $\mathcal{M}_2'$ is in the lower half plane. We can describe the intersections $\overline{\mathcal{M}_1} \cap \mathbb{R} = \{ z_1, \psi \}$, $\overline{\mathcal{M}_2} \cap \mathbb{R} = \{ -q, \psi \} = \overline{\mathcal{M}_2'} \cap \mathbb{R}$, and $\overline{\mathcal{M}_3} \cap \mathbb{R} = \{ z_2, \psi \}$. 

Now set $\mathcal{L}_1 = \overline{\mathcal{M}_1}$, $\mathcal{L}_2 = \overline{\mathcal{M}_2} \cup \overline{\mathcal{M}_2'}$, and $\mathcal{L}_3 = \overline{\mathcal{M}_3}$. Then, $\mathcal{L}_1$, $\mathcal{L}_2$, and $\mathcal{L}_3$ are all simple, closed curves passing through $\psi$. Parts \ref{vzpsij} and \ref{vzpsiall} of Proposition \ref{linesga} are satisfied by these curves since $\mathcal{M}_1 \cup \mathcal{M}_2 \cup \mathcal{M}_2' \cup \mathcal{M}_3 \cup \{ \psi \} = \mathcal{L}_1 \cup \mathcal{L}_2 \cup \mathcal{L}_3$. 

Furthermore, only one positive real number and only one negative real number lie on each of the closures $\overline{\mathcal{M}_1}$, $\overline{\mathcal{M}_2}$, $\overline{\mathcal{M}_2'}$, and $\overline{\mathcal{M}_3'}$. This implies that $0$ is in the interior of $\mathcal{L}_1$, $\mathcal{L}_2$, and $\mathcal{L}_3$. 

Now, property \ref{anglev} holds from Lemma \ref{sixzgazgapsi}; indeed, this lemma implies that any line $\ell \subset \mathbb{C}$ can only intersect $\mathcal{L}_1 \cup \mathcal{L}_2 \cup \mathcal{L}_3$ in at most six places. Since each $\mathcal{L}_j$ is a closed curve containing $0$ in its interior, $\ell$ intersects each $\mathcal{L}_j$ at least, and thus only, twice; furthermore, $0$ is between these two intersection points. This implies that for each angle $a \in \mathbb{R}$ and each index $j \in \{ 1, 2, 3 \}$, there is only one complex number $z_a \in \mathcal{L}_j$ such that $z_a / |z_a| = e^{ \textbf{i} a}$. This verifies property \ref{anglev}. 

To see property \ref{containmentv}, observe that $\mathcal{L}_i$ and $\mathcal{L}_j$ can only intersect at critical points of $G(z)$, for any distinct $i, j \in \{1, 2, 3 \}$. These critical points are $\psi$ and $-q$. Furthermore, only $\mathcal{L}_2$ contains $-q$; this implies that $\mathcal{L}_i \cap \mathcal{L}_j = \{ \psi \}$, for any $i, j \in \{ 1, 2, 3 \}$. Now, the containment follows from the fact that $\mathcal{L}_1$ intersects the negative real axis only once at $z_1$, that $\mathcal{L}_2$ intersects the negative real axis only once at $-q$, that $\mathcal{L}_3$ intersects the negative real axis only once at $z_2$, and that $z_2 < -q < z_1$. 

Property \ref{qkappaq0v} follows from the fact that $z_2 < -q < z_1$, and property \ref{psianglesv} follows from the containment property \ref{containmentv} and the estimate \eqref{gzpsi}. 

Furthermore, observe that $\Re G(z)$ can only change sign across some $\mathcal{L}_j$; thus, property \ref{positiverealv} and property \ref{negativerealv} follow from the facts that $\lim_{\varepsilon \rightarrow 0^+} \Re G \big( \varepsilon - q \big) = \infty$ and $\lim_{\varepsilon \rightarrow 0^-} \Re G \big( \varepsilon - q \big) = - \infty$, respectively.  
\end{proof}

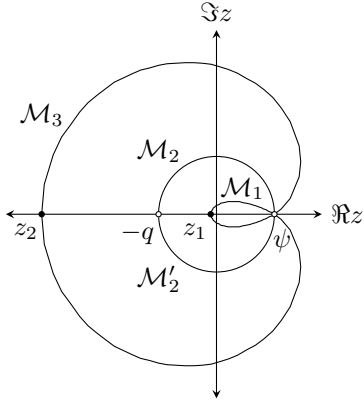
\begin{figure}

\begin{center}

\begin{tikzpicture}[
      >=stealth,
      auto,
      style={
        scale = .7 
      }
			]

			\draw[<->, black	] (9, -3.5) -- (9, 3.5) node[black, above = 0] {$\Im z$};
			\draw[<->, black] (5, 0) -- (11, 0) node[black, right = 0] {$\Re z$};
			
			\path[draw] (10.1, 0) -- (10.08, .02) -- (10.06, .03) -- (10.03, .04) -- (10, .05) -- (9.98, .06) -- (9.95, .08) -- (9.93, .09) -- (9.89, .1) -- (9.86, .12) -- (9.83, .13) -- (9.8, .14) -- (9.75, .16) -- (9.71, .17) -- (9.65, .19) -- (9.61, .2) -- (9.56, .21) -- (9.5, .22) -- (9.47, .23) -- (9.43, .24) -- (9.38, .24) -- (9.35, .24) -- (9.3, .24) -- (9.27, .24) -- (9.22, .24) -- (9.18, .23) -- (9.14, .22) -- (9.08, .21) -- (9.06, .19) -- (9.04, .18) -- (9, .17) -- (8.98, .15) -- (8.96, .13) -- (8.94, .12) -- (8.93, .1) -- (8.93, .09) -- (8.92, .08) -- (8.91, .07) -- (8.9, .04) -- (8.89, .03) -- (8.89, 0) node[black, right = 12, above = 2] {$\mathcal{M}_1$};
			
			\path[draw] (10.06, -.03) -- (10.03, -.04) -- (10, -.05) -- (9.98, -.06) -- (9.95, -.08) -- (9.93, -.09) -- (9.89, -.1) -- (9.86, -.12) -- (9.83, -.13) -- (9.8, -.14) -- (9.75, -.16) -- (9.71, -.17) -- (9.65, -.19) -- (9.61, -.2) -- (9.56, -.21) -- (9.5, -.22) -- (9.47, -.23) -- (9.43, -.24) -- (9.38, -.24) -- (9.35, -.24) -- (9.3, -.24) -- (9.27, -.24) -- (9.22, -.24) -- (9.18, -.23) -- (9.14, -.22) -- (9.08, -.21) -- (9.06, -.19) -- (9.04, -.18) -- (9, -.17) -- (8.98, -.15) -- (8.96, -.13) -- (8.94, -.12) -- (8.93, -.1) -- (8.93, -.09) -- (8.92, -.08) -- (8.91, -.07) -- (8.9, -.04) -- (8.89, -.03) -- (8.89, 0);

			\path[draw] (10.1, 0) -- (10.23, .1) -- (10.4, .3) -- (10.57, .71) -- (10.6, 1) -- (10.57, 1.3) -- (10.4, 1.76) -- (10.23, 2.02) -- (10.07, 2.22) -- (9.9, 2.37) -- (9.74, 2.49) -- (9.57, 2.58) -- (9.4, 2.68) -- (9.23, 2.75) -- (9.07, 2.8) -- (8.9, 2.85) -- (8.73, 2.87) -- (8.57, 2.88) -- (8.4, 2.88) -- (8.23, 2.87) -- (8.06, 2.86) -- (7.9, 2.83) -- (7.73, 2.78) -- (7.56, 2.74) -- (7.4, 2.67) -- (7.23, 2.59) -- (7.06, 2.51) -- (6.9, 2.39) -- (6.73, 2.26) -- (6.56, 2.12) -- (6.4, 1.94) -- (6.23, 1.7) -- (6.06, 1.47) -- (5.9, 1.13) -- (5.73, .57) -- (5.68, 0) node [black, above = 30] {$\mathcal{M}_3$}; 
			
			\path[draw] (10.1, 0)  -- (10.23, -.1) -- (10.4, -.3) -- (10.57, -.71) -- (10.6, -1) -- (10.57, -1.3) -- (10.4, -1.76) -- (10.23, -2.02) -- (10.07, -2.22) -- (9.9, -2.37) -- (9.74, -2.49) -- (9.57, -2.58) -- (9.4, -2.68) -- (9.23, -2.75) -- (9.07, -2.8) -- (8.9, -2.85) -- (8.73, -2.87) -- (8.57, -2.88) -- (8.4, -2.88) -- (8.23, -2.87) -- (8.06, -2.86) -- (7.9, -2.83) -- (7.73, -2.78) -- (7.56, -2.74) -- (7.4, -2.67) -- (7.23, -2.59) -- (7.06, -2.51) -- (6.9, -2.39) -- (6.73, -2.26) -- (6.56, -2.12) -- (6.4, -1.94) -- (6.23, -1.7) -- (6.06, -1.47) -- (5.9, -1.13) -- (5.73, -.57) -- (5.68, 0);

			\draw[black] (10.1, 0) arc (.1:177:1.1)  node [black, above = 15] {$\mathcal{M}_2$};

			\draw[black] (10.1, 0) arc (-.1:-177:1.1)  node [black, below = 15] {$\mathcal{M}_2'$};

			\filldraw[fill = white, draw = black] (7.9, 0) circle [radius=.05] node [black, left = 8, below = 0] {$-q$};
			
			\filldraw[fill = white, draw = black] (10.1, 0) circle [radius=.05] node [black, right = 3, below = 2] {$\psi$};
			
			\filldraw[fill = black, draw = black] (8.89, 0) circle [radius=.05] node [black, left = 6, below = 0] {$z_1$};
			
			\filldraw[fill = black, draw = black] (5.68, 0) circle [radius=.05] node [black, left = 6, below = 0] {$z_2$};

\end{tikzpicture}

\end{center}

\caption{\label{m1m2m3} Shown above are the curves $\mathcal{M}_1$, $\mathcal{M}_2$, $\mathcal{M}_2'$, and $\mathcal{M}_3$ from the proof of Proposition \ref{linesga}. }

\end{figure}

\subsection{Proof of Part 1 of Proposition \ref{processdeterminant} and Proposition \ref{modeldeterminant}} 

\label{VertexRight}

The goal of this section is to establish the first parts of Proposition \ref{processdeterminant} and Proposition \ref{modeldeterminant}. Since both proofs are similar, we will do them simultaneously. Thus, recall the notation $K \in \{ V, A \}$ that corresponds to the stochastic six-vertex model when $K = V$ and the ASEP when $K = A$.

If accordance with the statements of Proposition \ref{processdeterminant} and Proposition \ref{modeldeterminant}, we assume that $\eta \in (\theta, 1)$ when $K = A$ and that $\eta \in (\theta, \kappa)$ when $K = V$. In either case, set $x = x(T) = \lfloor \eta T \rfloor + 1$, and let $p_T = f_{\eta} s T^{1 / 3} - m_{\eta} T$, where we recall that $m_{\eta} = m_{\eta; K}$ is defined through \eqref{m13} and $f_{\eta} = f_{\eta; K}$ are defined through \eqref{ga13} and \eqref{gv13} in the cases $K = A$ and $K = V$, respectively. 

As in Section \ref{KExponentialRight}, we replace $x(T)$ with $\eta T + 1$ for notational convenience; this will not affect the asymptotics. We also simplify notation by omitting the subscript $K$.

Now, in order to establish Proposition \ref{processdeterminant} and Proposition \ref{modeldeterminant}, we must asymptotically analyze the kernel $K^{(p_T)}$ given by \eqref{vptright}. To that end, we will follow a Laplace method similar to what was outlined in Section 5 of \cite{SSVM}. Specifically, in Section \ref{VertexRightPsi}, we will analyze the contribution to the integral on the right side of \eqref{vptright} when $v$ is integrated in a small neighborhood of $\psi$; this will yield a kernel resembling the Airy kernel, which will lead to the $F_{\TW}$ asymptotics. Then, in Section \ref{ExponentialVertexRight}, we will show that the contribution to the integral decays exponentially when $v$ is integrated outside a neighborhood of $\psi$; we will then be able to ignore this contribution, which will lead to the proofs of the first parts of Proposition \ref{processdeterminant} and Proposition \ref{modeldeterminant}.

\subsubsection{Contribution Near \texorpdfstring{$\psi$}{}} 

\label{VertexRightPsi}

In this section we analyze the contribution to the right side of \eqref{vptright} when $w \in \mathcal{C}^{(1)}$ and $v \in \Gamma^{(1)}$, that is, when both $w$ and $v$ are near $\psi$. To that end, define 
\begin{flalign}
\begin{aligned}
\label{vzetatilderight}
\widetilde{K} (w, w') = \displaystyle\frac{1}{2 \textbf{i} \log q} \displaystyle\sum_{j = -\infty}^{\infty} & \displaystyle\int_{\Gamma^{(1)}} \displaystyle\frac{ \exp \Big( T \big( G(w) - G(v) \big) \Big)}{\sin \big( \frac{\pi}{\log q} (2 \pi \textbf{i} j + \log v - \log w) \big)} \displaystyle\prod_{k = 1}^m \displaystyle\frac{(q^{-1} \beta_k^{-1} v; q)_{\infty}}{(q^{-1} \beta_k^{-1} w; q)_{\infty}}\\
& \times \left( \displaystyle\frac{v}{w} \right)^{s f_{\eta} T^{1 / 3}} \displaystyle\frac{dv}{v (w' - v)}, 
\end{aligned}
\end{flalign}

\noindent for any $w, w' \in \mathcal{C}$. Observe that the integral in \eqref{vzetatilderight} is similar to the integral in \eqref{vptright}; the difference is that the former integral is along $\Gamma^{(1)}$, and the latter integral is along $\Gamma$.  

Now let us change variables, a procedure that will in effect ``zoom into'' the region around $\psi$. Denote $\sigma = \psi f_{\eta}^{-1} T^{-1 / 3}$, and set 
\begin{flalign}
\label{wwvvertexright}
w = \psi + \sigma \widehat{w}; \qquad w' = \psi + \sigma \widehat{w'}; \qquad v = \psi + \sigma \widehat{v}; \qquad \widehat{K} (\widehat{w}, \widehat{w'}) = \sigma \widetilde{K} (w, w').
\end{flalign} 

\noindent Also, for any contour $\mathcal{D}$, set $\widehat{\mathcal{D}} = \sigma^{-1} \big( \mathcal{D} - \psi \big)$, where $\sigma^{-1} (\mathcal{D} - \psi)$ denotes all numbers of the form $\sigma^{-1} (z - \psi)$ with $z \in \mathcal{D}$. In particular, from Definition \ref{crgammar} and Definition \ref{linearvertexright}, we find that $\mathcal{C}^{(1)} = \mathfrak{W}_{0, \varepsilon / \sigma}$ and $\Gamma^{(1)} = \mathfrak{V}_{-1, \varepsilon / \sigma}$. 

Our goal in this section is to establish the following lemma. 

\begin{lem}

\label{rightvgamma1tildev}

There exist constants $c, C > 0$ such that 
\begin{flalign}
\label{vright71} 
\left| \widehat{K} (\widehat{w}, \widehat{w'})  \right| \le \displaystyle\frac{C}{1 + |\widehat{w'}|} \exp \big( - c |\widehat{w}|^3 \big), 
\end{flalign} 

\noindent for each $\widehat{w} \in \widehat{\mathcal{C}}^{(1)}$ and $\widehat{w'} \in \widehat{\mathcal{C}}$. 

Furthermore, if we define the kernel $L_s (\widehat{w}, \widehat{w'})$ by 
\begin{flalign}
\label{definitionls}
L_s \big( \widehat{w}, \widehat{w'} \big) = \displaystyle\frac{1}{2 \pi \textbf{\emph{i}}} \displaystyle\int_{\mathfrak{V}_{-1, \infty}} \displaystyle\frac{1 }{\big( \widehat{v} - \widehat{w} \big) \big( \widehat{w'} - \widehat{v}  \big)} \exp \left( \displaystyle\frac{\widehat{w}^3}{3} - \displaystyle\frac{\widehat{v}^3}{3} + s \big( \widehat{v} - \widehat{w} \big) \right) d \widehat{v}, 
\end{flalign}

\noindent then we have that
\begin{flalign}
\label{vright1}
\displaystyle\lim_{T \rightarrow \infty} \widehat{K} \big( \widehat{w}, \widehat{w'} \big) = L_s \big( \widehat{w}, \widehat{w'} \big), 
\end{flalign} 

\noindent for each fixed $\widehat{w}, \widehat{w'} \in \mathfrak{W}_{0, \infty}$. 
\end{lem}

To establish this lemma, we first rewrite the kernel $\widehat{K}$. By \eqref{wwvvertexright} and the fact that $\sigma = \psi f_{\eta}^{-1} T^{-1 / 3}$, we deduce that 
\begin{flalign}
\label{vright111}
 \widehat{K} (\widehat{w}, \widehat{w'}) & = \displaystyle\frac{1}{2 \pi \textbf{i}}\displaystyle\int_{\widehat{\Gamma}^{(1)}} I \big( \widehat{w}, \widehat{w'}; \widehat{v} \big) d \widehat{v}, 
\end{flalign}

\noindent where 
\begin{flalign}
\begin{aligned}
\label{vright112}
I \big( \widehat{w}, \widehat{w'}; \widehat{v} \big) & = \displaystyle\frac{1}{\big( 1 + \psi^{-1} \sigma \widehat{v} \big) \big( \widehat{w'} - \widehat{v} \big)  } \exp \left( \displaystyle\frac{\widehat{w}^3 - \widehat{v}^3}{3} + T \big( R (T^{-1 / 3} \widehat{w}) - R (T^{-1 / 3} \widehat{v}) \big)  \right) \\
& \qquad \times \left( \displaystyle\frac{1 + \psi^{-1} \sigma \widehat{v}}{1 + \psi^{-1} \sigma \widehat{w}} \right)^{\psi \sigma^{-1} s } \displaystyle\prod_{k = 1}^m \displaystyle\frac{\big(q^{-1} \beta_k^{-1} ( \psi + \sigma \widehat{v}  ); q \big)_{\infty}}{\big( q^{-1} \beta_k^{-1} (\psi + \sigma \widehat{w} ); q \big)_{\infty}} \\
& \qquad \times \displaystyle\frac{\pi \psi^{-1} \sigma}{\log q} \displaystyle\sum_{j = -\infty}^{\infty}  \displaystyle\frac{1}{\sin \Big( \frac{\pi}{\log q} \big( 2 \pi \textbf{i} j + \log (1 + \psi^{-1} \sigma \widehat{v} )  - \log (1 + \psi^{-1} \sigma \widehat{w} ) \big) \Big)}. 
\end{aligned}
\end{flalign}

Lemma \ref{rightvgamma1tildev} will follow from the large $T$ asymptotics of (and a uniform estimate on) $I$, given by the following lemma. 

\begin{lem}

\label{uniformlimit13integrand}

There exist constants $c, C > 0$ such that
\begin{flalign}
\label{uniform13integrand}
\Big| I \big( \widehat{w}, \widehat{w'}; \widehat{v} \big) \Big| \le \displaystyle\frac{C}{1 + |\widehat{w'}|} \exp \big( - c (|\widehat{w}|^3 + |\widehat{v}|^3 ) \big), 
\end{flalign}

\noindent for all $\widehat{w} \in \mathfrak{W}_{0, \varepsilon / \sigma}$, $\widehat{w'} \in \widehat{\mathcal{C}}$, and $\widehat{v} \in \mathfrak{V}_{-1, \varepsilon / \sigma}$. 

Furthermore, 
\begin{flalign}
\label{limit13integrand}
\displaystyle\lim_{T \rightarrow \infty} I \big( \widehat{w}, \widehat{w'}; \widehat{v} \big) = \displaystyle\frac{1}{\big( \widehat{v} - \widehat{w} \big) \big( \widehat{w'} - \widehat{v} \big)} \exp \left( \displaystyle\frac{\widehat{w}^3}{3} - \displaystyle\frac{\widehat{v}^3}{3} + s \big( \widehat{w} - \widehat{v} \big) \right), 
\end{flalign}

\noindent for each fixed $\widehat{w}, \widehat{w'}, \widehat{v} \in \mathbb{C}$.  
\end{lem} 

\begin{proof}

Let us first establish the uniform estimate \eqref{uniform13integrand}. To that end, observe that there exists a constant $C_1 > 0$ such that the six inequalities 
\begin{flalign}
\begin{aligned} 
\label{termsintegrand13}
& \left| \displaystyle\frac{1}{1 + \psi^{-1} \sigma \widehat{v}} \right| < C_1; \quad \displaystyle\frac{1}{\widehat{v} - \widehat{w'}} < \displaystyle\frac{C_1}{1 + |\widehat{w'}|} ; \quad \left| \displaystyle\frac{1 + \psi^{-1} \sigma \widehat{v}}{1 + \psi^{-1} \sigma \widehat{w}} \right|^{\psi \sigma^{-1} s} \le C_1 \exp \Big( C_1 \big( |\widehat{w}| + |\widehat{v}| \big) \Big) ;  \\
& \qquad \qquad \displaystyle\frac{\pi \psi^{-1} \sigma}{\big| \log q \big|} \displaystyle\sum_{j \ne 0} \left| \displaystyle\frac{1}{\sin \Big( \frac{\pi}{\log q} \big( 2 \pi \textbf{i} j + \log (1 + \psi^{-1} \sigma \widehat{v} )  - \log (1 + \psi^{-1} \sigma \widehat{w} ) \big) \Big)} \right| \le C_1 T^{-1 / 3} ; \\
& \qquad \left| \left( \displaystyle\frac{\pi \psi^{-1} \sigma}{\log q} \right) \displaystyle\frac{1}{\sin \Big( \frac{\pi}{\log q} \big( \log (1 + \psi^{-1} \sigma \widehat{v} )  - \log (1 + \psi^{-1} \sigma \widehat{w} ) \big) \Big)} \right| < C_1;  \\
& \qquad \qquad \qquad \left| \displaystyle\prod_{k = 1}^m \displaystyle\frac{\big(q^{-1} \beta_k^{-1} ( \psi + \sigma \widehat{v}  ); q \big)_{\infty}}{\big( q^{-1} \beta_k^{-1} (\psi + \sigma \widehat{w} ); q \big)_{\infty}} \right| < C_1 \exp \Big( C_1 \big( |\widehat{w}| + |\widehat{v}| \big) \Big), 	
\end{aligned}
\end{flalign}

\noindent all hold for each  $\widehat{w} \in \mathfrak{W}_{0, \varepsilon / \sigma}$, $\widehat{w'} \in \widehat{\mathcal{C}}$, and $\widehat{v} \in \mathfrak{V}_{-1, \varepsilon / \sigma}$. 

Indeed, the first inequality holds since $v = \psi (1 + \psi^{-1} \sigma \widehat{v})$ is bounded away from $0$ for $v \in \Gamma$. The second inequality holds since $\big| \widehat{w'} - \widehat{v} \big| \ge 1$ for $\widehat{w'} \in \widehat{\mathcal{C}}$ and $\widehat{v} \in \mathfrak{V}_{-1, \varepsilon / \sigma}$. The third inequality holds since $1 + \psi^{-1} \sigma \widehat{w}$ is bounded away from $0$ for $\widehat{w} \in \widehat{\mathcal{C}}$.

The fourth inequality holds since $\sigma = \mathcal{O} \big( T^{-1 / 3} \big)$, since $\sin \big( \frac{\pi}{\log q} \big( 2 \pi \textbf{i} j + \log (1 + \psi^{-1} \sigma \widehat{v} )  - \log (1 + \psi^{-1} \sigma \widehat{w} ) \big) \big)$ increases exponentially in $|j|$, and since that term is also bounded away from $0$. The latter statement is true since $j$ is non-zero. 

The fifth inequality follows from a Taylor expansion, the fact that $v / w$ is always bounded away from any integral power of $q$, and the fact that $\big| \widehat{v} - \widehat{w} \big|^{-1} < C_1$ for sufficiently large $C_1 > 0$. The sixth inequality is true since its left side grows polynomially in $|\widehat{w}|$ and $|\widehat{v}|$, while its right side grows exponentially in these two quantities. 

The estimates \eqref{termsintegrand13} address all terms on the right side of \eqref{vright112}, except for the exponential term. To analyze this term, first recall $\big| R(z) \big| < |z|^3 / 8$, for all $z \in \widehat{\mathcal{C}}^{(1)} \cup \widehat{\Gamma}^{(1)}$; this was stipulated as the third part of Definition \ref{linearvertexright}. Thus, increasing $C_1$ if necessary, we find that 
\begin{flalign}
\begin{aligned}
\label{vright10}
\Bigg| \exp \left( \displaystyle\frac{\widehat{w}^3}{3} - \displaystyle\frac{\widehat{v}^3}{3} + T \big( R (T^{- 1 / 3} \widehat{w}) - R (T^{- 1 / 3} \widehat{v}) \big) \right) \Bigg| & = \exp \left( \displaystyle\frac{\widehat{w}^3}{3} - \displaystyle\frac{\widehat{v}^3}{3} + \displaystyle\frac{\big| \widehat{v} \big|^3}{8} + \displaystyle\frac{\big| \widehat{w} \big|^3}{8} \right) \\
& < C_1 \exp \left( - \displaystyle\frac{1}{5} \Big( \big| \widehat{w} \big|^3 + \big| \widehat{v} \big|^3 \Big) \right) . 
\end{aligned} 
\end{flalign}

\noindent In \eqref{vright10}, the last estimate follows from the fact that $\widehat{w}^3 - \widehat{v}^3 < 0$, for sufficiently large $\widehat{w} \in \widehat{\mathcal{C}}^{(1)}$ and $\widehat{v} \in \widehat{\Gamma}^{(1)}$, and that it decreases cubically in $|\widehat{w}|$ and $|\widehat{v}|$ as they tend to $\infty$. 

Now, the estimate \eqref{uniform13integrand} follows from the definition \eqref{vright112} of $I$, the six estimates \eqref{termsintegrand13}, and the exponential estimate \eqref{vright10}. 

In order to establish the limiting statement \eqref{limit13integrand}, we observe that 
\begin{flalign}
\begin{aligned} 
\label{termsintegrandlimit13}
& \displaystyle\lim_{T \rightarrow \infty} \displaystyle\frac{1}{1 + \psi^{-1} \sigma \widehat{v}} = 1; \qquad \displaystyle\lim_{T \rightarrow \infty} \left( \displaystyle\frac{1 + \psi^{-1} \sigma \widehat{v}}{1 + \psi^{-1} \sigma \widehat{w}} \right)^{\psi \sigma^{-1} s} = \exp \big( s (\widehat{v} - \widehat{w} ) \big); \\
& \qquad \qquad \displaystyle\lim_{T \rightarrow \infty} \exp \left( \displaystyle\frac{\widehat{w}^3 - \widehat{v}^3}{3} + T \big( R (T^{-1 / 3} \widehat{w}) - R (T^{-1 / 3} \widehat{v}) \big)  \right) = \exp \left( \displaystyle\frac{\widehat{w}^3}{3} - \displaystyle\frac{\widehat{v}^3}{3} \right); \\
& \qquad \qquad \displaystyle\lim_{T \rightarrow \infty} \displaystyle\frac{\pi \psi^{-1} \sigma}{\log q} \displaystyle\sum_{j \ne 0}  \displaystyle\frac{1}{\sin \Big( \frac{\pi}{\log q} \big( 2 \pi \textbf{i} j + \log (1 + \psi^{-1} \sigma \widehat{v} )  - \log (1 + \psi^{-1} \sigma \widehat{w} ) \big) \Big)} = 0; \\
& \qquad \displaystyle\lim_{T \rightarrow \infty} \left( \displaystyle\frac{\pi \psi^{-1} \sigma}{\log q} \right) \displaystyle\frac{1}{\sin \Big( \frac{\pi}{\log q} \big( \log (1 + \psi^{-1} \sigma \widehat{v} )  - \log (1 + \psi^{-1} \sigma \widehat{w} ) \big) \Big)} = \displaystyle\frac{1}{\widehat{v} - \widehat{w}} ; 
\end{aligned}
\end{flalign}

\noindent and 
\begin{flalign}
\label{productqv}  
\displaystyle\lim_{T \rightarrow \infty}  \displaystyle\prod_{k = 1}^m \displaystyle\frac{\big(q^{-1} \beta_k^{-1} ( \psi + \sigma \widehat{v}  ); q \big)_{\infty}}{\big( q^{-1} \beta_k^{-1} (\psi + \sigma \widehat{w} ); q \big)_{\infty}}  = 1, 
\end{flalign}

\noindent for each $\widehat{w}, \widehat{w'}, \widehat{v} \in \mathbb{C}$. Indeed, the first limit in \eqref{termsintegrandlimit13} and the limit \eqref{productqv} follow from the fact that $\sigma$ tends to $0$ as $T$ tends to $\infty$ and, in the latter limit, the fact that $q^{-1} \beta_k^{-1} \psi$ is uniformly bounded away from any non-positive power of $q$. The second and fifth limits in \eqref{termsintegrandlimit13} follow from a Taylor expansion and the fact that $\sigma$ tends to $0$ as $T$ tends to $\infty$. The fourth limit in \eqref{termsintegrandlimit13} follows from the fourth estimate in \eqref{termsintegrand13}. Furthermore, the third limit in \eqref{termsintegrandlimit13} follows from the fact that $\big| R(z) \big| = \mathcal{O} (|z|^4)$; see \eqref{gzpsi}. Thus, \eqref{limit13integrand} follows from multiplying the results of \eqref{termsintegrandlimit13} and \eqref{productqv}. 
\end{proof}

\noindent Now we can establish Lemma \ref{rightvgamma1tildev}. 

\begin{proof}[Proof of Lemma \ref{rightvgamma1tildev}]

The uniform estimate \eqref{vright71} follows from integrating the estimate \eqref{uniform13integrand} over $\widehat{v} \in \mathfrak{V}_{-1, \varepsilon / \sigma}$. By \eqref{vright111}, \eqref{uniform13integrand}, \eqref{limit13integrand} the dominated convergence theorem, and the fact that $\widehat{\Gamma}^{(1)} = \mathfrak{V}_{-1, \varepsilon / \sigma}$, we deduce that 
\begin{flalign}
\label{vrightlimit11}
\displaystyle\lim_{T \rightarrow \infty} \widehat{K} \big( \widehat{w}, \widehat{w'} \big) = \displaystyle\lim_{T \rightarrow \infty} \displaystyle\frac{1}{2 \pi \textbf{i}} \displaystyle\int_{\mathfrak{V}_{-1, \varepsilon / \sigma}} \displaystyle\frac{1}{\big( \widehat{v} - \widehat{w} \big) \big( \widehat{w'} - \widehat{v} \big)} \exp \left( \displaystyle\frac{\widehat{w}^3}{3} - \displaystyle\frac{\widehat{v}^3}{3} + s \big( \widehat{v} - \widehat{w} \big) \right) d \widehat{v}, 
\end{flalign} 

\noindent for each fixed $\widehat{w} \in \mathfrak{W}_{0, \infty}$ and $\widehat{w'} \in \mathfrak{W}_{0, \infty}$. 

Hence, the limit \eqref{vright1} follows from \eqref{vrightlimit11}, the fact that $\sigma$ tends to $0$ as $T$ tends to $\infty$, and the exponential decay in $\big| \widehat{v} \big|^3$ of the integrand defining the kernel $L_s$ (see \eqref{definitionls}). 
\end{proof}

\subsubsection{Exponential Decay of \texorpdfstring{$K$}{} on \texorpdfstring{$\mathcal{C}^{(2)}$}{} and \texorpdfstring{$\Gamma^{(2)}$}{}}

\label{ExponentialVertexRight}

In this section we analyze the integral \eqref{vptright} defining $K (w, w')$ when either $w$ or $v$ is not close to $\psi$, that is, when either $w \in \mathcal{C}^{(2)}$ or $v \in \Gamma^{(2)}$. In this case, will show that the integral decays exponentially in $T$, which will allow us to ignore this contribution and apply the results of the previous section. 

We begin with the following lemma, which estimates the exponential term in the integrand on the right side of \eqref{vptright}. 

\begin{lem}

\label{rightcv2gammav2exponential}

There exists some positive real number $c_1 > 0$, independent of $T$, such that 
\begin{flalign*}
\max \left\{ \sup_{\substack{w \in \mathcal{C} \\ v \in \Gamma^{(2)}}} \Re \big( G(w) - G(v) \big), \sup_{\substack{w \in \mathcal{C}^{(2)} \\ v \in \Gamma}} \Re \big( G(w) - G(v) \big) \right\} < - c_1.
\end{flalign*}
\end{lem}

\begin{proof}
From part \ref{positiverealv} and part \ref{negativerealv} of Proposition \ref{linesgv}, we find that $\Re \big( G(w) - G(v) \big) < 0$ for all $w \in \mathcal{C}$ and $v \in \Gamma^{(2)}$, and that $\Re \big( G(w) - G(v) \big) < 0$ for all $w \in \mathcal{C}^{(2)}$ and $v \in \Gamma$ (uniformly in $T$). Thus, the existence of the claimed constant $c_1$ follows from the compactness of the contours $\mathcal{C}$ and $\Gamma$. 
\end{proof}

\noindent Now, recall the definitions of $\sigma = \psi f_{\eta}^{-1} T^{-1 / 3}$; $\widehat{w}$, $\widehat{w'}$, and $\widehat{v}$ from \eqref{wwvvertexright}; and $I$ from \eqref{vright112}. Define 
\begin{flalign}
\label{kbar13}
\overline{K} \big( \widehat{w}, \widehat{w'} \big) = \sigma K (w, w') = \displaystyle\frac{1}{2 \pi \textbf{i}} \displaystyle\int_{\widehat{\Gamma}} I \big( \widehat{w}, \widehat{w'}; \widehat{v} \big) d \widehat{v},
\end{flalign}  	

\noindent for each $\widehat{w}, \widehat{w'} \in \widehat{\mathcal{C}}$. From the change of variables \eqref{wwvvertexright}, we have that 
\begin{flalign}
\label{rightvdeterminantkernelbar13}
\det \big( \Id + K \big)_{L^2 (\mathcal{C})} = \det \big( \Id + \overline{K} \big)_{L^2 (\widehat{\mathcal{C}})}. 
\end{flalign}

\noindent Lemma \ref{rightcv2gammav2exponential} implies the following result. 

\begin{cor}

\label{kclosekbarc2small}

There exist constants $c, C > 0$ such that 
\begin{flalign}
\label{kclosekbar}
\Big| \overline{K} \big( \widehat{w}, \widehat{w'} \big) - \widehat{K} \big( \widehat{w}, \widehat{w'} \big) \Big| < C \exp \Big( - c \big( T + |\widehat{w}|^3 \big) \Big), 
\end{flalign}

\noindent for all $\widehat{w} \in \widehat{\mathcal{C}}^{(1)}$ and $\widehat{w'} \in \widehat{\mathcal{C}} \cup \mathfrak{W}_{0, \infty}$, and such that 
\begin{flalign}
\label{kbarsmallc2}
\Big| \overline{K} \big( \widehat{w}, \widehat{w'} \big) \Big| < C \exp \Big( - c \big( T + |\widehat{w}|^3 \big) \Big), 
\end{flalign}

\noindent for all $\widehat{w} \in \widehat{\mathcal{C}}^{(2)}$ and $\widehat{w'} \in \widehat{\mathcal{C}} \cup \mathfrak{W}_{0, \infty}$. 
\end{cor}

\begin{proof}

Let us begin with the proof of \eqref{kclosekbar}. Defining the constant $c_1 > 0$ as in Lemma \ref{rightcv2gammav2exponential}, we have that 
\begin{flalign}
\label{kclosekbar1}
& \Big| \overline{K} \big( \widehat{w}, \widehat{w'} \big) - \widehat{K} \big( \widehat{w}, \widehat{w'} \big) \Big| \nonumber \\
& \quad \le \displaystyle\frac{\sigma}{2} \displaystyle\int_{\Gamma^{(2)}} \displaystyle\sum_{j = -\infty}^{\infty} \Bigg| \displaystyle\frac{\exp \Big( T \big( G(w) - G(v) \big) \Big) }{v (w' - v) \sin \big( \frac{\pi}{\log q} (2 \pi \textbf{i} j + \log v - \log w) \big)}	\left( \displaystyle\frac{v}{w} \right)^{s f_{\eta} T^{1 / 3}}   \displaystyle\prod_{k = 1}^m \displaystyle\frac{(q^{-1} \beta_k^{-1} v; q)_{\infty}}{(q^{-1} \beta_k^{-1} w; q)_{\infty}} \Bigg| dv \nonumber \\
& \quad  = \mathcal{O} \big( e^{- c_1 T} \big). 
\end{flalign}

The first estimate in \eqref{kclosekbar1} follows from the facts that $\widehat{K} = \sigma \widetilde{K}$, that $\overline{K} = \sigma K$, and that the definition for $K$ given in \eqref{vptright} and the definition for $\widetilde{K}$ given in \eqref{vzetatilderight} coincide except that the former integral is over $\Gamma$, and the latter integral is over $\Gamma^{(1)}$. 

The second estimate in \eqref{kclosekbar1} follows from the the following four facts. First, by Lemma \ref{rightcv2gammav2exponential} $\Re \big( G(w) - G(v) \big) < c_1$, for $w \in \mathcal{C}$ and $v \in \Gamma^{(2)}$; this gives rise to the $\mathcal{O} \big( e^{-c_1 T} \big)$ in \eqref{kclosekbar1}. Second, the denominator of the integrand is bounded away from $0$ (independently of $T$) if $\widehat{w} \in \widehat{\mathcal{C}}$, $\widehat{w'} \in \widehat{\mathcal{C}} \cup \mathfrak{W}_{0, \infty}$, and $\widehat{v} \in \widehat{\Gamma}^{(2)}$. Third, the denominator of the integrand increases exponentially in $|j|$; this verifies convergence of the sum in the integrand. Fourth, the contour $\Gamma$ is compact; this verifies that the integral is bounded by a constant multiple of the supremum of its integrand. 

Thus, the estimate \eqref{kclosekbar} follows from \eqref{kclosekbar1} and the fact that $|\widehat{w}| = \mathcal{O} \big( T^{1 / 3} \big)$ for $\widehat{w} \in \widehat{\mathcal{C}}$. 

Now, let us establish the estimate \eqref{kbarsmallc2}. We have that 
\begin{flalign}
\begin{aligned} 
\label{kbarsmallc21}
\Big| \overline{K} \big( \widehat{w}, \widehat{w'} \big) \Big| & \le \displaystyle\frac{\sigma}{2} \displaystyle\sum_{j = -\infty}^{\infty} \displaystyle\oint_{\Gamma} \Bigg| \displaystyle\frac{\exp \Big( T \big( G(w) - G(v) \big) \Big)}{v (w' - v) \sin \big( \frac{\pi}{\log q} (2 \pi \textbf{i} j + \log v - \log w) \big)} \\
& \qquad \qquad \qquad \times \left( \displaystyle\frac{v}{w} \right)^{s f_{\eta} T^{1 / 3}}  \displaystyle\prod_{k = 1}^m \displaystyle\frac{(q^{-1} \beta_k^{-1} v; q)_{\infty}}{(q^{-1} \beta_k^{-1} w; q)_{\infty}} \Bigg| dv = \mathcal{O} \big( e^{- c_1 T} \big), 
\end{aligned}
\end{flalign}

\noindent for all $\widehat{w} \in \widehat{\mathcal{C}}^{(2)}$. The first estimate is due to the definitions \eqref{vptright} and $\overline{K} = \sigma K$. The second estimate in \eqref{kbarsmallc21} follows from the estimate on $\Re \big( G(w) - G(v) \big)$, for $w \in \mathcal{C}^{(2)}$ and $v \in \Gamma$, given by Lemma \ref{rightcv2gammav2exponential}; the compactness of the contour $\Gamma$; and the fact that the denominator of the integrand is bounded below by $\sigma$ if $w' \in \mathcal{C}$ and $v \in \Gamma$ (since the factor $|w' - v|$ is bounded below by $\sigma$), and also increases exponentially in $|j|$.  

Thus, the estimate \eqref{kbarsmallc2} follows from \eqref{kbarsmallc21} and the fact that $|\widehat{w}| = \mathcal{O} \big( T^{1 / 3} \big)$ for $\widehat{w} \in \widehat{\mathcal{C}}$. 
\end{proof}

\noindent Now we can use Corollary \ref{determinantlimitkernels}, Lemma \ref{rightvgamma1tildev}, and Corollary \ref{kclosekbarc2small} to establish the first parts of Proposition \ref{processdeterminant} and Proposition \ref{modeldeterminant}. 

\begin{proof}[Proof of Part 1 of Proposition \ref{processdeterminant} and Proposition \ref{modeldeterminant}]

The proof will follow from an application of Corollary \ref{determinantlimitkernels} to the sequence of kernels $\{ \overline{K} (\widehat{w}, \widehat{w'}) \}$, the sequence of contours $\{ \mathfrak{W}_{0, \infty} \cup \widehat{\mathcal{C}}^{(2)} \}$ (recall that these kernels and contours depend on, and are hence indexed by, the time $T$), and the kernel $L_s (\widehat{w}, \widehat{w'})$ (given by \eqref{definitionls}). 

Let us verify that the conditions of that corollary hold. First, we require the existence of a dominating function $\textbf{K}$ satisfying the three conditions of this corollary. In fact, we may take $\textbf{K} (z)$ to be of the form $C \exp \big( -c |z|^3 \big)$, for some sufficiently small constant $c > 0$ and some sufficiently large constant $C > 0$; such a $\textbf{K}$ satisfies the second and third conditions of Corollary \ref{determinantlimitkernels}. That such a $\textbf{K}$ satisfies the first condition (for sufficiently small $c$ and sufficiently large $C$) follows from \eqref{vright71} and \eqref{kclosekbar} in the case $z \in \widehat{\mathcal{C}}^{(1)}$, and from \eqref{kbarsmallc2} in the case $z \in \widehat{\mathcal{C}}^{(2)}$. 

Second, we require the limit condition, which states that for any $\widehat{w}, \widehat{w'} \in \widehat{\mathcal{C}} \cup \mathfrak{W}_{0, \infty}$ we have that $\lim_{N \rightarrow \infty} \textbf{1}_{\widehat{w}, \widehat{w'} \in \widehat{\mathcal{C}}} \overline{K} \big( \widehat{w}, \widehat{w'} \big) = \textbf{1}_{\widehat{w}, \widehat{w'} \in \mathfrak{W}_{0, \infty}} L_s \big( \widehat{w}, \widehat{w'} \big)$. When $\widehat{w} \in \widehat{\mathcal{C}}^{(1)}$ and $\widehat{w'} \in \mathfrak{W}_{0, \infty}$, this follows from \eqref{vright1} and \eqref{kclosekbar}. When $\widehat{w'} \in \widehat{\mathcal{C}}^{(2)}$, this follows from the fact that both $\overline{K} (\widehat{w}, \widehat{w'})$ and $L_s \big( \widehat{w}, \widehat{w'} \big)$ tend to $0$ as $T$ tends to $\infty$; this can be deduced from \eqref{vright71} and \eqref{definitionls}, respectively. When $\widehat{w} \in \widehat{\mathcal{C}}^{(2)}$, this follows from \eqref{kbarsmallc2} and the fact that $L_s \big( \widehat{w}, \widehat{w'} \big)$ decays exponentially in $ \big| \widehat{w} \big|^3$. Furthermore, when $\widehat{w} \in \mathfrak{W}_{0, \infty} \setminus \widehat{\mathcal{C}}^{(1)} = \mathfrak{W}_{0, \infty} \setminus \mathfrak{W}_{0, \varepsilon / \sigma}$, this again follows from the fact that $L_s \big( \widehat{w}, \widehat{w'} \big)$ exhibits exponential decay in $ \big| \widehat{w} \big|^3$. This verifies the limit condition. 

Thus Corollary \ref{determinantlimitkernels} applies, and we obtain that 
\begin{flalign}
\label{determinant4kernelrightc1c}
 \displaystyle\lim_{T \rightarrow \infty} \det \big( \Id + K \big)_{L^2 (\mathcal{C})} = \displaystyle\lim_{T \rightarrow \infty} \det \big( \Id + \overline{K} \big)_{L^2 (\widehat{\mathcal{C}})} = \det \big( \Id + L_s \big)_{L^2 (\mathfrak{W}_{0, \infty})}, 
\end{flalign}

\noindent where we have used \eqref{rightvdeterminantkernelbar13} to deduce the first identity. 

Now, it is known (see, for instance, Lemma 8.6 of \cite{FEF}) that $\det \big( \Id + L_s \big)_{\mathfrak{W}_{0, \infty}} = F_{\TW} (s)$. Thus, we deduce from \eqref{determinant4kernelrightc1c} that 
\begin{flalign}
\label{determinant6kernelrightc1c}
\displaystyle\lim_{T \rightarrow \infty} \det \big( \Id + K \big)_{L^2 (\mathcal{C})} = F_{\TW} (s). 
\end{flalign}

\noindent Applying \eqref{determinant6kernelrightc1c} in the case $K = A$ yields the first part of Proposition \ref{processdeterminant}, and applying \eqref{determinant6kernelrightc1c} in the case $K = V$ yields the first part of Proposition \ref{modeldeterminant}. 
\end{proof}

\section{Baik-Ben-Arous-P\'{e}ch\'{e} Phase Transitions}

\label{VertexNear}

Our goal in this section is to establish the second parts of Proposition \ref{processdeterminant} and Proposition \ref{modeldeterminant}, which examine the case when $\eta$ is ``near'' the phase transition $\theta$. This will be similar to what was explained in Section \ref{VertexRight}, but there will be two main differences. Both are based on the fact that the critical point $\psi_K$ (defined in \eqref{derivativega} and \eqref{derivativegv} when $K = A$ and $K = V$, respectively) is now close to $q \beta$ when $\eta$ is close to $\theta$. 

The first difference is that we cannot take exactly the same contours as in Section \ref{ckgammakcontours}. Indeed, these contours will intersect the positive real line near $\psi = q \beta$, meaning that they might pass through some pole $q \beta_j$ of the integrand in \eqref{vptright}; this is not permitted by the statement of Propositon \ref{asymptoticheightvertex}. The second difference is that the approximation \eqref{productqv}, which was based on the assumption that $1 - q^{-1} \beta^{-1} \psi$ is bounded away from $0$ (independently of $T$), no longer holds. 

To remedy the first issue we perturb the contours by translating them to the left by some sufficiently large multiple of $T^{-1 / 3}$, so that they avoid the poles $q \beta_j$. To remedy the second issue, we Taylor expand the left side of \eqref{productqv}, taking into account the new location of $\psi$. Both of these will be done in Section \ref{ModificationNear}. Then, in Section \ref{ProofNear}, we will establish the second parts of Proposition \ref{processdeterminant} and Proposition \ref{modeldeterminant}; this will be similar to what was done in Section \ref{VertexRight}.

\subsection{Modifying the Approximation and Contours} 

\label{ModificationNear}

The goal of this section is twofold. First, we modify the contours $\mathcal{C}_K$ and $\Gamma_K$ to ensure that they still satisfy the requirements of Proposition \ref{asymptoticheightvertex}. Second, we modify the approximation \eqref{productqv} so that it holds when $\psi_K$ is close to $q \beta $. Let us begin with the contours; in what follows, we recall the notation of Theorem \ref{asymmetriclimit} and Theorem \ref{hlimit}, and that $\sigma = \sigma_K = \psi_K f_{\eta; K}^{-1} T^{-1 / 3}$. 

We would like to shift the contours $\mathcal{C}_K$ and $\Gamma_K$ to the left so that they both avoid each of the $q \beta_j$. To that end, recall that $\big| b_{j, T} - b \big| = \mathcal{O} \big( T^{-1 / 3} \big)$ for each $j \in [1, m]$, and $|\eta - \theta_K| = \mathcal{O} \big( T^{-1 / 3} \big)$. The first estimate implies that $\big| q \beta_{T, j} - q \beta \big| = \mathcal{O} \big( T^{-1 / 3} \big)$, for each $j$ (a more precise approximation for the error is given by \eqref{betatj}), and the second estimate implies that $\big| \psi_K - q \beta \big| = \mathcal{O} \big( T^{-1 / 3} \big)$ (again, a more precise approximation for the error is given by \eqref{psianear} and \eqref{psivnear}, in the cases $K = A$ and $K = V$, respectively). 

Thus, there exists some finite positive real number $E = E_K$ (independent of $T$) so that $\psi - \sigma E < \min_{1 \le j \le m} q \beta_j$; in fact, one can verify using \eqref{betatj}, \eqref{psianear}, and \eqref{psivnear} that it suffices to take $E > \max_{j \in [1, k]} c_j$, where $c_1, c_2, \ldots , c_m$ are given by \eqref{cprocessphasetransition} and \eqref{cmodelphasetransition} in the cases $K = A$ and $K = V$, respectively. The quantity $E_K$ will signify the translation of the contours $\mathcal{C}_K$ and $\Gamma_K$. Specifically, we make the following definitions, which are analogous to Definition \ref{linearvertexright}, Definition \ref{curvedvertexright}, and Definition \ref{vertexrightcontours}. In what follows, we recall the level lines $\mathcal{L}_{1; K}$, $\mathcal{L}_{2; K}$, and $\mathcal{L}_{3; K}$ defined in Section \ref{ckgammakcontours}.

\begin{definition}

\label{linearvertexnear} 

Let $\mathcal{C}_K^{(1)} = \mathfrak{W}_{\psi - \sigma E, \varepsilon}$ and $\Gamma_K^{(1)} = \mathfrak{V}_{\psi - \sigma E - \sigma, \varepsilon}$, where $\varepsilon$ is chosen to be sufficiently small (independent of $T$) such that the four properties listed in Definition \ref{linearvertexright} all hold. 

Such a positive real number $\varepsilon$ is guaranteed to exist by part \ref{psianglesv} of Proposition \ref{linesgv} (in the case $K = V$), part \ref{psianglesv} of Proposition \ref{linesga} (in the case $K = A$), and the estimate \eqref{gzpsi}. 

See Figure \ref{shiftedcontours13} for (rescaled) examples of these contours. 

\end{definition}

\begin{definition}

\label{curvedvertexnear}  

Let $\mathcal{C}_K^{(2)}$ denote a positively oriented contour from the top endpoint $\psi - \sigma E + \varepsilon e^{\pi \textbf{i} / 3}$ of $\mathcal{C}_K^{(1)}$ to the bottom endpoint $\psi - \sigma E + \varepsilon e^{- \pi \textbf{i} / 3}$ of $\mathcal{C}^{(1)}$, and let $\Gamma_K^{(2)}$ denote a positively oriented contour from the top endpoint $\psi - \sigma E - \sigma + \varepsilon e^{2 \pi \textbf{i} / 3}$ of $\Gamma_K^{(1)}$ to the bottom endpoint $\psi - \sigma E - \sigma + \varepsilon e^{-2 \pi \textbf{i} / 3}$ of $\Gamma^{(1)}$, satisfying the five properties listed in Definition \ref{curvedvertexright} (if $K = A$, then ignore that part of the fourth condition stating that $\Gamma_K^{(1)} \cup \Gamma_K^{(2)}$ does not contain $-q \kappa$).

Such contours $\mathcal{C}_K^{(2)}$ and $\Gamma_K^{(2)}$ are guaranteed to exist by part \ref{anglev} and part \ref{qkappaq0v} of Proposition \ref{linesgv} (when $K = V$) and by part \ref{anglev} and part \ref{qkappaq0v} of Proposition \ref{linesga} (in the case $K = A$). 

\end{definition} 

\begin{definition}

\label{vertexnearcontours}

Set $\mathcal{C}_K = \mathcal{C}_K^{(1)} \cup \mathcal{C}_K^{(2)}$ and $\Gamma_K = \Gamma_K^{(1)} \cup \Gamma_K^{(2)}$. 

\end{definition}

The following lemma states that $\mathcal{C}_K$ and $\Gamma_K$ satisfy their required properties, for sufficiently large $T$. We omit its proof, but let us mention that the reason as to why $\mathcal{C}_K$ and $\Gamma_K$ avoid the poles $\{ q \beta_j \}_{1 \le j \le m}$ is that they both intersect the positive real axis to the left of any $q \beta_j$ (since $\psi - \sigma E < \min_{1 \le j \le m} q \beta_j$). 

\begin{lem}

\label{rightcagammanear}

If $K = A$ and $T$ is sufficiently large, then $\Gamma_K$ is a positively oriented, star-shaped curve that contains $0$, but leaves outside $-q$ and $q \beta_j$ for each $j \in [1, m]$. Furthermore, $\mathcal{C}_K$ is a positively oriented, star-shaped contour that is contained inside $q^{-1} \Gamma_K$; that contains $0$, $-q$ and $\Gamma_K$; but that leaves outside $q \beta_j$ for each $j \in [1, m]$.

If $K = V$ and $T$ is sufficiently large, then $\Gamma_K$ is a positively oriented, star-shaped curve that contains $0$, but leaves outside $-q \kappa$ and $q \beta_j$ for each $j \in [1, m]$. Furthermore, $\mathcal{C}_K$ is a positively oriented, star-shaped contour that is contained inside $q^{-1} \Gamma_K$; that contains $0$, $-q$ and $\Gamma_K$; but that leaves outside $q \beta_j$ for each $j \in [1, m]$. 
\end{lem}

 Having defined the contours $\mathcal{C}_K$ and $\Gamma_K$, we now turn to the approximation \eqref{productqv}. 

Observe that the left side of \eqref{productqv} is a product of terms involving $q \beta_j$ and $\psi$. When $\eta$ is close to $\theta$ and all of the $b_j$ are close to $b$, the $q \beta_j$, $\psi_A$, and $\psi_V$ are all close to $q \beta$; in particular, this implies that the numerator and denominator of \eqref{productqv} both tend to $0$ as $T$ tends to $\infty$. Thus to analyze \eqref{productqv} further, we must obtain a more refined estimate on the error in approximating these quantities by $q \beta$. 

Let us begin with $q \beta_j$. To that end, we recall from the second part of Theorem \ref{asymmetriclimit} (or Theorem \ref{hlimit}) that the $\{ b_{j, T} \}_{T \in \mathbb{Z}_{> 0 }}$ satisfy $\lim_{T \rightarrow \infty} T^{1 / 3} (b_{j, T} - b) = d_j$. Thus, since $\beta_{j, T} = b_{j, T} / (1 - b_{j, T})$, Taylor expanding yields 
\begin{flalign}
\label{betatj} 
q \beta_{T, j} = \displaystyle\frac{q b_{T, j} }{1 - b_{T, j}} & = q \left( \displaystyle\frac{b + d_j T^{-1 / 3} + o(T^{-1 / 3})}{1 - b - d_j T^{-1 / 3} + o (T^{-1 / 3})} \right) = q \beta \left( 1 + \displaystyle\frac{d_j T^{-1 / 3}}{b (1 - b)} \right) + o (T^{-1 / 3}) . 
\end{flalign}

Next we analyze the error in approximating $\psi_K$. This depends on whether $K = A$ or $K = V$. 

First assume that $K = A$. Recall that $\psi_A$ is defined in \eqref{derivativega} and that $\{ \eta_T \}_{T \in \mathbb{Z}_{> 0}}$ satisfies $\lim_{T \rightarrow \infty} T^{1 / 3} (\eta_T - \theta_A) = d$. Thus, Taylor expanding $\psi_A$ and using the fact that $\theta_A = 1 - 2b$, we find that 
\begin{flalign}
\label{psianear}
\psi_A = q \left( \displaystyle\frac{1 - \theta_A - d T^{- 1 / 3} + o (T^{-1 / 3})}{1 + \theta_A + d T^{-1 / 3} + o(T^{-1 / 3}) } \right) = q \beta \left( 1 - \displaystyle\frac{d T^{-1 / 3}}{2b (1 - b)} \right) + o(T^{-1 / 3}).
\end{flalign}

Next assume that $K = V$. Recalling that $\psi_V$ is defined in \eqref{derivativegv}, Taylor expanding $\psi_V$, and recalling the facts that $\Lambda = b + \kappa (1 - b)$ and $\theta_V = \kappa^{-1} \Lambda^2$ (see \eqref{processlocationtransition}), we find that 
\begin{flalign}
\begin{aligned}
\label{psivnear}
\psi_V = \displaystyle\frac{q (\kappa - \sqrt{\kappa \eta})}{\sqrt{\kappa \eta} - 1} &= \displaystyle\frac{q \Big(\kappa - \big( 1 + \frac{d T^{-1 / 3}}{\theta_V} + o (T^{-1 / 3}) \big)^{1 / 2} \sqrt{\kappa \theta_V} \Big)}{\big( 1 + \frac{d T^{-1 / 3}}{\theta_V} + o (T^{-1 / 3}) \big)^{1 / 2} \sqrt{\kappa \theta_V} - 1} \\
&= q \beta \left( 1 - \displaystyle\frac{d \big( b + \kappa (1 - b) \big) T^{-1 / 3} }{2 \theta_V (\kappa - 1) b (1 - b)} + o(T^{-1 / 3}) \right) \\
&= q \beta \left( 1 - \displaystyle\frac{d \kappa T^{-1 / 3} }{2 (\kappa - 1) b (1 - b) \Lambda } \right) + o(T^{-1 / 3}) . 
\end{aligned}
\end{flalign}

Using \eqref{betatj}, \eqref{psianear}, and \eqref{psivnear}, we can establish the following lemma. 

\begin{lem}

\label{productqvnear}

Recall that the $c_j = c_{j; K}$ were defined in the second part of Theorem \ref{asymmetriclimit} (in the case $K = A$) and in the second part of Theorem \ref{hlimit} (in the case $K = V$). For any fixed $\widehat{w}, \widehat{v} \in \mathbb{C}$ (with $\widehat{w}$ not equal to any $-c_j$), we have that
\begin{flalign*}
\displaystyle\lim_{T \rightarrow \infty} \displaystyle\prod_{k = 1}^m \displaystyle\frac{\big(q^{-1} \beta_k^{-1} ( \psi_K + \sigma \widehat{v}  ); q \big)_{\infty}}{\big( q^{-1} \beta_k^{-1} (\psi_K + \sigma \widehat{w} ); q \big)_{\infty}} & = \displaystyle\prod_{j = 1}^m \displaystyle\frac{\widehat{v} + c_j}{ \widehat{w} + c_j}. 
\end{flalign*}

\end{lem}

\begin{proof} 
Observe that 
\begin{flalign}
\label{productqv1}
\displaystyle\prod_{k = 1}^m & \displaystyle\frac{\big(q^{-1} \beta_k^{-1} ( \psi + \sigma \widehat{v}  ); q \big)_{\infty}}{\big( q^{-1} \beta_k^{-1} (\psi + \sigma \widehat{w} ); q \big)_{\infty}} = \displaystyle\prod_{k = 1}^m \displaystyle\frac{1 - q^{-1} \beta_k^{-1} \big( \psi + \sigma \widehat{v} \big)}{1 - q^{-1} \beta_k^{-1} \big( \psi + \sigma \widehat{w} \big)} \displaystyle\prod_{k = 1}^m \displaystyle\frac{\big( \beta_k^{-1} ( \psi + \sigma \widehat{v}  ); q \big)_{\infty}}{\big( \beta_k^{-1} (\psi + \sigma \widehat{w} ); q \big)_{\infty}}. 
\end{flalign}

\noindent Using the fact that $\lim_{T \rightarrow \infty} \beta_k = \beta$ and $\lim_{T \rightarrow \infty} \big( \psi + \sigma \widehat{v} \big) = q \beta = \lim_{T \rightarrow \infty} \big( \psi + \sigma \widehat{w} \big)$, we deduce that $\lim_{T \rightarrow \infty} \big( \beta_k^{-1} ( \psi + \sigma \widehat{v}  ); q \big)_{\infty} = (q; q)_{\infty} = \lim_{T \rightarrow \infty} \big( \beta_k^{-1} ( \psi + \sigma \widehat{w}  ); q \big)_{\infty}$. 

Thus, \eqref{productqv1} implies that  
\begin{flalign}
\begin{aligned}
\label{productqv2}
\displaystyle\lim_{T \rightarrow \infty} \displaystyle\prod_{k = 1}^m \displaystyle\frac{\big(q^{-1} \beta_k^{-1} ( \psi + \sigma \widehat{v}  ); q \big)_{\infty}}{\big( q^{-1} \beta_k^{-1} (\psi + \sigma \widehat{w} ); q \big)_{\infty}} & = \displaystyle\lim_{T \rightarrow \infty} \displaystyle\prod_{k = 1}^m \displaystyle\frac{1 - q^{-1} \beta_k^{-1} \big( \psi + \sigma \widehat{v} \big)}{1 - q^{-1} \beta_k^{-1} \big( \psi + \sigma \widehat{v} \big)} \\
& =  \displaystyle\lim_{T \rightarrow \infty} \displaystyle\prod_{k = 1}^m \displaystyle\frac{\psi q^{-1} \beta_k^{-1} \widehat{v} + T^{1 / 3} f_{\eta; K} \big( q^{-1} \beta_k^{-1} \psi - 1 \big)  }{\psi q^{-1} \beta_k^{-1} \widehat{w} + T^{1 / 3} f_{\eta; K} \big( q^{-1} \beta_k^{-1} \psi - 1 \big)   }
\end{aligned} 
\end{flalign}

\noindent where in the second identity we used the fact that $\sigma = \psi f_{\eta}^{-1} T^{-1 / 3}$. Thus, it suffices to analyze the limit on the right side of \eqref{productqv2}; this will depend on whether $K = A$ or $K = V$. 

In the case $K = A$, we apply \eqref{betatj} and \eqref{psianear} to find that 
\begin{flalign}
\label{productqanear}
\displaystyle\lim_{T \rightarrow \infty} \displaystyle\prod_{k = 1}^m \displaystyle\frac{\psi q^{-1} \beta_k^{-1} \widehat{v} + T^{1 / 3} f_{\eta; A} \big( q^{-1} \beta_k^{-1} \psi - 1 \big)  }{\psi q^{-1} \beta_k^{-1} \widehat{w} + T^{1 / 3} f_{\eta; A} \big( q^{-1} \beta_k^{-1} \psi - 1 \big)  } = \displaystyle\prod_{j = 1}^m \displaystyle\frac{\widehat{v} - \frac{f_{\eta; A} d_j}{b (1 - b)} - \frac{f_{\eta; A} d }{2 b (1 - b)} }{ \widehat{w} - \frac{f_{\eta; A} d_j}{b (1 - b)} - \frac{f_{\eta; A} d }{2 b (1 - b)} }.
\end{flalign}

\noindent In the case $K = V$, we apply  \eqref{betatj} and \eqref{psivnear} to find that 
\begin{flalign}
\label{productqvnear2}
\displaystyle\lim_{T \rightarrow \infty} \displaystyle\prod_{k = 1}^m \displaystyle\frac{\psi q^{-1} \beta_k^{-1} \widehat{v} + T^{1 / 3} f_{\eta; V} \big( q^{-1} \beta_k^{-1} \psi - 1 \big)  }{\psi q^{-1} \beta_k^{-1} \widehat{w} + T^{1 / 3} f_{\eta; V} \big( q^{-1} \beta_k^{-1} \psi - 1 \big)  } = \displaystyle\prod_{j = 1}^m \displaystyle\frac{\widehat{v} - \frac{f_{\eta; V } d_j}{b (1 - b)} - \frac{f_{\eta; V} d \kappa}{2 (\kappa - 1) b (1 - b) \Lambda} }{ \widehat{w} - \frac{f_{\eta; V} d_j}{b (1 - b)} - \frac{f_{\eta; V} d \kappa}{2 (\kappa - 1) b (1 - b) \Lambda} }. 
\end{flalign}

Now, in the case $K = A$, the lemma follows from \eqref{productqv2}, \eqref{productqanear}, and the definition of $c_{j; A}$ (given in the statement of Theorem \ref{asymmetriclimit}). Similarly, in the case $K = V$, the lemma follows from \eqref{productqv2}, \eqref{productqvnear2}, and the definition of $c_{j; V}$ (given in the statement of Theorem \ref{hlimit}). 
\end{proof}

\subsection{Recovery of the Baik-Ben-Arous-P\'{e}ch\'{e} Asymptotics} 

\label{ProofNear}

In what follows, we omit the subscript $K$ from our notation. 

The goal of this section is to establish the second parts of Theorem \ref{asymmetriclimit} and Theorem \ref{hlimit}. This will be similar to what was done in Section \ref{VertexRight}. In particular, we first analyze the kernel $K$ (defined in \eqref{vptright}) when $w \in \mathcal{C}^{(1)}$ and $v \in \Gamma^{(1)}$, that is, when both $w$ and $v$ are near $\psi$. Then, we analyze the kernel when either $w \in \mathcal{C}^{(2)}$ or $v \in \mathcal{C}^{(2)}$. 

Let us perform the former task first. To that end, recall the change of variables \eqref{wwvvertexright}, which also defined the kernel $\widehat{K}$. Furthermore, recall that, for any contour $\mathcal{D} \subset \mathbb{C}$, we set $\widehat{\mathcal{D}} = \sigma^{-1} \big( \mathcal{D} - \psi \big)$. In particular, from Definition \ref{linearvertexnear}, we find that $\mathcal{C}^{(1)} = \mathfrak{W}_{-E, \varepsilon / \sigma}$ and $\Gamma^{(1)} = \mathfrak{V}_{-E - 1, \varepsilon / \sigma}$; see Figure \ref{shiftedcontours13} for an example of these contours in the case $m = 3$.

\begin{figure}

\begin{center}

\begin{tikzpicture}[
      >=stealth,
			scale = .5	
			]
			
			\draw[<->] (-7, 0) -- (5, 0);
			\draw[<->] (0, -5) -- (0, 5.5);
			
			\draw[->,black,very thick] (-.5, -3.465) -- (-1.5, -1.732);
			\draw[-,black,very thick] (-1.5, -1.732) -- (-2.5, 0);
			\draw[-,black,very thick] (-1.5, 1.732) -- (-.5, 3.465) node [black, left = 12, above = 0] {$\mathfrak{W}_{- E, \infty}$};
			\draw[->,black,very thick] (-2.5, 0) -- (-1.5, 1.732);
			
			\draw[->,black,very thick] (-6, -3.465) -- (-5, -1.732);
			\draw[-,black,very thick] (-4, 0) -- (-5, -1.732);
			\draw[-,black,very thick]  (-5, 1.732) -- (-6, 3.465) node [black, above = 0] {$\mathfrak{V}_{- E - 1, \infty}$};
			\draw[->,black,very thick] (-4, 0) -- (-5, 1.732);
			
			\filldraw[fill=black, draw=black] (-2.5, 0) circle [radius=.1] node [black, left = 7, above = 0] {$- E$};
			
			\filldraw[fill=black, draw=black] (-.7, 0) circle [radius=.1] node [black, above = 0] {$-c_3$};
			
			\filldraw[fill=black, draw=black] (1, 0) circle [radius=.1] node [black, above = 0] {$-c_2$};
			
			\filldraw[fill=black, draw=black] (2.5, 0) circle [radius=.1] node [black, above = 0] {$-c_1$};

\end{tikzpicture}

\end{center}

\caption{\label{shiftedcontours13} Shown above are the contours $\mathfrak{W}_{- E, \infty}$ and $\mathfrak{V}_{- E - 1, \infty}$. }
\end{figure}

Now, we would like to establish the following lemma about the asymptotics of the kernel $\widehat{K}$. 

\begin{lem}

\label{nearvgamma1tildev}

There exist constants $c, C > 0$ such that 
\begin{flalign}
\label{vnear71} 
\left| \widehat{K} (\widehat{w}, \widehat{w'})  \right| \le \displaystyle\frac{C}{1 + | \widehat{w'} |}  \exp \big( c \widehat{w}^3 \big), 
\end{flalign} 

\noindent for each $\widehat{w} \in \widehat{\mathcal{C}}^{(1)}$ and $\widehat{w'} \in \widehat{\mathcal{C}}$. 

Furthermore, if we define the kernel $L_{s; \textbf{\emph{c}}} (\widehat{w}, \widehat{w'})$ by 
\begin{flalign}
\label{neardefinitionls}
L_{s; \textbf{\emph{c}}} \big( \widehat{w}, \widehat{w'} \big) = \displaystyle\frac{1}{2 \pi \textbf{\emph{i}}} \displaystyle\int_{\mathfrak{V}_{-1, \infty}} \displaystyle\frac{1}{\big( \widehat{v} - \widehat{w} \big) \big(  \widehat{w'} - \widehat{v} \big)} \exp \left( \displaystyle\frac{\widehat{w}^3}{3} - \displaystyle\frac{\widehat{v}^3}{3} + s \big( \widehat{w} - \widehat{v} \big) \right) \displaystyle\prod_{j = 1}^m \displaystyle\frac{\widehat{v} + c_j}{\widehat{w} + c_j} d \widehat{v}, 
\end{flalign}

\noindent then we have that
\begin{flalign}
\label{vnear1}
\displaystyle\lim_{T \rightarrow \infty} \widehat{K} \big( \widehat{w}, \widehat{w'} \big) = L_s \big( \widehat{w}, \widehat{w'} \big), 
\end{flalign} 

\noindent for each fixed $\widehat{w}\in \mathfrak{W}_{-E, \infty}$ and $\widehat{w'} \in \mathfrak{W}_{-E, \infty}$. 
\end{lem}

\noindent To establish this lemma, we recall from \eqref{vright111} that $\widehat{K}$ can be written as an integral of the function $I$ defined by \eqref{vright112}. Lemma \ref{nearvgamma1tildev} will follow from the large $T$ asymptotics of (and uniform estimate on) $I$ given by the following lemma. 

\begin{lem}

\label{nearuniformlimit13integrand}

There exist constants $c, C > 0$ such that
\begin{flalign}
\label{nearuniform13integrand}
\Big| I \big( \widehat{w}, \widehat{w'}; \widehat{v} \big) \Big| \le \displaystyle\frac{C}{1 + | \widehat{w'} |} \exp \big( c (\widehat{w}^3 - \widehat{v}^3 ) \big). 
\end{flalign}

\noindent for all $\widehat{w} \in \mathfrak{W}_{-E, \varepsilon / \sigma}$, $\widehat{w'} \in \widehat{\mathcal{C}}$, and $\widehat{v} \in \mathfrak{V}_{- E - 1, \varepsilon / \sigma}$. 

Furthermore, 
\begin{flalign}
\label{nearlimit13integrand}
\displaystyle\lim_{T \rightarrow \infty} I \big( \widehat{w}, \widehat{w'}; \widehat{v} \big) = \displaystyle\frac{1}{\big( \widehat{v} - \widehat{w} \big) \big( \widehat{w'} - \widehat{v} \big)} \exp \left( \displaystyle\frac{\widehat{w}^3}{3} - \displaystyle\frac{\widehat{v}^3}{3} + s \big( \widehat{w} - \widehat{v} \big) \right) \displaystyle\prod_{j = 1}^m \displaystyle\frac{\widehat{v} + c_j}{\widehat{w} + c_j} , 
\end{flalign}

\noindent for each fixed $\widehat{w}, \widehat{w'}, \widehat{v} \in \mathbb{C}$.  
\end{lem} 

\begin{proof}

The derivation of \eqref{nearuniform13integrand} is very similar to that of \eqref{uniform13integrand} and is thus omitted. 

Now let us establish the limiting statement \eqref{nearlimit13integrand}; this will also be similar to the proof of \eqref{limit13integrand}. Indeed, the five limit identities listed in \eqref{termsintegrandlimit13} still hold for all $\widehat{w}, \widehat{v} \in \mathbb{C}$. However, \eqref{productqv} no longer holds and must be replaced by Lemma \ref{productqvnear}. Multiplying the five limits in \eqref{termsintegrandlimit13} and the statement of Lemma \ref{productqvnear} yields \eqref{nearlimit13integrand}. 
\end{proof}

\noindent Now we can establish Lemma \ref{nearvgamma1tildev}. 

\begin{proof}[Proof of Lemma \ref{nearvgamma1tildev}]

The uniform estimate \eqref{vnear71} follows from integrating the estimate \eqref{nearuniform13integrand} over $\widehat{v} \in \mathfrak{V}_{- E - 1, \varepsilon / \sigma}$. Then, the limit \eqref{vnear1} follows from \eqref{nearlimit13integrand}, \eqref{nearuniform13integrand}, the fact that $\widehat{\Gamma}^{(1)}$ is contained in and converges to $\mathfrak{V}_{- E - 1, \infty}$ as $T$ tends to $\infty$, the exponential decay of the kernel $L_{s; \textbf{c}} \big( \widehat{w}, \widehat{w'} \big)$ in $\big| \widehat{w} \big|^3$, and the dominated convergence theorem. 
\end{proof}

Next, we analyze the integral \eqref{vptright} defining $K (w, w')$ when either $w$ or $v$ is not close to $\psi$, that is, when either $w \in \mathcal{C}^{(2)}$ or $v \in \Gamma^{(2)}$. 

Recalling the rescaled kernel $\overline{K}$ defined in \eqref{kbar13}, we obtain the following result, which is the analog of Corollary \ref{kclosekbarc2small}. 

\begin{cor}

\label{nearkclosekbarc2small}

There exist constants $c, C > 0$ such that 
\begin{flalign}
\label{nearkclosekbar}
\Big| \overline{K} \big( \widehat{w}, \widehat{w'} \big) - \widehat{K} \big( \widehat{w}, \widehat{w'} \big) \Big| < C \exp \big( c \big| \widehat{w} \big|^3 - cT \big), 
\end{flalign}

\noindent for all $\widehat{w} \in \widehat{\mathcal{C}}^{(1)}$ and $\widehat{w'} \in \widehat{\mathcal{C}} \cup \mathfrak{W}_{-E, \infty}$, and such that 
\begin{flalign}
\label{nearkbarsmallc2}
\Big| \overline{K} \big( \widehat{w}, \widehat{w'} \big) \Big| < C \exp \big( c \big| \widehat{w} \big| ^3 - cT \big), 
\end{flalign}

\noindent for all $\widehat{w} \in \widehat{\mathcal{C}}^{(2)}$ and $\widehat{w'} \in \widehat{\mathcal{C}} \cup \mathfrak{W}_{-E, \infty}$. 
\end{cor}

The proof of this corollary is very similar to that of Corollary \ref{kclosekbarc2small} and is thus omitted. 

Now we can use Corollary \ref{determinantlimitkernels}, Lemma \ref{nearvgamma1tildev}, and Corollary \ref{nearkclosekbarc2small} to establish the second parts of Proposition \ref{processdeterminant} and Proposition \ref{modeldeterminant}. 

\begin{proof}[Proof of Part 2 of Proposition \ref{processdeterminant} and Proposition \ref{modeldeterminant}]

This will be similar to the proof of the first parts of Proposition \ref{processdeterminant} and Proposition \ref{modeldeterminant}, given in Section \ref{ExponentialVertexRight}. In particular, it will follow from an application of Corollary \ref{determinantlimitkernels} to the sequence of kernels $\{ \overline{K} (\widehat{w}, \widehat{w'}) \}$, the sequence of contours $\{ \mathfrak{W}_{-E, \infty} \cup \widehat{\mathcal{C}}^{(2)} \}$, and the kernel $L_{s; \textbf{c}} (\widehat{w}, \widehat{w'})$ (given by \eqref{definitionls}).

The verification of the conditions of that corollary in this setting is similar to the verification provided in the proof of the first parts of Proposition \ref{processdeterminant} and Proposition \ref{modeldeterminant} in Section \ref{ExponentialVertexRight}. Indeed, the existence of the dominating function $\textbf{K}$ (satisfying the thre conditions of Corollary \ref{determinantlimitkernels}) follows from the uniform estimates \eqref{vnear71}, \eqref{nearkclosekbar}, and \eqref{nearkbarsmallc2}. 

Furthermore, the convergence $\lim_{N \rightarrow \infty} \textbf{1}_{\widehat{w}, \widehat{w'} \in \widehat{\mathcal{C}}} \overline{K} \big( \widehat{w}, \widehat{w'} \big) = \textbf{1}_{\widehat{w}, \widehat{w'} \in \mathfrak{W}_{-E, \infty}} L_{s; \textbf{c}} \big( \widehat{w}, \widehat{w'} \big)$, for each $\widehat{w}, \widehat{w'} \in \widehat{\mathcal{C}} \cup \mathfrak{W}_{-E, \infty}$ follows from \eqref{vright1} and \eqref{kclosekbar} (when $\widehat{w} \in \widehat{\mathcal{C}}^{(1)}$), and from \eqref{kbarsmallc2} and the exponential decay of $L_{s; \textbf{c}} \big( \widehat{w}, \widehat{w'} \big)$ in $\big| \widehat{w} \big|^3$ (when $\widehat{w} \in \widehat{\mathcal{C}}^{(2)}$ or $\widehat{w} \in \mathfrak{W}_{-E, \infty} \setminus\widehat{\mathcal{C}}^{(1)})$.

Now, applying Corollary \ref{determinantlimitkernels} yields that 
\begin{flalign}
\label{determinant4kernelnearc1c}
 \displaystyle\lim_{T \rightarrow \infty} \det \big( \Id + K \big)_{L^2 (\mathcal{C})} = \displaystyle\lim_{T \rightarrow \infty} \det \big( \Id + \overline{K} \big)_{L^2 (\widehat{\mathcal{C}})} = \det \big( \Id + L_{s; \textbf{c}} \big)_{L^2 (\mathfrak{W}_{-E, \infty})}, 
\end{flalign}

\noindent where we have used \eqref{rightvdeterminantkernelbar13} to deduce the first identity. 

Now, it is known (see, for instance, Lemma 8.7 of \cite{FEF}) that $\det \big( \Id + L_{s; \textbf{c}} \big)_{\mathfrak{W}_{-E, \infty}} = F_{\BBP; \textbf{c}} (s)$. Thus, we deduce from \eqref{determinant4kernelnearc1c} that 
\begin{flalign}
\label{determinant6kernelnearc1c}
\displaystyle\lim_{T \rightarrow \infty} \det \big( \Id + K \big)_{L^2 (\mathcal{C})} = F_{\BBP; \textbf{c}} (s). 
\end{flalign}

\noindent Applying \eqref{determinant6kernelnearc1c} in the case $K = A$ yields the second part of Proposition \ref{processdeterminant} and applying \eqref{determinant6kernelnearc1c} in the case $K = V$ yields the second part of Proposition \ref{modeldeterminant}. 
\end{proof}

\section{Gaussian Fluctuations} 

\label{Fluctuations12}

We now turn to the third parts of Proposition \ref{processdeterminant} and Proposition \ref{modeldeterminant}, whose proofs will be partly similar to those of the first two parts discussed previously. In Section \ref{VertexLeft}, we select contours $\mathcal{C}_K$ and $\Gamma_K$ that will be suitable for saddle point analysis. In Section \ref{KernelLeft}, we use these contours to implement this analysis, which will conclude the proofs of Proposition \ref{processdeterminant} and Proposition \ref{modeldeterminant}.

\subsection{The Contours \texorpdfstring{$\mathcal{C}_K$}{} and \texorpdfstring{$\Gamma_K$}{} for \texorpdfstring{$T^{1 / 2}$}{} Fluctuations} 

\label{VertexLeft}

Observe that $m_{\eta; K}$ and $f_{\eta; K}$ from the first and second parts of Proposition \ref{processdeterminant} and Proposition \ref{modeldeterminant} have now changed to $m_{\eta; K}'$ and $f_{\eta; K}'$, respectively. Thus, the function $G_K (z)$ (defined by \eqref{ga13} and \eqref{gv13} in the cases $K = A$ and $K = V$, respectively) changes, so we will change the contours $\mathcal{C}_K$ and $\Gamma_K$ as well. 

The goal of this section is to explain how to do this, which we will discuss further in Section \ref{ContoursVertexLeft}. However, before describing these contours, let us see how the identities in Section \ref{KExponentialRight} change with our new $m_{\eta; K}'$ and $f_{\eta; K}'$.

\subsubsection{A New Exponential Form for the Kernel \texorpdfstring{$K_{\zeta}^{(p_T)}$}{}}

\label{KExponentialLeft}

As in Section \ref{ckgammakcontours}, we rewrite the kernel $K (w, w') = K^{(p_T)} (w, w')$ given in \eqref{kvw1} through the identity 
\begin{flalign}
\begin{aligned}
\label{vptleft} 
K (w, w') = \displaystyle\frac{1}{2 \textbf{i}\log q} \displaystyle\sum_{j \in \mathbb{Z}} & \displaystyle\oint_{\Gamma_V} \displaystyle\frac{ \exp \left( T \big( G_K (w) - G_K (v) \big) \right) }{\sin \big( \pi (\log q)^{-1} (2 \pi \textbf{i}j + \log v - \log w) \big)} \displaystyle\frac{(q^{-1} \beta^{-1} v; q)_{\infty}^m}{(q^{-1} \beta^{-1} w; q)_{\infty}^m} \\
& \times \left( \displaystyle\frac{v}{w} \right)^{s f_{\eta; K}' T^{1 / 2}} \displaystyle\frac{dv}{v (w' - v)} ,
\end{aligned} 
\end{flalign}

\noindent where we have used the fact that all $b_j$ are equal to $b$ in the third parts of Proposition \ref{processdeterminant} and Proposition \ref{modeldeterminant}. In \eqref{vptleft}, $G_K (z)$ and $f_{\eta; K}'$ depend on $K$. Explicitly,  
\begin{flalign}
\label{ga12}
G_A (z) = \displaystyle\frac{q}{z + q} + \eta \log (z + q) + m_{\eta; A}' \log z; \qquad f_{\eta; A}' = \chi^{1 / 2} (\theta - \eta)^{1 / 2};
\end{flalign}
 
\begin{flalign}
\label{gv12}
G_V (z) = \eta \log (\kappa^{-1} z + q) - \log (z + q) + m_{\eta; V}' \log z; \qquad f_{\eta; V}' = \chi^{1 / 2} \left( 1 - \theta^{-1} \eta \right)^{1 / 2}. 
\end{flalign}

\noindent In \eqref{ga12} and \eqref{gv12}, 
\begin{flalign}
\label{m12}
m_{\eta; A}' = b(1 - b) - b \eta = \chi - b \eta; \qquad m_{\eta; V}' = b - \displaystyle\frac{b \eta}{b + \kappa (1 - b)} = b - b \Lambda^{-1} \eta, 
\end{flalign} 

\noindent where we recall the definition of $\Lambda$ from \eqref{processlocationtransition}. As in Section \ref{KExponentialRight}, we require the critical points of $G_K$. We find that 
\begin{flalign}
\label{derivativegaleft} 
G_A' (z) & = \left( \displaystyle\frac{(1 - b) (b + \eta) }{z (z + q)^2} \right) (z - \psi) ( z - \vartheta_A ); 
\end{flalign}

\begin{flalign}
\label{derivativegvleft} 
G_V' (z) & = \displaystyle\frac{1 - b}{z (z + q) (z + q \kappa )} \left( \Lambda^{-1} \kappa \eta - 1 \right) (z - \psi) ( z - \vartheta_V ), 
\end{flalign} 

\noindent where $\psi = \psi_A = \psi_V = q \beta$ and 
\begin{flalign}
\label{thetava}
\vartheta_A = q \left( \displaystyle\frac{1 - b - \eta}{b + \eta} \right); \qquad \vartheta_V = q \kappa \left( \displaystyle\frac{\Lambda - \eta}{\kappa \eta - \Lambda} \right). 
\end{flalign} 

In both cases, $\psi = q \beta$ is a critical point of $G_K$. Furthermore, observe that 
\begin{flalign*}
G_A'' (\psi) & = \displaystyle\frac{(1 - b)^3 (2b - 1 + \eta) }{q^2 b} = - \displaystyle\frac{f_{\eta; A}'^2}{\psi^2}; \qquad G_V'' (\psi) = \displaystyle\frac{(1 - b)^3}{q^2 b \Lambda^2} \left( \kappa \eta - \Lambda^2 \right) = - \displaystyle\frac{f_{\eta; V}'^2}{\psi^2}. 
\end{flalign*}

\noindent Thus, $G_K'' (\psi) = - \psi^{-2} f_{\eta; K}^2$ for $K \in \{ V, A \}$, which implies by a Taylor expansion that
\begin{flalign}
\begin{aligned}
\label{gzpsileft}
G_K (z) - G_K (\psi) & = - \displaystyle\frac{1}{2} \left( \displaystyle\frac{f_{\eta; K} (z - \psi)}{\psi} \right)^2 + R_K (z) \\
& = - \displaystyle\frac{1}{2} \left( \displaystyle\frac{f_{\eta; K} (z - \psi)}{\psi} \right)^2 + \mathcal{O} \big( |z - \psi|^3 \big), \qquad \qquad \text{as} \quad |z - \psi| \rightarrow 0. 
\end{aligned}
\end{flalign}

\noindent Here, $R_K (z) = G_K (z) - G_K (\psi) + \big( \psi^{-1} f_{\eta; K} (z - \psi) \big)^2 / 2$.

\subsubsection{New Choice of Contours \texorpdfstring{$\mathcal{C}_K$}{} and \texorpdfstring{$\Gamma_K$}{}} 

\label{ContoursVertexLeft}

The goal of this section is to exhibit contours $\mathcal{C}_K$ and $\Gamma_K$ satisfying the conditions of Proposition \ref{asymptoticheightvertex} and such that $\Re \big( G_K (w) - G_K (v) \big) < 0$ for $w \in \mathcal{C}_K$ and $v \in \Gamma_K$ both away from $\psi$. When $K = A$, we will only be able to do this when $\eta \ge -b$; when $K = V$, we will only be able to do this when $\eta \ge \kappa^{-1} \Lambda = \Lambda^{-1} \theta$. 

Throughout, we denote 
\begin{flalign}
\label{thetainfiniteva}
\alpha_A = -b; \qquad \alpha_V = \kappa^{-1} b + 1 - b = \kappa^{-1} \Lambda, 
\end{flalign}

\noindent as the values of $\eta$ at which $\vartheta_A$ and $\vartheta_V$ are infinite. According to our restriction on $\eta$, we have that $\eta > \alpha_K$; this implies that $\vartheta_K > q \beta$, for each $K \in \{ V, A \}$. In what follows, we will simplify notation by omitting the subscript $K$. 

As in Section \ref{ckgammakcontours}, we will take $\mathcal{C}$ and $\Gamma$ to follow level lines of the equation $\Re G(z) = G(\psi)$. However, since the function $G$ is different from the one used in Section \ref{ckgammakcontours}, we must perform a new analysis of these level lines; see Figure \ref{l1l2kleft}. 

Specifically, we require the following proposition, which we will establish in Section \ref{ContoursProcessLeft}. 

\begin{prop}

\label{l1l2left}

Assume that $\eta > \alpha$. Then, there exist three simple, closed curves, $\mathcal{L}_1 = \mathcal{L}_{1; K} $, $\mathcal{L}_2 = \mathcal{L}_{2; K} $, and $\mathcal{L}_3 = \mathcal{L}_{3; K}$ satisfying properties \ref{vzpsij}, \ref{vzpsiall}, \ref{anglev}, \ref{qkappaq0v}, \ref{positiverealv}, and \ref{negativerealv} of Proposition \ref{linesgv} (in the case $K = V$) and Proposition \ref{linesga} (in the case $K = A$), with properties \ref{containmentv} and \ref{psianglesv} respectively replaced  by the following. 

\begin{enumerate}

\item[4.]{ \label{containmentvleft} We have that $\mathcal{L}_1 \cap \mathcal{L}_2 = \{ \psi \}$, that $\mathcal{L}_2$ and $\mathcal{L}_3$ are disjoint, and that $\mathcal{L}_1$ and $\mathcal{L}_3$ are disjoint. Furthermore, $\mathcal{L}_1 \setminus \{ \psi \}$ is contained in the interior of $\mathcal{L}_2$, which is contained in the interior of $\mathcal{L}_3$. }

\item[6.]{ \label{psianglesvleft} The level line $\mathcal{L}_1$ meets the positive real axis (at $\psi$) at angles $3 \pi / 4$ and $- 3 \pi / 4$, and the level line $\mathcal{L}_2$ meets the positive real axis (at $\psi$) at angles $\pi / 4$ and $- \pi / 4$. }

\end{enumerate}

\end{prop}

\begin{rem}

When $\eta = \alpha$ or $\eta$ is slightly smaller than $\alpha$, the level line $\mathcal{L}_3$ from Proposition \ref{l1l2left} does not exist; this alone does not pose an issue for asymptotic analysis, since we will chose the contours $\mathcal{C}$ and $\Gamma$ to be very close to $\mathcal{L}_2$, instead of $\mathcal{L}_3$ (see Figure \ref{l1l2kleft}). More troublesome for us is that the contour $\mathcal{L}_2$ does not remain star-shaped, and also that it can contain $-q \kappa$ (in the case $K = V$), for $\eta$ sufficiently small. Then, we cannot choose $\mathcal{C}$ and $\Gamma$ to closely follow $\mathcal{L}_2$ without violating the restrictions imposed by Proposition \ref{asymptoticheightvertex}. 
\end{rem}

\begin{figure}[t] 
 \begin{minipage}{0.5\linewidth} 
  \centering

\begin{tikzpicture}[
      >=stealth,
      auto,
      style={
        scale = .8
      }
			]

			\draw[<->, black	] (0, -3.5) -- (0, 3.5) node[black, above = 0] {$\Im z$};
			\draw[<->, black] (-4.5, 0) -- (3.3, 0) node[black, right = 0] {$\Re z$};
			
			\draw[->,black, thick] (.8, -.6) -- (.8, .6);
			\draw[black, thick] (.8, .6) arc (35:180:1.1) node [black, above = 14, right = 10 	] {$\Gamma_V$}; 
			\draw[black, thick] (.8, -.6) arc (-35:-180:1.1);
			
			\draw[->, black, thick] (.9, 0) -- (1.48847, .11754); 
			\draw[->, black, thick] (.9, 0) -- (1.48847, -.11754); 
			\draw[black, thick] (1.48847, .11754) arc (4:180:1.5931)  node [black, above = 40, right = 10 	] {$\mathcal{C}_V$}; 
			\draw[black, thick] (1.48847, -.11754) arc (-4:-180:1.5931);

			\path[draw, dashed] (1, 0) -- (.96, .04) -- (.91, .08) -- (.85, .12) -- (.78, .16) -- (.69, .2) -- (.6, .23) -- (.51, .25) -- (.42, .27) -- (.33, .28) -- (.24, .27) -- (.15, .26) -- (.07, .23) -- (0, .2) -- (-.04, .16) -- (-.08, .14) -- (-.11, .07) -- (-.12, .02) -- (-.13, 0);
			
			\path[draw, dashed] (1, 0) -- (.96, -.04) -- (.91, -.08) -- (.85, -.12) -- (.78, -.16) -- (.69, -.2) -- (.6, -.23) -- (.51, -.25) -- (.42, -.27) -- (.33, -.28) -- (.24, -.27) -- (.15, -.26) -- (.07, -.23) -- (0, -.2) -- (-.04, -.16) -- (-.08, -.14) -- (-.11, -.07) -- (-.12, -.02) -- (-.13, 0);
			
			\path[draw, dashed] (1, 0) -- (1.05, .05) -- (1.07, .1) -- (1.11, .15) -- (1.14, .2) -- (1.15, .25) -- (1.17, .3) -- (1.18, .35) -- (1.18, .4) -- (1.18, .45) -- (1.13, .7) -- (1.11, .75) -- (1.08, .8) -- (1.05, .85) -- (1, .93) -- (.95, .98) -- (.91, 1.03) -- (.84, 1.09) -- (.76, 1.15) -- (.69, 1.2) -- (.6, 1.26) -- (.54, 1.29) -- (.46, 1.32) -- (.37, 1.35) -- (.29, 1.37) -- (.22, 1.39) -- (.12, 1.4) -- (.04, 1.41) -- (0, 1.41) -- (-.11, 1.4) -- (-.2, 1.39) -- (-.31, 1.37) -- (-.39, 1.35) -- (-.49, 1.32) -- (-.58, 1.29) -- (-.66, 1.25) -- (-.73, 1.22) -- (-.79, 1.18) -- (-.85, 1.13) -- (-.92, 1.08) -- (-.96, 1.03) -- (-1.03, .99) -- (-1.07, .92) -- (-1.12, .87) -- (-1.17, .8) -- (-1.21, .74) -- (-1.25, .66) -- (-1.29, .58) -- (-1.33, .5) -- (-1.35, .43) -- (-1.37, .36) -- (-1.39, .27) -- (-1.4, .2) -- (-1.41, .13) -- (-1.41, .06) -- (-1.42, 0);
			
			\path[draw, dashed] (1, 0) -- (1.05, -.05) -- (1.07, -.1) -- (1.11, -.15) -- (1.14, -.2) -- (1.15, -.25) -- (1.17, -.3) -- (1.18, -.35) -- (1.18, -.4) -- (1.18, -.45) -- (1.13, -.7) -- (1.11, -.75) -- (1.08, -.8) -- (1.05, -.85) -- (1, -.93) -- (.95, -.98) -- (.91, -1.03) -- (.84, -1.09) -- (.76, -1.15) -- (.69, -1.2) -- (.6, -1.26) -- (.54, -1.29) -- (.46, -1.32) -- (.37, -1.35) -- (.29, -1.37) -- (.22, -1.39) -- (.12, -1.4) -- (.04, -1.41) -- (0, -1.41) -- (-.11, -1.4) -- (-.2, -1.39) -- (-.31, -1.37) -- (-.39, -1.35) -- (-.49, -1.32) -- (-.58, -1.29) -- (-.66, -1.25) -- (-.73, -1.22) -- (-.79, -1.18) -- (-.85, -1.13) -- (-.92, -1.08) -- (-.96, -1.03) -- (-1.03, -.99) -- (-1.07, -.92) -- (-1.12, -.87) -- (-1.17, -.8) -- (-1.21, -.74) -- (-1.25, -.66) -- (-1.29, -.58) -- (-1.33, -.5) -- (-1.35, -.43) -- (-1.37, -.36) -- (-1.39, -.27) -- (-1.4, -.2) -- (-1.41, -.13) -- (-1.41, -.06) -- (-1.42, 0);

			\path[draw, dashed] (2.61, 0) -- (2.66, .15) -- (2.75, .3) -- (2.83, .45) -- (2.88, .6) -- (2.92, .75) -- (2.95, .9) -- (2.94, 1.05) -- (2.92, 1.2) -- (2.9, 1.35) -- (2.84, 1.5) -- (2.77, 1.65) -- (2.68, 1.8) -- (2.59, 1.95) -- (2.44, 2.1) -- (2.27, 2.25) -- (2.07, 2.4) -- (1.84, 2.55) -- (1.49, 2.7) -- (1.06, 2.85) -- (.2, 3) -- (-.07, 3.01) -- (-.32, 3) -- (-.5, 2.985) -- (-.7, 2.97) -- (-.9, 2.94) -- (-1.1, 2.91) -- (-1.3, 2.85) -- (-1.5, 2.79) -- (-1.7, 2.73) -- (-1.76, 2.7) -- (- 2.12, 2.55) -- (-2.42, 2.4) -- (-2.64, 2.25) -- (-2.85, 2.1) -- (-3.03, 1.95) -- (-3.18, 1.8) -- (-3.29, 1.65) -- (-3.43, 1.5) -- (-3.53, 1.35) -- (-3.63, 1.2) -- (-3.7, 1.05) -- (-3.77, .9) -- (-3.83, .75) -- (-3.87, .6) -- (-3.9, .45) -- (-3.93, .3) -- (-3.94, .15) -- (-3.95, 0);
			
			\path[draw, dashed] (2.61, 0) -- (2.66, -.15) -- (2.75, -.3) -- (2.83, -.45) -- (2.88, -.6) -- (2.92, -.75) -- (2.95, -.9) -- (2.94, -1.05) -- (2.92, -1.2) -- (2.9, -1.35) -- (2.84, -1.5) -- (2.77, -1.65) -- (2.68, -1.8) -- (2.59, -1.95) -- (2.44, -2.1) -- (2.27, -2.25) -- (2.07, -2.4) -- (1.84, -2.55) -- (1.49, -2.7) -- (1.06, -2.85) -- (.2, -3) -- (-.07, -3.01) -- (-.32, -3) -- (-.5, -2.985) -- (-.7, -2.97) -- (-.9, -2.94) -- (-1.1, -2.91) -- (-1.3, -2.85) -- (-1.5, -2.79) -- (-1.7, -2.73) -- (-1.76, -2.7) -- (- 2.12, -2.55) -- (-2.42, -2.4) -- (-2.64, -2.25) -- (-2.85, -2.1) -- (-3.03, -1.95) -- (-3.18, -1.8) -- (-3.29, -1.65) -- (-3.43, -1.5) -- (-3.53, -1.35) -- (-3.63, -1.2) -- (-3.7, -1.05) -- (-3.77, -.9) -- (-3.83, -.75) -- (-3.87, -.6) -- (-3.9, -.45) -- (-3.93, -.3) -- (-3.94, -.15) -- (-3.95, 0);

			\filldraw[fill=black, draw=black] (-.8, 0) circle [radius=.03] node [black, below = 0] {$-q$};
			
			\filldraw[fill=black, draw=black] (-2.2, 0) circle [radius=.03] node [black, below = 0] {$-q \kappa$};

\end{tikzpicture}

 \end{minipage}%
 \begin{minipage}{0.5\linewidth} 
  \centering 
	
	\begin{tikzpicture}[
      >=stealth,
			scale=.4
			]
			
			\draw[<->] (-4, 0) -- (4, 0);
			\draw[<->] (0, -4) -- (0, 4);
			
			\draw[-,black,very thick] (1, .8283) -- (3, 1.656) node[black, above = 0] {$\mathfrak{U}_{-1, \infty}$};
			\draw[->,black,very thick] (-1, 0) -- (1, .8283);
			\draw[-,black,very thick] (3, -1.656) -- (1, -.8283);
			\draw[->,black,very thick] (-1, 0) -- (1, -.8283);

			\draw[->,black,very thick] (-3, -2.5) -- (-3, -1.25);
			\draw[->,black,very thick] (-3, -1.25) -- (-3, 1.25);
			\draw[-,black,very thick] (-3, 1.25) -- (-3, 2.5) node[black, above = 0] {$\mathfrak{X}_{-2, \infty}$};
\end{tikzpicture}

 \end{minipage} 

\caption{\label{l1l2kleft} Above and to the left, the three level lines $\mathcal{L}_1$, $\mathcal{L}_2$, and $\mathcal{L}_3$ are depicted as dashed curves; the contours $\mathcal{C}$ and $\Gamma$ are depicted as solid curves and are labeled. Shown above and to the right are the contours $\mathfrak{U}_{-1, \infty}$ and $\mathfrak{X}_{-2, \infty}$. }
\end{figure}
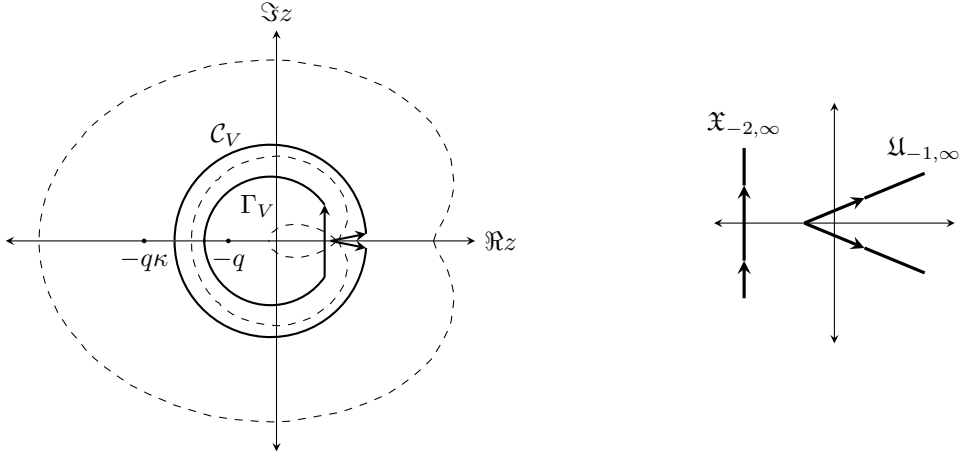
 
Assuming this proposition, let us define the contours $\mathcal{C}$ and $\Gamma$. As in Section \ref{ckgammakcontours}, they will each be the union of two contours, a small piecewise linear part near $\psi$ and a large curved part that closely follows the level line $\mathcal{L}_2$. 

Let us be more specific. In view of the fact that the level lines $\mathcal{L}_j$ now pass through $q \beta$ at angles $\pi / 4$, $3 \pi / 4$, $-3 \pi / 4$, and $- \pi / 4$, we must select different linear contours from those given by Definition \ref{crgammar}. Instead, we use the following contours; see Figure \ref{l1l2kleft}. 

\begin{definition}

\label{crgammarleft}

For a real number $r \in \mathbb{R}$ and a positive real number $\varepsilon > 0$, let $\mathfrak{U}_{r, \varepsilon}$ denote the piecewise linear curve in the complex plane that connects $r + \varepsilon e^{- \pi \textbf{i}/ 8}$ to $r$ to $r + \varepsilon e^{\pi \textbf{i}/ 8}$. Similarly, let $\mathfrak{X}_{r, \varepsilon}$ denote the linear curve in the complex plane that connects $r - i \varepsilon$ to $r + i \varepsilon$. 

\end{definition} 

Now, let us make the following definitions; Definition \ref{linearvertexright} and Definition \ref{curvedvertexright} define the piecewise linear and curved parts of the contours $\mathcal{C}$ and $\Gamma$, respectively. Then, Definition \ref{vertexrightcontours} defines the contours $\mathcal{C}$ and $\Gamma$. 

\begin{definition}

\label{linearvertexleft} 

Let $\mathcal{C}^{(1)} = \mathfrak{U}_{\psi - \psi f_{\eta}'^{-1} T^{-1 / 2}, \varepsilon}$ and $\Gamma^{(1)} = \mathfrak{X}_{\psi - 2 \psi f_{\eta}'^{-1} T^{-1 / 2}, \varepsilon}$, where $\varepsilon$ is chosen to be sufficiently small (independent of $T$) so that the first, second, and fourth properties listed in Definition \ref{linearvertexright} all hold for sufficiently large $T$. The third is replaced by the following. 

\begin{itemize}

\item{We have that $|R\big( \psi^{-1} f_{\eta} (z - \psi) \big)| < |f_{\eta} (z - \psi) / 3 \psi|^2$, for all $z \in \mathcal{C}^{(1)} \cup \Gamma^{(1)}$, where we recall the definition of $R$ from \eqref{gzpsileft}.}

\end{itemize}

Such a positive real number $\varepsilon$ is guaranteed to exist by part \ref{psianglesv} of Proposition \ref{l1l2left}, and also by the estimate \eqref{gzpsileft}. 

\end{definition}

\begin{definition}

\label{curvedvertexleft} 

Let $\mathcal{C}^{(2)}$ denote a positively oriented contour starting from the top endpoint $\psi - \psi f_{\eta}'^{-1} T^{-1 / 2} + \varepsilon e^{\pi \textbf{i}/ 8}$ of $\mathcal{C}^{(1)}$ to the bottom endpoint $\psi - \psi f_{\eta}'^{-1} T^{-1 / 2} + \varepsilon e^{- \pi \textbf{i}/ 8}$ of $\mathcal{C}^{(1)}$, and let $\Gamma^{(2)}$ denote a positively oriented contour starting from the top endpoint $\psi - 2 \psi f_{\eta}'^{-1} T^{-1 / 2} + i \varepsilon$ of $\Gamma^{(1)}$ to the bottom endpoint $\psi - 2 \psi f_{\eta}'^{-1} T^{-1 / 2} - i \varepsilon$ of $\Gamma^{(1)}$, satisfying the five properties listed in Definition \ref{curvedvertexright}. 

Such contours $\mathcal{C}^{(2)}$ and $\Gamma^{(2)}$ are guaranteed to exist by part \ref{anglev} and part \ref{qkappaq0v} of Proposition \ref{l1l2left}. 

\end{definition} 

\begin{definition}

\label{vertexleftcontours}

Set $\mathcal{C} = \mathcal{C}^{(1)} \cup \mathcal{C}^{(2)}$ and $\Gamma = \Gamma^{(1)} \cup \Gamma^{(2)}$. 

\end{definition}

Examples of the contours $\mathcal{C}$ and $\Gamma$ are depicted in Figure \ref{l1l2kleft}. It can be quickly verified that Lemma \ref{rightcagammanear} again holds with the contours $\mathcal{C}$ and $\Gamma$ defined above, meaning that we can use them for the proof of Proposition \ref{processdeterminant} and Proposition \ref{modeldeterminant} in Section \ref{KernelLeft}.

\subsubsection{Proof of Proposition \ref{l1l2left}} 

\label{ContoursProcessLeft}

The goal of this section is to establish Proposition \ref{l1l2left}. The proof of this proposition will be similar to that of Proposition \ref{linesga}. 

To that end, we have the following two lemmas, which are the analogs of Lemma \ref{realgarightbounded} and Lemma \ref{sixzgazgapsi}. In what follows, we omit the subscript $K$ from notation.

\begin{lem}

\label{realgaleftbounded} 

For all $\eta \in (\alpha, \theta)$, we have that $\lim_{z \rightarrow \infty} \big| \Re G(z) \big| = \infty$. In particular, the set of $z \in \mathbb{C}$ for which $\Re G(z) = G (\psi)$ is bounded.

\end{lem}

\begin{proof}

To establish this lemma, one Taylor expands $\Re G_K (z)$ in each case $K = V$ and $K = A$. Here, we assume that $K = V$, for the case $K = A$ is entirely analogous. Then, we obtain  
\begin{flalign*}
\Re G (z) = \eta \log |z + q \kappa| - \log |z + q| + m_{\eta}' \log |z| = (\eta + m_{\eta}' - 1) \log |z| +\mathcal{O} \left( |z|^{-1} \right). 
\end{flalign*}

\noindent Thus, $\big| \Re G (z) \big|$ tends to $\infty$ as $|z|$ tends to $\infty$; this implies the first, and hence also the second, statement of the lemma. 
\end{proof}

\begin{lem}

\label{sixzgazgapsileft}

Let $\ell \subset \mathbb{C}$ be any line through $0$. There exist at most $6$ complex numbers $z \in \ell$ such that $\Re G(z) = G(\psi)$. 

Moreover, if $K = A$ and $\ell$ is the real axis, then there exist only four such $z$; one is $z = \psi$, one is in $(\vartheta, \infty)$, one is in $(-q, 0)$, and one is in $(-\infty, -q)$. 

Similarly, if $K = V$ and $\ell$ is the real axis, then there exist only five such $z$; one is $z = \psi$, one is in $(\vartheta, \infty)$, one is in $(-q, 0)$, one is in $(-q \kappa, -q)$, and one is in $(-\infty, -q \kappa)$. 

\end{lem} 

The proof of this lemma is very similar to that of Lemma \ref{sixzgazgapsileft} and is thus omitted. 

Now we can establish Proposition \ref{l1l2left}. 

\begin{proof}[Proof of Proposition \ref{l1l2left}] 

We only consider the case $K = V$, since the proof is entirely analogous in the case $K = A$ (see also the proof of Proposition \ref{linesga} in Section \ref{PropositionLinesProcessRight} for the types of modifications necessary for the case $K = A$). In what follows, we omit the subscript $V$ to simplify notation. 

Let $S = \mathbb{C} \setminus  \{ \psi \}$. Then $\Re G(z)$ and $\Re \big( G' (z) \big)$ are smooth on $S \setminus  \{ -q \kappa, -q, 0, \vartheta  \}$. Furthermore, from \eqref{derivativegaleft}, we deduce that $G' (z)$ is non-zero on $S \setminus  \{ -q \kappa, -q, 0, \vartheta  \}$, which implies that $\Re G (z)$ has no critical points on $S \setminus  \{ -q \kappa, -q, 0, \vartheta  \}$. Thus, the implicit function theorem shows that the set of $z \in S$ satisfying $\Re G(z) = G (\psi)$ is a one-dimensional submanifold of $S$; here, we used the fact that $\Re G(z) \ne \Re G(\psi)$ for any $z \in \{ -q \kappa, -q, 0, \vartheta \}$ (where the fact that $G(\vartheta) \ne G(\psi)$ follows from the fact that $G'(z) < 0$ for all $z \in (\psi, \vartheta)$). In particular, the set of such $z$ is a union of connected components, each of which is homeomorphic to a line or circle.

By Lemma \ref{realgaleftbounded}, all of these connected components are bounded. Therefore, the closure $\overline{\mathcal{M}}$ (in $\mathbb{C}$) of any such component $\mathcal{M}$ is homeomorphic to a circle. In particular, by the maximum principle for harmonic functions, one of the singularities $-q \kappa$, $- q$, or $0$ of $\Re G(z)$ must lie on or in the interior of any closure $\overline{\mathcal{M}}$ (since $\Re G(z)$ harmonic for $z \in \mathbb{C} \setminus  \{ -q \kappa, -q, 0 \}$). 

Now, due to \eqref{gzpsileft}, four (not necessarily distinct) components exit from $\psi$, at angles $\pi / 4$, $3 \pi / 4$, $-3 \pi / 4$, and $- \pi / 4$. Lemma \ref{sixzgazgapsileft} also shows that there exist $z_1 \in (\vartheta, \infty)$, $z_2 \in (-q, 0)$, $z_3 \in (-q \kappa, -q)$, and $z_4 \in (-\infty, -q \kappa)$ such that $\Re G(z_1) = \Re G(z_2) = \Re G(z_3) = \Re G(z_4) = G(\psi)$; this implies that one of the components passes through each of the points $z_1$, $z_2$, $z_3$, and $z_4$. The second part of Lemma \ref{sixzgazgapsileft} states that there are no other real numbers $z \notin \{ z_1, z_2, z_3, z_4, \psi \}$ such that $\Re G(z) = G(\psi)$. 

These facts imply that there are three one-dimensional components $\mathcal{M}_1, \mathcal{M}_2, \mathcal{M}_3 \subset S$ such that the set of $z \in S$ satisfying $\Re G(z) = G (\psi)$ is equal to the union $\mathcal{M}_1 \cup \mathcal{M}_2 \cup \mathcal{M}_3$; these components can be described as follows. First, we have that $\psi \in \overline{\mathcal{M}_1}, \overline{\mathcal{M}_2}$ and $z_1 \in \mathcal{M}_3$. Furthermore, $z_2 \in \mathcal{M}_1$, $z_3 \in \mathcal{M}_2$, and $z_4 \in \mathcal{M}_3$. We can describe the intersections $\overline{\mathcal{M}_1} \cap \mathbb{R} = \{ z_2, \psi \}$, $\overline{\mathcal{M}_2} \cap \mathbb{R} = \{ z_3, \psi \}$, and $\mathcal{M}_3 \cap \mathbb{R} = \{ z_1, z_4 \}$. 

Now set $\mathcal{L}_1 = \overline{\mathcal{M}_1}$, $\mathcal{L}_2 = \overline{\mathcal{M}_2} $, and $\mathcal{L}_3 = \mathcal{M}_3$. Then, $\mathcal{L}_1$, $\mathcal{L}_2$, and $\mathcal{L}_3$ are all simple, closed curves. Parts \ref{vzpsij} of Proposition \ref{vzpsiall} are satisfied by these curves since $\mathcal{M}_1 \cup \mathcal{M}_2 \cup \mathcal{M}_3 \cup \{ \psi \} = \mathcal{L}_1 \cup \mathcal{L}_2 \cup \mathcal{L}_3$. 

Furthermore, only one positive real number and only one negative real number lie on each of the curves $\overline{\mathcal{M}_1}$, $\overline{\mathcal{M}_2}$, and $\mathcal{M}_3$. This implies that $0$ is in the interior of $\mathcal{L}_1$, $\mathcal{L}_2$, and $\mathcal{L}_3$. 

Now, property \ref{anglev} holds from Lemma \ref{sixzgazgapsileft}; indeed, this lemma implies that any line $\ell \subset \mathbb{C}$ can only intersect $\mathcal{L}_1 \cup \mathcal{L}_2 \cup \mathcal{L}_3$ in at most six places. Since each $\mathcal{L}_j$ is a closed curve containing $0$ in its interior, $\ell$ intersects each $\mathcal{L}_j$ at least, and thus only, twice; furthermore, $0$ is between these two intersection points. This implies that for each angle $a \in \mathbb{R}$ and each index $j \in \{ 1, 2, 3 \}$, there is only one complex number $z_a \in \mathcal{L}_j$ such that $z_a / |z_a| = e^{\textbf{i} a}$. This verifies property \ref{anglev}. 

To see property \ref{containmentv}, observe that $\mathcal{L}_i$ and $\mathcal{L}_j$ can only intersect at critical points of $G(z)$, for any distinct $i, j \in \{1, 2, 3 \}$. These critical points are $\psi$, $\vartheta$, $0$, $-q$, and $-q \kappa$, but none of the $\mathcal{L}_j$ pass through $\vartheta$, $0$, $-q$, or $-q \kappa$. Now, the containment follows from the fact that $\mathcal{L}_1$ intersects the negative real axis only once at $z_2$, that $\mathcal{L}_2$ intersects the negative real axis only once at $z_3$, that $\mathcal{L}_3$ intersects the negative real axis only once at $z_4$, and that $z_4 < z_3 < z_2$. 

Property \ref{qkappaq0v} follows from the fact that $z_4 < - \kappa q < z_3 < -q < z_2$, and property \ref{psianglesv} follows from the containment property \ref{containmentv} and the estimate \eqref{gzpsileft}. 

Furthermore, observe that $\Re G(z)$ can only change sign across some $\mathcal{L}_j$; thus, property \ref{positiverealv} and property \ref{negativerealv} follow since $\lim_{z \rightarrow -q} \Re G \big( z \big) = \infty$ and $\lim_{z \rightarrow 0} \Re G \big( z \big) = - \infty$. 
\end{proof}

\subsection{Proof of Part 3 of Proposition \ref{processdeterminant} and Proposition \ref{modeldeterminant}} 

\label{KernelLeft}

The goal of this section is to establish the third parts of Proposition \ref{processdeterminant} and Proposition \ref{modeldeterminant}, which exhibit Gaussian-type asymptotic fluctuations in the case $\eta \in (\alpha, \theta)$. As in Section \ref{VertexRight} and Section \ref{ProofNear}, the proofs of both statements are similar, so we will do them simultaneously. Thus, recall the notation $K \in \{ V, A \}$ that corresponds to the stochastic six-vertex model when $K = V$, and to the ASEP when $K = A$. In what follows, we simplify notation by omitting the subscript $K$, and we always assume that $\eta \in (\alpha, \theta)$. 

In either case $K \in \{ V, A \}$, let $x = \eta T + 1$; as in Section \ref{VertexRight}, the replacement of $\lfloor \eta T \rfloor + 1$ (which was in the statement of Proposition \ref{processdeterminant} and Proposition \ref{modeldeterminant}) by $\eta T + 1$ will not affect the asymptotics. Furthermore, let $p_T = f_{\eta}' s T^{1 / 2} - m_{\eta}' T$, where we recall that $m_{\eta}'$ is defined through \eqref{m12} and $f_{\eta}'$ are defined through \eqref{ga12} and \eqref{gv12} in the cases $K = A$ and $K = V$, respectively. 

The proofs of the third parts of Proposition \ref{processdeterminant} and Proposition \ref{modeldeterminant} will be very similar the proofs of the first parts given in Section \ref{VertexRight} in the case $\eta > \theta$. Specifically, we will analyze the contribution to the integral on the right side of \eqref{vptleft} when $v$ is integrated in a small neighborhood of $\psi$; this will yield a Gaussian-type kernel. Then, in Section \ref{ExponentialVertexLeft}, we will show that the contribution to the integral decays exponentially when $v$ is integrated outside a neighborhood of $\psi$; we will then be able to ignore this contribution, which will lead to the proofs of the third parts of Proposition \ref{processdeterminant} and Proposition \ref{modeldeterminant}.

\subsubsection{Contribution Near \texorpdfstring{$\psi$}{}} 

\label{VertexLeftPsi}

In this section, we analyze the contribution to the kernel $K (w, w')$, when $w \in \mathcal{C}^{(1)}$ and when $v$ is integrated along $\Gamma^{(1)}$. This will be similar to what was done in Section \ref{VertexRightPsi}. 

Similarly to \eqref{vzetatilderight}, define 
\begin{flalign}
\begin{aligned}
\label{vzetatildeleft}
\widetilde{K} (w, w') = \displaystyle\frac{1}{2 \textbf{i} \log q} \displaystyle\sum_{j = -\infty}^{\infty} & \displaystyle\int_{\Gamma^{(1)}} \displaystyle\frac{ \exp \left( T \big( G(w) - G(v) \big) \right)} {\sin \big( \frac{\pi}{\log q} (2 \pi \textbf{i} j + \log v - \log w) \big)} \displaystyle\frac{(q^{-1} \beta^{-1} v; q)_{\infty}^m}{(q^{-1} \beta^{-1} w; q)_{\infty}^m} \\
& \times \left( \displaystyle\frac{v}{w} \right)^{s f_{\eta}' T^{1 / 2}} \displaystyle\frac{dv}{v (w' - v)}, 
\end{aligned}
\end{flalign}

\noindent for any $w, w' \in \mathcal{C}$. 

Now, let us change variables. Denote $\sigma = \psi f_{\eta}'^{-1} T^{-1 / 2}$, and set 
\begin{flalign}
\label{wwvvertexleft}
w = \psi + \sigma \widehat{w}; \qquad w' = \psi + \sigma \widehat{w'}; \qquad v = \psi + \sigma \widehat{v}; \qquad \widehat{K} (\widehat{w}, \widehat{w'}) = \sigma \widetilde{K} (w, w').
\end{flalign} 

\noindent Also, for any contour $\mathcal{D}$, set $\widehat{\mathcal{D}} = \sigma^{-1} \big( \mathcal{D} - \psi \big)$. 

The following lemma provides the asymptotics for the kernel $\widehat{K}$. 

\begin{lem}

\label{leftvgamma1tildev}

There exist constants $c, C > 0$ such that 
\begin{flalign}
\label{vleft71} 
\left| \widehat{K} (\widehat{w}, \widehat{w'})  \right| \le \displaystyle\frac{C}{1 + |\widehat{w'}|} \exp \big( - c |\widehat{w}|^2 \big), 
\end{flalign} 

\noindent for each $\widehat{w} \in \widehat{\mathcal{C}}^{(1)}$ and $\widehat{w'} \in \widehat{\mathcal{C}}$. 

Furthermore, if we define the kernel $L_{s; m} (\widehat{w}, \widehat{w'})$ by 
\begin{flalign}
\label{leftdefinitionls}
L_{s; m} \big( \widehat{w}, \widehat{w'} \big) = \displaystyle\frac{1}{2 \pi \textbf{\emph{i}}} \displaystyle\int_{\mathfrak{X}_{-2, \infty}} \displaystyle\frac{1}{\big( \widehat{v} - \widehat{w} \big) \big( \widehat{w'} - \widehat{v} \big)} \exp \left( \displaystyle\frac{\widehat{v}^2}{2} - \displaystyle\frac{\widehat{w}^2}{2} + s \big( \widehat{v} - \widehat{w} \big) \right) \left( \displaystyle\frac{\widehat{v}}{\widehat{w}} \right)^m d \widehat{v}, 
\end{flalign}

\noindent then we have that
\begin{flalign}
\label{vleft1}
\displaystyle\lim_{T \rightarrow \infty} \widehat{K} \big( \widehat{w}, \widehat{w'} \big) = L_{s; m} \big( \widehat{w}, \widehat{w'} \big), 
\end{flalign} 

\noindent for each fixed $\widehat{w}, \widehat{w'} \in \mathfrak{U}_{-1, \infty}$. 
\end{lem}

\noindent To establish this lemma, we first rewrite the kernel $\widehat{K}$. Analogously to \eqref{vright111} and \eqref{vright112}, we have that 
\begin{flalign}
\label{vleft111}
 \widehat{K} (\widehat{w}, \widehat{w'}) & = \displaystyle\frac{1}{2 \pi \textbf{i}}\displaystyle\int_{\widehat{\Gamma}^{(1)}} I \big( \widehat{w}, \widehat{w'}; \widehat{v} \big) d \widehat{v}, 
\end{flalign}

\noindent where 
\begin{flalign}
\begin{aligned}
\label{vleft112}
I \big( \widehat{w}, \widehat{w'}; \widehat{v} \big) & = \displaystyle\frac{1}{\big( 1 + \psi^{-1} \sigma \widehat{v} \big) \big(  \widehat{w'} - \widehat{v} \big)  } \exp \left( \displaystyle\frac{\widehat{v}^2 - \widehat{w}^2}{2} + T \big( R (T^{-1 / 2} \widehat{w}) - R (T^{-1 / 2} \widehat{v}) \big)  \right)   \\
& \qquad \times \left( \displaystyle\frac{1 + \psi^{-1} \sigma \widehat{v}}{1 + \psi^{-1} \sigma \widehat{w}} \right)^{\psi \sigma^{-1} s } \displaystyle\frac{\big(q^{-1} \beta^{-1} ( \psi + \sigma \widehat{v}  ); q \big)_{\infty}^m }{\big( q^{-1} \beta^{-1} (\psi + \sigma \widehat{w} ); q \big)_{\infty}^m } \\
& \qquad \times \displaystyle\frac{\pi \psi^{-1} \sigma}{\log q} \displaystyle\sum_{j = -\infty}^{\infty}  \displaystyle\frac{1}{\sin \Big( \frac{\pi}{\log q} \big( 2 \pi \textbf{i} j + \log (1 + \psi^{-1} \sigma \widehat{v} )  - \log (1 + \psi^{-1} \sigma \widehat{w} ) \big) \Big)}. 
\end{aligned} 
\end{flalign}

\noindent Lemma \ref{leftvgamma1tildev} will follow from the large $T$ asymptotics of (and uniform estimate on) $I$, given by the following lemma. 

\begin{lem}

\label{leftuniformlimit13integrand}

There exist constants $c, C > 0$ such that
\begin{flalign}
\label{leftuniform13integrand}
\Big| I \big( \widehat{w}, \widehat{w'}; \widehat{v} \big) \Big| \le \displaystyle\frac{C}{1 + |\widehat{w'}|} \exp \big( - c (|\widehat{v}|^2 + |\widehat{w}|^2 ) \big). 
\end{flalign}

\noindent for all $\widehat{w} \in \mathfrak{U}_{-1, \varepsilon / \sigma}$, $\widehat{w'} \in \widehat{\mathcal{C}} \cup \mathfrak{U}_{-1, \infty}$, and $\widehat{v} \in \mathfrak{X}_{-2, \varepsilon / \sigma}$. 

Furthermore, 
\begin{flalign}
\label{leftlimit13integrand}
\displaystyle\lim_{T \rightarrow \infty} I \big( \widehat{w}, \widehat{w'}; \widehat{v} \big) = \displaystyle\frac{1}{\big( \widehat{v} - \widehat{w} \big) \big( \widehat{w'} - \widehat{v}  \big)} \exp \left( \displaystyle\frac{\widehat{v}^2}{2} - \displaystyle\frac{\widehat{w}^2}{2} + s \big( \widehat{w} - \widehat{v} \big) \right) \left( \displaystyle\frac{\widehat{v}}{\widehat{w}} \right)^m, 
\end{flalign}

\noindent for each fixed $\widehat{w}, \widehat{w'}, \widehat{v} \in \mathbb{C}$.  
\end{lem} 

\begin{proof}

The proof of this lemma will be similar to that of Lemma \ref{uniformlimit13integrand}. 

Let us first establish the uniform estimate \eqref{uniform13integrand}. To that end, observe that there exists a constant $C_1 > 0$ such that the inequalities 
\begin{flalign}
\begin{aligned}
\label{lefttermsintegrand13}
& \left| \displaystyle\frac{1}{1 + \psi^{-1} \sigma \widehat{v}} \right| < C_1; \qquad \displaystyle\frac{1}{\widehat{v} - \widehat{w'}} < \displaystyle\frac{C_1}{1 + |\widehat{w}'|}; \qquad \left| \displaystyle\frac{1 + \psi^{-1} \sigma \widehat{v}}{1 + \psi^{-1} \sigma \widehat{w}} \right|^{\psi \sigma^{-1} s} \le C_1 \exp \Big( C_1 \big( |\widehat{w}| + |\widehat{v}| \big) \Big) ;  \\
& \qquad \qquad \displaystyle\frac{\pi \psi^{-1} \sigma}{\big| \log q \big|} \displaystyle\sum_{j \ne 0} \left| \displaystyle\frac{1}{\sin \Big( \frac{\pi}{\log q} \big( 2 \pi \textbf{i} j + \log (1 + \psi^{-1} \sigma \widehat{v} )  - \log (1 + \psi^{-1} \sigma \widehat{w} ) \big) \Big)} \right| \le C_1 T^{-1 / 2} ;  \\
& \qquad \left| \left( \displaystyle\frac{\pi \psi^{-1} \sigma}{\log q} \right) \displaystyle\frac{1}{\sin \Big( \frac{\pi}{\log q} \big( \log (1 + \psi^{-1} \sigma \widehat{v} )  - \log (1 + \psi^{-1} \sigma \widehat{w} ) \big) \Big)} \right| < C_1 \exp \big( C_1 (|\widehat{w}| + \widehat{v}|) \big);  \\
& \qquad \qquad \qquad \left| \displaystyle\frac{\big(q^{-1} \beta^{-1} ( \psi + \sigma \widehat{v}  ); q \big)_{\infty}^m}{\big( q^{-1} \beta^{-1} (\psi + \sigma \widehat{w} ); q \big)_{\infty}^m} \right| < C_1 \exp \Big( C_1 \big( |\widehat{w}| + |\widehat{v}| \big) \Big), 
\end{aligned} 	
\end{flalign}

\noindent all hold for each  $\widehat{w} \in \mathfrak{U}_{-1, \varepsilon / \sigma}$, $\widehat{w'} \in \widehat{\mathcal{C}}$, and $\widehat{v} \in \mathfrak{X}_{-2, \varepsilon / \sigma}$. The verification of these estimates is entirely analogous to the verification of \eqref{termsintegrand13} and is thus omitted. 

The estimates \eqref{lefttermsintegrand13} address all terms on the right side of \eqref{vleft112}, except for the exponential term. To analyze this term, first recall $\big| R(z) \big| < |z|^2 / 9$, for all $z \in \widehat{\mathcal{C}}^{(1)} \cup \widehat{\Gamma}^{(1)}$; this was stipulated in Definition \ref{linearvertexleft}. Thus, increasing $C_1$ if necessary, we obtain that 
\begin{flalign}
\begin{aligned}
\label{vleft10}
\Bigg| \exp \left( \displaystyle\frac{\widehat{v}^2}{2} - \displaystyle\frac{\widehat{w}^2}{2} + T \big( R (T^{- 1 / 2} \widehat{w}) - R (T^{- 1 / 2} \widehat{v}) \big) \right) \Bigg| & = \exp \left( \displaystyle\frac{\widehat{v}^2}{2} - \displaystyle\frac{\widehat{w}^2}{2} + \displaystyle\frac{|\widehat{w}|^2}{9} + \displaystyle\frac{|\widehat{v}|^2}{9} \right)  \\
& < C_1 \exp \left( \displaystyle\frac{|\widehat{v}|^2}{9} - \displaystyle\frac{|\widehat{w}|^2}{9}\right) . 
\end{aligned} 
\end{flalign}

\noindent In \eqref{vleft10}, the last estimate follows from the facts that there exist constants $C_2$ and $C_3$ such that $\widehat{v}^2 < C_2 - \big| \widehat{v} \big|^2 / 4$ and $\widehat{w}^2 > C_2 + \big| \widehat{w} \big|^2 / 4$, for all $\widehat{w} \in \widehat{\mathcal{C}}^{(1)}$ and $\widehat{v} \in \widehat{\Gamma}^{(1)}$. 

Now, the estimate \eqref{leftuniform13integrand} follows from the definition \eqref{vleft112} of $I$, the six estimates \eqref{lefttermsintegrand13}, and the exponential estimate \eqref{vleft10}. 

In order to establish the limiting statement \eqref{leftlimit13integrand}, we observe that 
\begin{flalign}
\begin{aligned}
\label{lefttermsintegrandlimit13}
& \displaystyle\lim_{T \rightarrow \infty} \displaystyle\frac{1}{1 + \psi^{-1} \sigma \widehat{v}} = 1; \qquad \displaystyle\lim_{T \rightarrow \infty} \left( \displaystyle\frac{1 + \psi^{-1} \sigma \widehat{v}}{1 + \psi^{-1} \sigma \widehat{w}} \right)^{\psi \sigma^{-1} s} = \exp \big( s (\widehat{v} - \widehat{w} ) \big);  \\
& \qquad \qquad  \displaystyle\lim_{T \rightarrow \infty} \exp \left( \displaystyle\frac{\widehat{v}^2 - \widehat{w}^2}{2} + T \big( R (T^{-1 / 2} \widehat{w}) - R (T^{-1 / 2} \widehat{v}) \big)  \right) = \exp \left( \displaystyle\frac{\widehat{v}^2}{2} - \displaystyle\frac{\widehat{w}^2}{2} \right);  \\
& \qquad \quad \displaystyle\lim_{T \rightarrow \infty} \displaystyle\frac{\pi \psi^{-1} \sigma}{\log q} \displaystyle\sum_{j \ne 0}  \displaystyle\frac{1}{\sin \Big( \frac{\pi}{\log q} \big( 2 \pi \textbf{i} j + \log (1 + \psi^{-1} \sigma \widehat{v} )  - \log (1 + \psi^{-1} \sigma \widehat{w} ) \big) \Big)} = 0;  \\
& \qquad \displaystyle\lim_{T \rightarrow \infty} \left( \displaystyle\frac{\pi \psi^{-1} \sigma}{\log q} \right) \displaystyle\frac{1}{\sin \Big( \frac{\pi}{\log q} \big( \log (1 + \psi^{-1} \sigma \widehat{v} )  - \log (1 + \psi^{-1} \sigma \widehat{w} ) \big) \Big)} = \displaystyle\frac{1}{\widehat{v} - \widehat{w}} ; 
\end{aligned} 
\end{flalign}

\noindent and 
\begin{flalign}
\label{leftproductqv}  
\displaystyle\lim_{T \rightarrow \infty}  \displaystyle\frac{\big(q^{-1} \beta^{-1} ( \psi + \sigma \widehat{v}  ); q \big)_{\infty}^m}{\big( q^{-1} \beta^{-1} (\psi + \sigma \widehat{w} ); q \big)_{\infty}^m}  = \left( \displaystyle\frac{\widehat{v}}{\widehat{w}} \right)^m, 
\end{flalign}

\noindent for each $\widehat{w}, \widehat{w'}, \widehat{v} \in \mathbb{C}$. The verification of the five identities \eqref{lefttermsintegrandlimit13} is very similar to that of \eqref{termsintegrandlimit13} and is thus omitted. The limit \eqref{leftproductqv} follows from the fact that $\psi = q \beta$ and that $\sigma$ tends to $0$ as $T$ tends to $\infty$. 

Now, \eqref{leftlimit13integrand} follows from multiplying the results of \eqref{lefttermsintegrandlimit13} and \eqref{leftproductqv}. 
\end{proof}

\noindent We can now establish Lemma \ref{leftvgamma1tildev}. 

\begin{proof}[Proof of Lemma \ref{leftvgamma1tildev}]

The uniform estimate \eqref{vleft71} follows from integrating the estimate \eqref{leftuniform13integrand} over $\widehat{v} \in \mathfrak{X}_{-2, \varepsilon / \sigma}$. Then, the limit \eqref{vleft1} follows from \eqref{leftlimit13integrand}, \eqref{leftuniform13integrand}, the fact that $\widehat{\Gamma}^{(1)}$ is contained in and converges to $\mathfrak{X}_{-2, \infty}$ as $T$ tends to $\infty$, the exponential decay of the kernel $L_{s; m}$ in $\big| \widehat{v} \big|^2$, and the dominated convergence theorem. 
\end{proof}

\subsubsection{Exponential Decay of \texorpdfstring{$K$}{} on \texorpdfstring{$\mathcal{C}^{(2)}$}{} and \texorpdfstring{$\Gamma^{(2)}$}{}}

\label{ExponentialVertexLeft}

In this section we analyze the integral \eqref{vptleft} defining $K (w, w')$ when either $w \in \mathcal{C}^{(2)}$ or $v \in \Gamma^{(2)}$. In this case, we show that the integral decays exponentially in $T$, which will allow us to ignore this contribution and apply the results of Section \ref{VertexLeftPsi}. 

We begin with the following lemma, which is the analog of Lemma \ref{rightcv2gammav2exponential} and estimates the exponential term in the integrand on the right side of \eqref{vptleft}. 

\begin{lem}

\label{leftcv2gammav2exponential}

There exists some positive real number $c_1 > 0$, independent of $T$, such that 
\begin{flalign*}
\max \left\{ \sup_{\substack{w \in \mathcal{C} \\ v \in \Gamma^{(2)}}} \Re \big( G(w) - G(v) \big), \sup_{\substack{w \in \mathcal{C}^{(2)} \\ v \in \Gamma}} \Re \big( G(w) - G(v) \big) \right\} < - c_1.
\end{flalign*}
\end{lem}

\begin{proof}
From Proposition \ref{l1l2left}, we find that $\Re \big( G(w) - G(v) \big) < 0$ for all $w \in \mathcal{C}$ and $v \in \Gamma^{(2)}$ and that $\Re \big( G(w) - G(v) \big) < 0$ for all $w \in \mathcal{C}^{(2)}$ and $v \in \Gamma$, where both estimates are uniform in $T$. Thus, the existence of the claimed constant $c_1$ follows from the compactness of the contours $\mathcal{C}$ and $\Gamma$. 
\end{proof}

\noindent Now, recall the definitions of $\sigma = \psi f_{\eta}'^{-1} T^{-1 / 2}$; $\widehat{w}$, $\widehat{w'}$, and $\widehat{v}$ from \eqref{wwvvertexleft}; and $I$ from \eqref{vleft112}. Define 
\begin{flalign}
\label{leftkbar13}
\overline{K} \big( \widehat{w}, \widehat{w'} \big) = \sigma K (w, w') = \displaystyle\frac{1}{2 \pi \textbf{i}} \displaystyle\int_{\widehat{\Gamma}} I \big( \widehat{w}, \widehat{w'}; \widehat{v} \big) d \widehat{v},
\end{flalign}  	

\noindent for each $\widehat{w}, \widehat{w'} \in \mathcal{C}$. From the change of variables \eqref{wwvvertexleft}, we have that 
\begin{flalign}
\label{leftvdeterminantkernelbar13}
\det \big( \Id + K \big)_{L^2 (\mathcal{C})} = \det \big( \Id + \overline{K} \big)_{L^2 (\widehat{\mathcal{C}})}. 
\end{flalign}

\noindent Given Lemma \ref{leftcv2gammav2exponential}, the proof of the following result is entirely analogous to that of Corollary \ref{kclosekbarc2small} and Corollary \ref{nearkclosekbarc2small} and is thus omitted.

\begin{cor}

\label{leftkclosekbarc2small}

There exist constants $c, C > 0$ such that 
\begin{flalign}
\label{leftkclosekbar}
\Big| \overline{K} \big( \widehat{w}, \widehat{w'} \big) - \widehat{K} \big( \widehat{w}, \widehat{w'} \big) \Big| < C \exp \big( - c (|T + |\widehat{w}|^2) \big), 
\end{flalign}

\noindent for all $\widehat{w} \in \widehat{\mathcal{C}}^{(1)}$ and $\widehat{w'} \in \widehat{\mathcal{C}}$, and such that 
\begin{flalign}
\label{leftkbarsmallc2}
\Big| \overline{K} \big( \widehat{w}, \widehat{w'} \big) \Big| < C \exp \big( - c (T + |\widehat{w}|^2) \big), 
\end{flalign}

\noindent for all $\widehat{w} \in \widehat{\mathcal{C}}^{(2)}$ and $\widehat{w'} \in \widehat{\mathcal{C}}$. 
\end{cor}

\noindent Now we can use Corollary \ref{determinantlimitkernels}, Lemma \ref{leftvgamma1tildev}, and Corollary \ref{leftkclosekbarc2small} to establish the third parts of Proposition \ref{processdeterminant} and Proposition \ref{modeldeterminant}. 

\begin{proof}[Proof of Part 3 of Proposition \ref{processdeterminant} and Proposition \ref{modeldeterminant}]

This will be very similar to the proof of the first (and second) parts of Proposition \ref{processdeterminant} and Proposition \ref{modeldeterminant}, given in Section \ref{ExponentialVertexRight} (and Section \ref{ProofNear}). In particular, it will follow from an application of Corollary \ref{determinantlimitkernels} to the sequence of kernels $\{ \overline{K} (\widehat{w}, \widehat{w'}) \}$, the sequence of contours $\{ \mathfrak{U}_{-1, \infty} \cup \widehat{\mathcal{C}}^{(2)} \}$, and the kernel $L_{s; m} (\widehat{w}, \widehat{w'})$ (given by \eqref{leftdefinitionls}).

The verification of the conditions of that corollary in this setting is very similar to the verification provided in the proofs of the first (in Section \ref{ExponentialVertexRight}) and second (in Section \ref{ProofNear}) parts of Proposition \ref{processdeterminant} and Proposition \ref{modeldeterminant}, and is thus omitted. Applying Corollary \ref{determinantlimitkernels}, we obtain that 
\begin{flalign}
\label{determinant4kernelleftc1c}
 \displaystyle\lim_{T \rightarrow \infty} \det \big( \Id + K \big)_{L^2 (\mathcal{C})} = \displaystyle\lim_{T \rightarrow \infty} \det \big( \Id + \overline{K} \big)_{L^2 (\widehat{\mathcal{C}})} = \det \big( \Id + L_{s; m} \big)_{L^2 (\mathfrak{U}_{-1, \infty})}, 
\end{flalign}

\noindent where we have used \eqref{leftvdeterminantkernelbar13} to deduce the first identity. 

Now, it is known (see, for instance, Proposition 5 of \cite{PTQT}) that $\det \big( \Id + L_{s; m} \big)_{\mathfrak{U}_{-1, \infty}} = G_m (s)$. Thus, we deduce from \eqref{determinant4kernelrightc1c} that 
\begin{flalign}
\label{determinant6kernelleftc1c}
\displaystyle\lim_{T \rightarrow \infty} \det \big( \Id + K \big)_{L^2 (\mathcal{C})} = G_m (s). 
\end{flalign}

\noindent Applying \eqref{determinant6kernelrightc1c} in the case $K = A$ yields the third part of Proposition \ref{processdeterminant} and applying \eqref{determinant6kernelrightc1c} in the case $K = V$ yields the third part of Proposition \ref{modeldeterminant}. 
\end{proof}

\appendix

\section{Fredholm Determinants}

\label{Determinants1} 

In this section, we provide the definition and some properties of Fredholm determinants in the form that is convenient for us. 

\begin{definition}

\label{definitiondeterminant}

Fix a contour $\mathcal{C} \subset \mathbb{C}$ in the complex plane, and let $K: \mathcal{C} \times \mathcal{C} \rightarrow \mathbb{C}$ be a meromorphic function with no poles on $\mathcal{C} \times \mathcal{C}$. We define the \emph{Fredholm determinant}
\begin{flalign}
\label{determinantsum}
\det \big( \Id + K \big)_{L^2 (\mathcal{C})} = 1 + \displaystyle\sum_{k = 1}^{\infty} \displaystyle\frac{1}{ (2 \pi \textbf{i} )^k k!} \displaystyle\int_{\mathcal{C}} \cdots \displaystyle\int_{\mathcal{C}} \det \big[ K(x_i, x_j) \big]_{i, j = 1}^k \displaystyle\prod_{j = 1}^k d x_j. 
\end{flalign}

\end{definition}

\begin{rem}

\label{sumconvergesdeterminant}

Generally, the definition of the Fredholm determinant requires that $K$ give rise to a trace-class operator on $L^2 (\mathcal{C})$. We will not use that convention here; instead, we only require that the sum on the right side \eqref{determinantsum} be convergent. 

\end{rem}

\begin{rem}

\label{determinantcontourdeform}

Suppose that the contour $\mathcal{C}$ can be continuously deformed to another contour $\mathcal{C}' \subset \mathbb{C}$ without crossing any poles of the kernel $K$. Then, $\det \big( \Id + K \big)_{L^2 (\mathcal{C})} = \det \big( \Id + K \big)_{L^2 (\mathcal{C}')}$, assuming that the right side of \eqref{determinantsum} remains uniformly convergent throughout the deformation process. Indeed, this can be seen directly from the definition \eqref{determinantsum}, since each summand on the right side of this identity remains the same after deforming $\mathcal{C}$ to $\mathcal{C}'$. Alternatively, see Proposition 1 of \cite{AASIC}. 

\end{rem}

The Fredholm determinant satisfies several stability properties that are useful for asymptotic analysis. We record some of these properties here. 

The lemma below quantifies the change in a Fredholm determinant after perturbing its kernel; this specific estimate will also be useful in the parallel work \cite{CFCLEP}.

\begin{lem}

\label{determinantclosekernels}

Adopt the notation of Definition \ref{definitiondeterminant}, and let $K_1 (z, z'): \mathcal{C} \times \mathcal{C} \rightarrow \mathbb{C}$ be another meromorphic function with no poles on $\mathcal{C} \times \mathcal{C}$. Suppose that there exists a function $\textbf{\emph{K}} : \mathcal{C} \rightarrow \mathbb{R}_{\ge 0}$ such that $\sup_{z' \in \mathcal{C}} \big| K (z, z') \big| \le \textbf{\emph{K}} (z)$ and $\sup_{z' \in \mathcal{C}} \big| K_1 (z, z') \big| \le \textbf{\emph{K}} (z)$ for all $z \in \mathcal{C}$, and also such that $\displaystyle\int_{\mathcal{C}} \big| \textbf{\emph{K}} (z) \big| dz < \infty$. 

Then, the Fredholm determinants $ \det \big( \Id + K \big)_{L^2 (\mathcal{C})} $ and $ \det \big( \Id + K_1 \big)_{L^2 (\mathcal{C})}$ converge absolutely (as series given by the right side of \eqref{determinantsum}). 

Furthermore, denote $K' (z, z') = K_1 (z, z') - K (z, z')$, and define the constant 
\begin{flalign}
\begin{aligned}
\label{cdeterminant2}
C = \displaystyle\sum_{k = 1}^{\infty} \displaystyle\frac{2^k k^{k / 2}}{(k - 1)!} \displaystyle\int_{\mathcal{C}} \cdots & \displaystyle\int_{\mathcal{C}} \displaystyle\prod_{i = 2}^k  \left| \displaystyle\frac{1}{k} \displaystyle\sum_{j = 1}^k \Big( \big| K (x_i, x_j) \big|^2 + \big| K' (x_i, x_j) \big|^2  \Big)  \right|^{1 / 2}  \\
& \qquad \times \left| \displaystyle\frac{1}{k} \displaystyle\sum_{j = 1}^k \big| K' (x_1, x_j) \big|^2 \right|^{1 / 2}  \displaystyle\prod_{i = 1}^k dx_i.
\end{aligned}
\end{flalign}

\noindent Then,
\begin{flalign*}
\Big| \det \big( \Id + K_1 \big)_{L^2 (\mathcal{C})} - \det \big( \Id + K \big)_{L^2 (\mathcal{C})}  \Big| < C.
\end{flalign*} 

\end{lem}

\begin{proof}

Let us begin with the first statement. Denote $\int_{\mathcal{C}} \big| \textbf{K} (z) \big| dz = B$. Applying Definition \ref{definitiondeterminant}, we obtain that
\begin{flalign}
\begin{aligned}
\label{kdeterminantfinite}
\Big| \det \big( \Id + K \big)_{L^2 (\mathcal{C})} \Big| & \le \displaystyle\sum_{k = 1}^{\infty} \displaystyle\frac{1}{k!} \displaystyle\int_{\mathcal{C}} \cdots \displaystyle\int_{\mathcal{C}} \left| \det \big[ K (x_i, x_j) \big]_{i, j = 1}^k \right| \displaystyle\prod_{j = 1}^k d x_j \\
& \le \displaystyle\sum_{k = 1}^{\infty} \displaystyle\frac{1}{k!} \displaystyle\int_{\mathcal{C}} \cdots \displaystyle\int_{\mathcal{C}} \displaystyle\prod_{i = 1}^k \left| \displaystyle\sum_{j = 1}^k \big| K (x_i, x_j) \big|^2 \right|^{1 / 2} \displaystyle\prod_{j = 1}^k d x_j \\
& \le \displaystyle\sum_{k = 1}^{\infty} \displaystyle\frac{1}{k!} \displaystyle\int_{\mathcal{C}} \cdots \displaystyle\int_{\mathcal{C}} \displaystyle\prod_{i = 1}^k \big( k^{1 / 2} \textbf{K} (x_i) \big) \displaystyle\prod_{j = 1}^k d x_j \le \displaystyle\sum_{k = 1}^{\infty} \displaystyle\frac{B^k k^{k / 2}}{k!},
\end{aligned}
\end{flalign}

\noindent which is finite. To deduce the second estimate in \eqref{kdeterminantfinite}, we applied Hadamard's identity. This establishes the absolute convergence of $\det \big( \Id + K \big)_{L^2 (\mathcal{C})}$; the proof for $\det \big( \Id + K_1 \big)_{L^2 (\mathcal{C})}$ is similar. 

Next, let us establish the second statement of the lemma. Applying Definition \ref{definitiondeterminant} and the absolute and uniform convergence of all sums and integrands (guaranteed by \eqref{kdeterminantfinite}) to commute integration with summation, we deduce that 
\begin{flalign}
\begin{aligned}
\label{determinantk1k2}
\Big| \det \big( \Id & + K_1 \big)_{L^2 (\mathcal{C})} - \det \big( \Id + K \big)_{L^2 (\mathcal{C})} \Big|  \\ 
& \le  \displaystyle\sum_{k = 1}^{\infty} \displaystyle\frac{1}{k!} \displaystyle\int_{\mathcal{C}} \cdots \displaystyle\int_{\mathcal{C}} \left| \det \big[ K_1 (x_i, x_j) \big]_{i, j = 1}^k - \det \big[ K (x_i, x_j) \big]_{i, j = 1}^k \right| \displaystyle\prod_{j = 1}^k d x_j.
\end{aligned}
\end{flalign}

\noindent The $k$-fold integral on the right side of \eqref{determinantk1k2} can be written as 
\begin{flalign}
\begin{aligned}
\label{determinantsmallk1} 
& \displaystyle\int_{\mathcal{C}} \cdots \displaystyle\int_{\mathcal{C}} \left| \det \big[ K (x_i, x_j) + K' (x_i, x_j) \big]_{i, j = 1}^k - \det \big[ K (x_i, x_j) \big]_{i, j = 1}^k \right| \displaystyle\prod_{i = 1}^k d x_i \\
& \qquad  = \displaystyle\int_{\mathcal{C}} \cdots \displaystyle\int_{\mathcal{C}} \bigg| \displaystyle\sum_{\substack{S \subseteq \{ 1, 2, \ldots , k \} \\ |S| \ge 1}} \det \textbf{K}^{(S)} \bigg| \displaystyle\prod_{i = 1}^k d x_i \le \displaystyle\int_{\mathcal{C}} \cdots \displaystyle\int_{\mathcal{C}} \displaystyle\sum_{j = 1}^k \displaystyle\sum_{\substack{S \subseteq \{ 1, 2, \ldots , k \} \\ j \in S}} \big| \det \textbf{K}^{(S)} \big| \displaystyle\prod_{i = 1}^k d x_i \\
& \qquad \qquad \qquad \qquad \qquad \qquad \qquad \qquad \qquad \quad = k \displaystyle\int_{\mathcal{C}} \cdots \displaystyle\int_{\mathcal{C}} \displaystyle\sum_{\substack{S \subseteq \{ 1, 2, \ldots , k \} \\ 1 \in S }} \big| \det \textbf{K}^{(S)} \big| \displaystyle\prod_{i = 1}^k d x_i , 
\end{aligned} 
\end{flalign}

\noindent where $\textbf{K}^{(S)}$ is the $k \times k$ matrix whose $(i, j)$-entry is $K (x_i, x_j)$ if $i \notin S$ and is $K' (x_i, x_j)$ if $i \in S$. The last identity was deduced by symmetry. 

From Hadamard's inequality, we have that 
\begin{flalign}
\label{determinantsmallk12}
\big| \det \textbf{K}^{(S)} \big| \le  \left| \displaystyle\frac{1}{k} \displaystyle\sum_{j = 1}^k \big| K' (x_1, x_j) \big|^2 \right|^{1 / 2} \displaystyle\prod_{i = 2}^k \left| \displaystyle\frac{1}{k} \displaystyle\sum_{j = 1}^k \Big( \big| K (x_i, x_j) \big|^2 + K' (x_i, x_j) \big|^2 \Big) \right|^{1 / 2} ,
\end{flalign} 

\noindent if $1 \in S$. Now the lemma follows from substituting \eqref{determinantsmallk12} into \eqref{determinantsmallk1}, observing that there are at most $2^k$ possible choices for $S$ in \eqref{determinantsmallk1}, and substituting the result into \eqref{determinantk1k2}. 	
\end{proof}

\noindent Using Lemma \ref{determinantclosekernels}, we can establish the following result for convergence of Fredholm determinants.

\begin{cor}

\label{determinantlimitkernels}

Adopt the notation of Definition \ref{definitiondeterminant}. Let $\{ \mathcal{C}_T^{(2)} \}_{T \in \mathbb{R}_{> 0}}$ be a family of contours in the complex plane, and denote $\mathcal{C}_T = \mathcal{C} \cup \mathcal{C}_T^{(2)}$ for each positive real number $T$. For each $T > 0$, let $K_T: \mathcal{C}_T \times \mathcal{C}_T \rightarrow \mathbb{C}$ be a meromorphic function with no poles on $\mathcal{C}_T \times \mathcal{C}_T$.  

Assume that there exists a function $\textbf{\emph{K}} : \mathbb{C} \rightarrow \mathbb{R}_{\ge 0}$ such that the following three properties hold. 

\begin{itemize}

\item{For each $T > 0$ and $z \in \mathcal{C}_T$, we have that $\sup_{z' \in \mathcal{C}_T} \big| K_T (z, z') \big| < \textbf{\emph{K}} (z)$.}

\item{ There exists a constant $B > 0$ such that $\displaystyle\int_{\mathcal{C}_T} \big| \textbf{\emph{K}} (z) \big| dz < B$, for each $T > 0$.}

\item{ For each positive real number $\varepsilon > 0$, there exists some real number $M = M_{\varepsilon}$ such that $\displaystyle\int_{\mathcal{C}_T^{(2)}} \big| \textbf{\emph{K}} (z) \big| dz < \varepsilon$, for all $T > M$. }

\end{itemize} 

\noindent If $\lim_{T \rightarrow \infty} \textbf{\emph{1}}_{z, z' \in \mathcal{C}_T} K_T (z, z') = \textbf{\emph{1}}_{z, z' \in \mathcal{C}} K (z, z')$ for each fixed $z, z' \in \mathbb{C}$, then 
\begin{flalign*}
\displaystyle\lim_{T \rightarrow \infty} \Big( \det \big( \Id + K \big)_{L^2 (\mathcal{C})} - \det \big( \Id + K_T \big)_{L^2 (\mathcal{C}_T)} \Big) = 0.
\end{flalign*} 
\end{cor}

\begin{proof}

In what follows, we extend $K$ and each $K_T$ by zero to $\mathbb{C} \times \mathbb{C}$. Applying Lemma \ref{determinantclosekernels} with the kernels $K_T (z, z')$ and $K (z, z')$, and the contour $\mathcal{C}_T$, we obtain that 
\begin{flalign}
\begin{aligned}
\label{exponentialestimatekdeterminantkernellimit1}
& \Big| \det \big( \Id + K_T \big)_{L^2 (\mathcal{C}_T)} - \det \big( \Id + K \big)_{L^2 (\mathcal{C})}  \Big| \\
& \quad \le \displaystyle\sum_{k = 1}^{\infty} \displaystyle\frac{2^k k^{k / 2}}{(k - 1)!} \displaystyle\int_{\mathcal{C} \cup \mathcal{C}_T^{(2)}} \cdots \displaystyle\int_{\mathcal{C} \cup \mathcal{C}_T^{(2)}} \left| \displaystyle\frac{1}{k} \displaystyle\sum_{j = 1}^k \big| K (x_1, x_j) -  K_T (x_1, x_j) \big|^2 \right|^{1 / 2} \\
& \quad \qquad \qquad \qquad \times \displaystyle\prod_{i = 2}^k  \left| \displaystyle\frac{1}{k} \displaystyle\sum_{j = 1}^k \Big( \big| K (x_i, x_j) \big|^2 + \big| K (x_i, x_j) - K_T (x_i, x_j) \big|^2 \Big) \right|^{1 / 2} \displaystyle\prod_{i = 1}^k d x_i \\
& \quad \le \displaystyle\sum_{k = 1}^{\infty} \displaystyle\frac{2^k k^{k / 2}}{(k - 1)!} \displaystyle\int_{\mathcal{C} \cup \mathcal{C}_T^{(2)}} \cdots \displaystyle\int_{\mathcal{C} \cup \mathcal{C}_T^{(2)}} \left| \displaystyle\frac{1}{k} \displaystyle\sum_{j = 1}^k \big| K (x_1, x_j) - K_T (x_1, x_j) \big|^2 \right|^{1 / 2}  \\
& \quad \qquad \qquad \qquad \qquad \qquad \qquad \qquad \times  \displaystyle\prod_{i = 2}^k  \left| 5 \big| \textbf{K} (x_i) \big|^2 \right|^{1 / 2} \displaystyle\prod_{i = 1}^k d x_i, 
\end{aligned}
\end{flalign}

\noindent where in the last estimate we used the fact that $\big| K (x_i, x_j) \big|^2 + \big| K (x_i, x_j) - K_T (x_i, x_j) \big|^2 \le 5 \big| \textbf{K} (x_i) \big|^2$. 

For each positive integer $k$, the $k$-fold integral on the right side of \eqref{exponentialestimatekdeterminantkernellimit1} can be written as the sum of $2^k$ integrals, in which each $x_i$ is either integrated along $\mathcal{C}$ or $\mathcal{C}_T^{(2)}$. Let us first estimate one of these integrals, in which all of the $x_i$ are integrated along $\mathcal{C}$. 

To that end, we have that 
\begin{flalign}
\label{sumdeterminantterm1}
\displaystyle\lim_{T \rightarrow \infty} \displaystyle\int_{\mathcal{C}}  \cdots \displaystyle\int_{\mathcal{C}} \left| \displaystyle\frac{1}{k} \displaystyle\sum_{j = 1}^k \big| K (x_1, x_j) - K_T (x_1, x_j) \big|^2 \right|^{1 / 2}  \displaystyle\prod_{i = 2}^k  \left| 5 \big| \textbf{K} (x_i) \big|^2 \right|^{1 / 2} \displaystyle\prod_{i = 1}^k d x_i = 0,
\end{flalign}

\noindent where we have used the pointwise convergence of $K_T$ to $K$, the integrability of $\textbf{K}$ along the contour $\mathcal{C}$, and the dominated convergence theorem. 

Next, to estimate each of the other $2^k - 1$ integrals, we observe that 
\begin{flalign}
\label{sumdeterminantterm2}
\left| \displaystyle\frac{1}{k} \displaystyle\sum_{j = 1}^k \big| K (x_1, x_j) - K_T (x_1, x_j) \big|^2 \right|^{1 / 2}  \displaystyle\prod_{i = 2}^k  \left| 5 \big| \textbf{K} (x_i) \big|^2 \right|^{1 / 2} \le 5^{k / 2} \displaystyle\prod_{i = 1}^k \big| \textbf{K} (x_i) \big|.
\end{flalign}

\noindent Thus, from the third condition on $\textbf{K}$ (given in the statement of the corollary) and \eqref{sumdeterminantterm2}, we obtain
\begin{flalign}
\label{sumdeterminantterm3}
\displaystyle\int  \cdots \displaystyle\int \left| \displaystyle\frac{1}{k} \displaystyle\sum_{j = 1}^k \big| K (x_1, x_j) - K_T (x_1, x_j) \big|^2 \right|^{1 / 2}  \displaystyle\prod_{i = 2}^k  \left| 5 \big| \textbf{K} (x_i) \big|^2 \right|^{1 / 2} \displaystyle\prod_{i = 1}^k d x_i < 5^{k / 2} B^{k - 1} \varepsilon,
\end{flalign}

\noindent if $T > M_{\varepsilon}$ and at least one $x_j$ is integrated along $\mathcal{C}_T^{(2)}$. 

Now the lemma follows from substituting \eqref{sumdeterminantterm1} and \eqref{sumdeterminantterm3} into \eqref{exponentialestimatekdeterminantkernellimit1}, letting $T$ tend to $\infty$, applying the dominated convergence theorem, and letting $\varepsilon$ tend to $0$.
\end{proof}

\section{Asymptotics Through Schur Measures}

\label{Asymptotics2}

In this section we outline an alternative route to establishing Theorem \ref{hlimit}, which is based on a comparison between the stochastic higher spin vertex models and Schur measures. This type of comparison was first observed in \cite{SHSVMMM} to provide a new proof of the GUE Tracy-Widom current fluctuations in the stochastic six-vertex model with step initial data. Here, we outline how to modify that proof so that it applies to the stochastic six-vertex model with generalized step-Bernoulli initial data. Since our work here will be similar to what was done in \cite{SHSVMMM}, we will not go through all details carefully. 

First, we require some notation on symmetric functions; see Chapter 1 of Macdonald's book \cite{SFP} for a more thorough introduction. Denote the set of all partitions by $\mathbb{Y}$, and fix an infinite set of variables $X = (x_1, x_2, \ldots )$. Let $\Sym = \Sym_X$ denote the algebra of symmetric functions in $X$. For any $\lambda \in \mathbb{Y}$, let $h_{\lambda} (X) \in \Sym$ and $s_{\lambda} (X) \in \Sym$ denote the \emph{complete symmetric function} and the \emph{Schur function} associated with $\lambda$, respectively. A \emph{specialization} of $\Sym$ is a homomorphism $\rho: \Sym \rightarrow \mathbb{C}$. We call $\rho$ \emph{Schur non-negative} if $s_{\lambda} (\rho) = \rho \big( s_{\lambda} (X) \big) \ge 0$ for each $\lambda \in \mathbb{Y}$. 

Since the complete symmetric functions $h_1 (X), h_2 (X), \ldots $ generate the algebra $\Sym$, any specialization $\rho$ is determined by its action on the formal generating series 
\begin{flalign*}
F_X (z) = \displaystyle\sum_{n = 0}^{\infty} h_n (X) z^n = \displaystyle\prod_{j = 1}^{\infty} \displaystyle\frac{1}{1 - x_i z}.
\end{flalign*}

In this section we focus on specializations that are unions of those known as \emph{$\alpha$-specializations} and \emph{$\beta$-specializations}. Specifically, fix (possibly infinite) sets of non-negative parameters $\mathfrak{A} = (\mathfrak{a}_1, \mathfrak{a}_2, \ldots )$ and $\mathfrak{B} = (\mathfrak{b}_1, \mathfrak{b}_2, \ldots )$ such that $\sum_{j = 1}^{\infty} \big( \mathfrak{a}_j + \mathfrak{b}_j \big) < \infty$.\footnote{This is non-standard notation; these sets are typically denoted $\alpha = (\alpha_1, \alpha_2, \ldots )$ and $\beta = (\beta_1, \beta_2, \ldots )$. However, that would conflict with the notation for the parameters $\beta_j = b_j / (1 - b_j)$, given by \eqref{stochasticparameters}.} Define the specialization $\rho = \rho_{\mathfrak{A}; \mathfrak{B}}$ by setting 
\begin{flalign*}
F_{\rho} (z) = \rho \big( F_X (z) \big) = \displaystyle\prod_{j = 1}^{\infty} \displaystyle\frac{1 + \mathfrak{b}_j z }{1 - \mathfrak{a}_j z}. 
\end{flalign*}

The convergence of the sum $\sum_{j = 1}^{\infty} \big( \mathfrak{a}_j + \mathfrak{b}_j \big)$ guarantees that this specialization exists; furthermore, the branching rule for Schur polynomials implies that specialization $\rho$ is Schur non-negative. 

The \emph{Schur measure} was introduced by Okounkov in \cite{IWRP} as a measure on $\mathbb{Y}$ weighted by products of Schur functions. We will define one of its specializations that is of particular interest to us. In what follows, we recall the definitions of $\kappa$, $q$, $\Lambda$, $s$, $u$, and $\beta_j$ from Theorem \ref{hlimit} and \eqref{stochasticparameters}. 

Let $\eta > 0$ be a real number; let $N > 0$ be an integer; and set $x = x_N = \lfloor \eta N \rfloor + 1$. Define $\mathfrak{A} = (s, s, \ldots , s)$, which consists of $x - 1$ copies of $s$; $\mathfrak{B} = \bigcup_{j = 1}^m \big\{ u \beta_j^{-1}, q u \beta_j^{-1}, q^2 u \beta_j^{-1}, \ldots \big\}$; and $Y = \big( u^{-1}, u^{-1}, \ldots , u^{-1} \big)$, which consists of $N$ copies of $u^{-1}$. 

Now set $\rho = \rho_{\mathfrak{A}; \mathfrak{B}}$, and define the probability measure $\textbf{SM}$ on $\mathbb{Y}$ that assigns weight 
\begin{flalign*}
\textbf{SM} (\lambda) = \textbf{SM}_{\rho; Y} \big( \{ \lambda \} \big) = Z^{-1} s_{\lambda} (\rho) s_{\lambda} (Y)
\end{flalign*}

\noindent to each $\lambda \in \mathbb{Y}$, where $Z = \prod_{j = 1}^n F_{\rho} (y_j)$ is the normalization constant chosen to ensure that $\sum_{\lambda \in \mathbb{Y}} \textbf{SM} (\lambda) = 1$ (this follows from the \emph{Cauchy identity}). 

Our interest in this measure is given by the following proposition. In what follows, $\ell (\lambda)$ denotes the \emph{length} (number of non-zero parts) of the partition $\lambda$, and we recall the notion of \emph{asymptotic equivalence} from Definition 4.2 of \cite{SHSVMMM}.

\begin{prop}

\label{lengthpartitionheight} 

Sample a partition $\lambda$ randomly from the Schur measure $\textbf{\emph{SM}}$ defined above, and consider the stochastic six-vertex model from Theorem \ref{hlimit}. Then, the sequences of random variables $\big\{ \mathfrak{H} (\eta N, N) \big\}_{N \ge 1}$ and $\big\{ N - \ell (\lambda) \big\}_{N \ge 1}$ are asymptotically equivalent. 

\end{prop}

\begin{proof}

This is a direct consequence of Corollary 4.11 of \cite{SHSVMMM}, after recalling the discussion from Remark \ref{columnsgeneralizedinitial} and Section \ref{GeneralizedSpecialization} that recasts the stochastic six-vertex model with generalized step Bernoulli initial data as a certain stochastic higher spin vertex model with step initial data. We omit further details. 
\end{proof}

Let us explain how Proposition \ref{lengthpartitionheight} can be used to establish Theorem \ref{hlimit}. As an example, we illustrate how this works for the second part of Theorem \ref{hlimit}, specialized to the case $\textbf{c} = 0^m = (0, 0, \ldots , 0)$. 

To that end, we set $\eta = \kappa^{-1} \Lambda^2$, $b_1 = b_2 = \cdots = b_m = b$, and $\beta_1 = \beta_2 = \cdots = \beta_m = \beta$. We would like to show that  
\begin{flalign*}
\displaystyle\lim_{N \rightarrow \infty} \mathbb{P} \left[ \displaystyle\frac{m N - \mathfrak{H} (\eta N, N)}{f N^{1 / 3}} \le s \right] = F_{\BBP; 0^m} (s),
\end{flalign*}

\noindent where $m = m_{\eta; V} = m_{\kappa^{-1} \Lambda^2; V}$ is defined by \eqref{m13}, and $f = f_{\eta; V} = f_{\kappa^{-1} \Lambda^2; V}$ is defined by \eqref{gv13}. 

In view of Proposition \ref{lengthpartitionheight}, it suffices to establish that 
\begin{flalign}
\label{lengthpartitionasymptotic}
\displaystyle\lim_{N \rightarrow \infty} \mathbb{P} \left[ \displaystyle\frac{(m - 1) N + \ell (\lambda)}{f N^{1 / 3}} \le s \right] = F_{\BBP; 0^m} (s).  
\end{flalign}

The remainder of this section will be devoted to a formal derivation of \eqref{lengthpartitionasymptotic}. What makes such a proof possible is that the configuration $\mathfrak{S} (\lambda) = \{ \lambda_k - k \}_{k \ge 1}$ forms a \emph{determinantal point process}; we refer to the survey \cite{DPP} (or Section 3 of \cite{IP}) for a precise definition of and more information about these processes. Its correlation kernel is known (see Theorem 2 of \cite{IWRP}; equations (21) and (22) of \cite{CPLGR}; or Remark 2 in Chapter 5 of \cite{IP}) to be given by a two-fold contour integral. Under our specialization, this kernel becomes 
\begin{flalign}
\label{determinantalkernel1}
\begin{aligned}
K (i, j) & = \displaystyle\frac{1}{2 \pi \textbf{i}} \displaystyle\oint \displaystyle\oint \displaystyle\frac{F_X (w^{-1}) F_{\rho} (v) }{F_{\rho} (w) F_X (v^{-1})} \displaystyle\frac{w^j d v d w}{(v - w) v^{i + 1}} \\
& = \displaystyle\frac{1}{2 \pi \textbf{i}} \displaystyle\oint \displaystyle\oint \left( \displaystyle\frac{1 - s^{-1} \kappa^{-1} v^{-1}}{1 - s^{-1} \kappa^{-1} w^{-1}} \right)^N \left( \displaystyle\frac{1 - s w}{1 - s v} \right)^{x - 1} \displaystyle\frac{w^j}{v^{i + 1}} \displaystyle\frac{\big( - s \kappa \beta^{-1} v; q \big)_{\infty}^m}{\big( -s \kappa \beta^{-1} w; q \big)_{\infty}^m} \displaystyle\frac{dv dw}{v - w}, 
\end{aligned}
\end{flalign}

\noindent where $v$ and $w$ are integrated along positively oriented, simple, closed loops satisfying the following two properties. First, the contour for $w$ is contained inside the contour for $v$. Second, both contours enclose $0$ and $s^{-1} \kappa^{-1}$, but leave outside $s^{-1}$ and $- q^{-j} s^{-1} \kappa^{-1} \beta$ for each $j \in \mathbb{Z}_{\ge 0}$. 

Now, we would like to understand the asymptotic behavior of $- \ell (\lambda)$. Since $- \ell (\lambda)$ is the smallest integer not contained in $\mathfrak{S} (\lambda)$, Kerov's complementation principle for determinantal point processes (see Proposition A.8 of \cite{AMSG}) implies that $\mathbb{P} \big[ - \ell (\lambda) > M \big] = \det \big( \Id - \widetilde{K} \big)_{L^2 (\{ M, M + 1, \ldots \})}$, for any integer $M$, where $\widetilde{K} (i, j) = \textbf{1}_{i = j} - K (i, j)$. In particular, since $\textbf{1}_{i = j}$ is the residue of the right side of \eqref{determinantalkernel1} at $w = v$, we observe that $- \widetilde{K} (i, j)$ is equal to the right side of \eqref{determinantalkernel1}, in which the contours for $v$ and $w$ are switched (so that the contour for $v$ is now contained inside the contour for $w$).

Further changing variables $\widetilde{v} = q s \kappa v$ and $\widetilde{w} = q s \kappa w$ in the right side of \eqref{determinantalkernel1}, we deduce that 
\begin{flalign}
\label{determinantalkernel2} 
\begin{aligned}
\widetilde{K} (i + (m - 1) N & , j + (m - 1) N) \\
& = \displaystyle\frac{(q s \kappa)^{i - j}}{2 \pi \textbf{i}} \displaystyle\oint \displaystyle\oint \exp \Big( T \big( G(\widetilde{w}) - G(\widetilde{v}) \big) \Big) \displaystyle\frac{\big( q^{- 1} \beta^{-1} \widetilde{v}; q \big)_{\infty}^m}{\big( q^{-1} \beta^{-1} \widetilde{w}; q \big)_{\infty}^m} \displaystyle\frac{\widetilde{w}^j d\widetilde{v} d \widetilde{w}}{(\widetilde{v} - \widetilde{w}) \widetilde{v}^{i + 1}}, 
\end{aligned}
\end{flalign}

\noindent where $G (z) = G_V (z)$ was defined in \eqref{gv13}, and where where $\widetilde{w}$ and $\widetilde{v}$ are integrated along positively oriented, simple, closed loops satisfying the following two properties. First, the contour for $\widetilde{v}$ is contained inside the contour for $\widetilde{w}$. Second, both contours enclose $0$ and $-q $, but leave outside $-q \kappa $ and $q^{1 - j} \beta$ for each $j \in \mathbb{Z}_{\ge 0}$. In particular, we can use the contours $\mathcal{C}_V$ (for $\widetilde{w}$) and $\Gamma_V$ (for $\widetilde{v}$) from Definition \ref{vertexnearcontours}.  

Recall that $\Re \big( G (\widetilde{w}) - G (\widetilde{v}) \big) < 0$ for all $\widetilde{w}$ and $\widetilde{v}$ away from the critical point $\psi = \beta q$ (see \eqref{derivativegv}) of $G$. Therefore, the integrand defining the kernel $\widetilde{K}$ decays exponentially in $N$ when either $\widetilde{w}$ or $\widetilde{v}$ is away from $q \beta$. Hence, we can localize around $q \beta$ and change variables 
\begin{flalign*}
\sigma = \displaystyle\frac{1}{f T^{1 / 3}}; \qquad \widetilde{v} = q \beta (1 + \sigma \widehat{v}) ; \qquad \widetilde{w} = q \beta (1 + \sigma \widehat{w}).  
\end{flalign*}

\noindent Thus, 
\begin{flalign}
\label{determinantalkernel3} 
\begin{aligned}
\widetilde{K} & ((m - 1) N - \sigma^{-1} r, (m - 1) N - \sigma^{-1} r' ) \\
& = \displaystyle\frac{(s \kappa \beta^{-1})^{ \sigma^{-1} (r - r')} \sigma}{2 \pi \textbf{i}} \displaystyle\oint \displaystyle\oint \exp \Big( T \big( G(\widehat{w}) - G(\widehat{v}) \big) \Big) \displaystyle\frac{\big( 1 + \sigma \widehat{v}; q \big)_{\infty}^m}{\big( 1 + \sigma \widehat{w}; q \big)_{\infty}^m} \displaystyle\frac{(1 + \sigma \widehat{v})^{\sigma^{-1} r - 1} d\widehat{v} d \widehat{w}}{ (\widehat{v} - \widehat{w}) (1 + \sigma \widehat{w})^{\sigma^{-1} r'} }.
\end{aligned}
\end{flalign}

\noindent Now denote 
\begin{flalign*}
\overline{K} (r, r') = \sigma^{-1} (s \kappa \beta^{-1})^{\sigma^{-1} (r - r')} \widetilde{K} \big( (m - 1) N - \sigma^{-1} r, (m - 1) N - \sigma^{-1} r' \big).
\end{flalign*}

\noindent Then, $\det \big( \Id - \widetilde{K} \big)_{L^2 (\{ M, M + 1, \ldots \} )} = \det \big( \Id - \overline{K} \big)_{L^2 (\{ (m - 1) N - M, (m - 1) N - M + \sigma, \ldots \})}$.

 By \eqref{gzpsi}, we deduce that 
\begin{flalign*}
T \big( G_V (\widehat{w}) - G_V (\widehat{v}) \big) = & \displaystyle\frac{\widehat{w}^3}{3} - \displaystyle\frac{\widehat{v}^3}{3} + o (1); \qquad \displaystyle\frac{\big( 1 + \sigma \widehat{v}; q \big)_{\infty}^m}{\big( 1 + \sigma \widehat{w}; q \big)_{\infty}^m} = \left( \displaystyle\frac{\widehat{v}}{\widehat{w}} \right)^m + o (1); \\
& \displaystyle\frac{\big( 1 + \sigma \widehat{v} \big)^{\sigma^{-1} r - 1}}{\big( 1 + \sigma \widehat{w} \big)^{\sigma^{-1} r'}} = \exp (r \widehat{v} - r' \widehat{w}) + o (1).
\end{flalign*}  

\noindent Hence, it follows that 
\begin{flalign*}
\displaystyle\lim_{N \rightarrow \infty} \overline{K} (r, r') = \displaystyle\frac{1}{2 \pi \textbf{i}} \displaystyle\oint \displaystyle\oint \exp \left( \displaystyle\frac{\widehat{w}^3}{3} - \displaystyle\frac{\widehat{v}^3}{3} - r \widehat{v} + r' \widehat{w} \right) \left( \displaystyle\frac{\widehat{v}}{\widehat{w}} \right)^m \displaystyle\frac{d \widehat{w} d \widehat{v}}{\widehat{w} - \widehat{v}} = K_{\BBP; 0^m}, 
\end{flalign*}

\noindent from which we deduce that 
\begin{flalign*}
\displaystyle\lim_{N \rightarrow \infty} \mathbb{P} \left[ -\ell (\lambda) \ge (m - 1) N - s f N^{1 / 3} \right] &= \displaystyle\lim_{N \rightarrow \infty} \det \big( \Id - \widetilde{K} \big)_{L^2 (\mathbb{Z} \cap [(m - 1) N - s f N^{1 / 3}, \infty ))} \\
& = \displaystyle\lim_{N \rightarrow \infty} \det \big( \Id - \overline{K} \big)_{L^2 (\{ s, s + \sigma , \ldots \})} \\
& = \det \big( \Id - K_{\BBP; 0^m} \big)_{L^2 (s, \infty)} = F_{\BBP; 0^m},
\end{flalign*}

\noindent which (at least on a formal level) establishes \eqref{lengthpartitionasymptotic}.

\end{document}